\documentclass[english,twoside,reqno]{amsart}
\usepackage[square,comma]{natbib}
\usepackage[english]{babel} 
\usepackage{exscale}   
\usepackage{amsmath}
\usepackage{amsfonts}
\usepackage{amssymb}
\usepackage{amsxtra}
\usepackage{pstricks}
\usepackage{stmaryrd}
\usepackage{xspace}
\usepackage{epsfig}
\usepackage{mathrsfs}
\usepackage[left=3.0cm,top=3cm,right=3.0cm,bottom=3cm]{geometry}

\numberwithin{equation}{section}
\newtheorem{Theorem}{Theorem}[section]
\newtheorem{Lemma}[Theorem]{Lemma}
\newtheorem{Corollary}[Theorem]{Corollary}
\newtheorem{Proposition}[Theorem]{Proposition}
\newtheorem{Definition}[Theorem]{Definition}
\newtheorem{Remark}[Theorem]{Remark}

\newtheorem{Example}[Theorem]{Example}

\newcommand{\Bset}{\mathbb{B}}
\newcommand{\Cset}{\mathbb{C}}

\newcommand{\Nset}{\mathbb{N}}

\newcommand{\Sset}{\mathbb{S}}

\newcommand{\cA}{\ensuremath{{\mathcal A}}\xspace}         
\newcommand{\cB}{\ensuremath{{\mathcal B}}\xspace}         
\newcommand{\cC}{\ensuremath{{\mathcal C}}\xspace}         
\newcommand{\cH}{\ensuremath{{\mathcal H}}\xspace}         
\newcommand{\cM}{\ensuremath{{\mathcal M}}\xspace}         
\newcommand{\cN}{\ensuremath{{\mathcal N}}\xspace}         
\newcommand{\cU}{\ensuremath{{\mathcal U}}\xspace}         
\newcommand{\cZ}{\ensuremath{{\mathcal Z}}\xspace}         
\newcommand{\scrI}{\ensuremath{{\mathscr{I}}}\xspace} %
\newcommand{\scrtI}{\ensuremath{{\widetilde{\mathscr{I}}}}\xspace} %
\newcommand{\ii}{\mathbf{i}}

\newcommand{\bfM}{\ensuremath{\mathbf{M}}}     

\newcommand{\1}{\ensuremath{{\rm 1\kern-.25em l}}\xspace}  
 
\newcommand{\Ad}{\mathop{\mathrm{Ad}}}                     
\newcommand{\diag}{\ensuremath{\operatorname{diag}}}    
\newcommand{\id}{\operatorname{id}}                        

\newcommand{\sh}{\operatorname{sh}}           
\newcommand{\spec}{\operatorname{spec}}           
\newcommand{\trace}{\operatorname{tr}}        
\newcommand{\Trace}{\operatorname{Tr}}        
\newcommand{\tail}{\ensuremath{\mathrm{tail}}}               
\newcommand{\Aut}[1]{\ensuremath{\operatorname{Aut}(#1)}\xspace}  
\newcommand{\set}[2]{\mathopen{\{}#1\mathop{|}#2\mathclose{\}}}
\newcommand{\linh}{\ensuremath{\operatorname{span}}\xspace}  
\newcommand{\vN}{\operatorname{vN}} 
\newcommand{\lara}[1]{\ensuremath{\langle #1  \rangle}}

\newcommand{\distr}{\operatorname{distr}}   
\newcommand{\sotlim}{\ensuremath{\textsc{sot-}\lim}\xspace}   
\newcommand{\wotlim}{\ensuremath{\textsc{wot-}\lim}\xspace}   
\newcommand{\sot}{\ensuremath{\textsc{sot}}\xspace}           
\begin{document}
\title[Noncommutative independence from the symmetric group $\Sset_\infty$]{Noncommutative independence from \\characters of the infinite symmetric group $\Sset_\infty$}
\author[R. Gohm]{Rolf Gohm}
\author[C. K\"ostler]{Claus K\"ostler}
\address{Institute of Mathematics and Physics, Aberystwyth University, Aberystwyth, SY23 3BZ,UK}
\email{rog@aber.ac.uk}
\email{cck@aber.ac.uk} 
\date{\today}

\begin{abstract}
We provide an operator algebraic proof of a classical theorem
of Thoma which characterizes the extremal characters of the infinite symmetric group $\Sset_\infty$. Our methods are based on noncommutative conditional independence emerging from exchangeability \cite{GoKo09a,Koes10a} and we reinterpret Thoma's theorem as a noncommutative de Finetti type result.
Our approach is, in parts, inspired by Jones' subfactor theory and by Okounkov's spectral proof of Thoma's theorem \cite{Okou99a}, and we link them by inferring spectral properties from certain commuting squares. 
\end{abstract}

\maketitle

{%
\def\widedotfill{\leaders\hbox to 10pt{\hfil.\hfil}\hfill}
\def\pg#1{\widedotfill\rlap{\hbox to 15pt{\hfill{\small#1}}}\par}
\rightskip=15pt\leftskip=10pt%
\newcommand{\sct}[2]{\noindent\llap{\hbox to%
    10pt{{#1}\hfill}}~\mbox{#2~}}
\newcommand{\separ}{\vspace{-0.3em}}
\section*{Contents}
\sct{}{Introduction}%
\pg{\pageref{section:intro}}%
\sct{1}{Some basics of the infinite symmetric group $\Sset_\infty$}%
\pg{\pageref{section:symmetric}}%
\sct{2}{Exchangeability and conditional independence in noncommutative probability}%
\pg{\pageref{section:exchange-indep}}%
\sct{3}{Unitary representations of $\Sset_\infty$ and Thoma multiplicativity}%
\pg{\pageref{section:multiplicative}}%
\sct{4}{Fixed point algebras, limit cycles and generalized Thoma multiplicativity}%
\pg{\pageref{section:limit-cycles}}%
\sct{5}{Noncommutative Markov shifts and commuting squares from unitary representations of $\Sset_\infty$}%
\pg{\pageref{section:bernoulli-markov}}%
\sct{6}{Commuting squares over $\Cset$ with a normality condition}%
\pg{\pageref{section:commsquares}}%
\sct{7}{Thoma measures -- Okounkov's argument revisited}%
\pg{\pageref{section:okounkov}}%
\sct{8}{Completion of the proof of Thoma's theorem and the connection with Powers factors}%
\pg{\pageref{section:Thoma}}%
\sct{9}{Irreducible subfactors and Markov traces}%
\pg{\pageref{section:Markov}}%
\sct{}{References}%
\pg{\pageref{section:bibliography}}
}

\section*{Introduction}
\label{section:intro}
The infinite symmetric group $\Sset_\infty$ is the group of all finite permutations of the countable infinite set $\Nset_0 = \{0,1,2, \ldots\}$. Its left regular representation is a paradigm for the appearance of II$_1$-factors in the representation theory of large groups, as already known to Murray and von Neumann \cite{MuNe43a}. Even though an explicit motivation for their foundational work on operator algebras was to investigate the representation theory of large groups \cite{MuNe36a}, present classification results for representations of $\Sset_\infty$ still originate only little from an operator algebraic toolkit. The purpose of the present paper is to revive an operator algebraic treatment of such classification problems for the infinite symmetric group $\Sset_\infty$, by means coming from noncommutative probability. 

Being the prototype of a `wild' group \cite{KOV04a}, remarkable progress was made during the past decades in understanding the representation theory of $\Sset_\infty$, notably by work of Thoma, Vershik, Kerov, Olshanski and Okounkov among many others, and its broader context is of continuing interest for many researchers, see \cite{StVo75a,Hirai91a,Bian96a,Bian98a,HoOb07a,HiHiHo09a} and references therein for a still very incomplete list of different aspects. Here we will focus on an operator algebraic proof of the famous classical result obtained by Thoma \cite{Thom64a} which gives an explicit parametrization of all extremal characters of $\Sset_\infty$, i.e. 
the extremal points of the convex set of all positive definite functions on $\Sset_\infty$ which are constant on conjugacy classes and normalized at the group identity. 
\begin{Theorem}[Thoma]\label{thm:ethoma}
An extremal character of the group $\Sset_\infty$ is of the form 
\[
\chi(\sigma)= \prod_{k=2}^\infty \left(\sum_{i=1}^\infty a_i^k + (-1)^{k-1} \sum_{j=1}^\infty b_j^k\right)^{m_k(\sigma)}.
\]
Here $m_k(\sigma)$ is the number of $k$-cycles in the permutation $\sigma$
and the two sequences $(a_i)_{i=1}^\infty, (b_j)_{j=1}^\infty $ satisfy 
\vspace{-6pt}
\begin{eqnarray*}
a_1 \ge a_2 \ge \cdots \ge 0, \qquad b_1 \ge b_2 \ge \cdots \ge 0, \qquad \sum_{i=1}^\infty a_i + \sum_{j=1}^\infty b_j \le 1.
\end{eqnarray*}
\end{Theorem}
The original proof of Thoma was obtained in an indirect way by hard analysis and the later proofs are not easy too.  We mention the proof of Vershik and Kerov based on the asymptotics of Young diagrams \cite{VeKe81a,VeKe81b} and the more spectral theoretic proof of Okounkov \cite{Okou99a} using Olshansky semigroups \cite{Olsh89a,Olsh91a,Olsh91b,Olsh06a}. Surveys about these developments are given in \cite{Okou99a,KOV04a,Olsh06a}. None of these proofs of Thoma's theorem is mainly operator algebraic though the problem can be stated in this way: extremal characters correspond to finite factor representations in a von Neumann algebraic sense and now the corresponding traces, also called Thoma traces, need to be identified. 

Quite unrelated to these developments, the infinite symmetric group $\Sset_\infty$ also plays a crucial role in the classical subject of distributional symmetries and invariance principles in probability theory, see \cite{Kalle05a} for a recent account. Here exchangeability of an infinite sequence of random variables means that the joint distribution of this sequence is invariant under finite permutations of the random variables. 
In particular, exchangeability is equivalent to the existence of a certain
representation of the symmetric group $\Sset_\infty$. Now the celebrated de Finetti theorem states that an exchangeable infinite sequence is conditionally independent and identically distributed, where this conditioning is 
uniquely determined by the tail $\sigma$-algebra of the sequence. Even though stated only implicitly in its classical formulation, it tells us that an exchangeable infinite sequence is modeled on an infinite product of measurable spaces, equipped with a uniquely determined convex combination of product measures. Thus de Finetti's theorem can also be interpreted as a specific result in the representation theory of $\Sset_\infty$ within the subject of classical probability, where the role of extremal characters is now played by infinite product measures.  
     
Recently one of the authors has obtained in \cite{Koes10a} an operator algebraic version of de Finetti's theorem. It shows that exchangeability of noncommutative random variables implies noncommutative conditional independence which can be expressed equivalently in terms of commuting squares (as known in subfactor theory) and yields powerful factorization results. These results are refined and applied in \cite{GoKo09a}, where we have introduced `braidability' as a new kind of noncommutative probabilistic symmetry connected to the infinite braid group $\Bset_\infty$. This allowed us 
to establish a braided version of de Finetti's theorem which proved to be a strong new tool in the study of representations of  $\Bset_\infty$. Applying this machinery to the infinite symmetric group $\Sset_\infty$ we deal with exchangeability in an operator algebraic setting. By using the full force of noncommutative probability, as in fact appropriate for noncommutative groups, it turns out that there is also much to gain from this approach for the representation theory of $\Sset_\infty$. 

The completely new viewpoint of our approach is to interpret Thoma's theorem as an example of a noncommutative de Finetti theorem. This surprising connection becomes more evident from the von Neumann algebraic formulation of the following fact which allows us to apply our methods from \cite{GoKo09a,Koes10a}. Suppose $\chi$ is a character of $\Sset_\infty$ and let $\gamma_i \in \Sset_\infty$ denote the transposition $(0,i)$. Then it is elementary to verify that the sequence $(\gamma_i)_{i \in \Nset}$ is `exchangeable', i.e.
\[
\chi\big(\gamma_{\ii(1)} \gamma_{\ii(2)} \cdots \gamma_{\ii(n)} \big)
= \chi\big(\gamma_{\sigma(\ii(1))} \gamma_{\sigma (\ii(2))} \cdots \gamma_{\sigma(\ii(n))} \big)\qquad \qquad (\sigma \in \Sset_\infty)
\]    
for all $n$-tuples $\ii\colon \{1,\ldots, n\} \to \Nset$ and $n \in \Nset$. 
Now the task to identify the law of an exchangeable infinite sequence in the classical de Finetti theorem becomes the task to identify a character of the infinite symmetric group $\Sset_\infty$. Also the need to go beyond a purely algebraic treatment of the group $\Sset_\infty$ becomes apparent, since the role of the tail $\sigma$-algebra in de Finetti's theorem is played by the tail von Neumann algebra of the sequence $\big(\pi(\gamma_i)\big)_{i\in \Nset}$, where $\pi$ denotes the unitary representation associated to the character $\chi$. A second important feature of the so-called star generators $\gamma_i$ is that they neatly go along with disjoint cycle decompositions of permutations in $\Sset_\infty$, for example
\[
(3,5,1,10,7) (4,2) = (\gamma_3 \gamma_5 \gamma_1 \gamma_{10} \gamma_7 \gamma_3) \, (\gamma_4 \gamma_2 \gamma_4).
\]
Hence represented disjoint cycles become conditionally independent in the noncommutative sense. Altogether this provides us with very strong factorization properties for the mapping $\Sset_\infty \ni \sigma \to \chi(\sigma)$.

It turns out that our methods are exactly what is needed to give a fully operator algebraic proof of Thoma's theorem. We emphasize that the notion of noncommutative independence with universality properties such as tensor independence or freeness \cite{Spei97a,BeSc02a} is insufficient for this purpose. We need a more general notion of noncommutative independence as it emerges out of the noncommutative de Finetti theorem in terms of commuting squares of von Neumann algebras \cite{GoKo09a,Koes10a}. To complete the proof of Thoma's theorem the tail algebra appearing in the noncommutative de Finetti theorem has to be identified explicitly. Here spectral theory comes into play and some similarities with Okounkov's proof in \cite{Okou99a} become apparent.

It seems to us that the best way to illustrate the strength of these new ideas is to work out a full proof of Thoma's theorem which is self-contained given the de Finetti type results in \cite{Koes10a,GoKo09a}.
Along the way we prove a number of new operator algebraic results motivated from the probabilistic point of view. Let us discuss our plan in more detail.

\noindent

In Section \ref{section:symmetric} we recall some basic facts about the infinite symmetric group $\Sset_\infty$. Throughout we will work in two presentations of $\Sset_\infty$, the usual one given by the Coxeter generators $\sigma_i= (i-1,i)$ and, as already motivated above, a less familiar one given by the so-called star generators $\gamma_i = (0,i)$ (see Proposition \ref{prop:star-presentation}). The latter enjoy the additional property that $k$-cycles are elegantly expressed in terms of the $\gamma_i$'s
(see Lemma \ref{lem:cycle-1} and above for an example). Further we provide a discussion of certain shift endomorphisms on $\Sset_\infty$ and of some basics on the connection between characters and representations, as needed to put Thoma's theorem into context. We emphasize that our approach does not require any results about the representation theory of the symmetric groups $\Sset_n$ for $2 \le n < \infty$.

In Section \ref{section:exchange-indep} we concisely present our general noncommutative framework of exchangeability and independence which is the main tool for the rest of the paper. The emphasis is on tailoring the theory for the applications to come, here we mostly refer to \cite{Koes10a,GoKo09a} for proofs and further discussions. A few results are specific to the group $\Sset_\infty$, such as the possibility to restrict an automorphic representation to a minimal algebra (see Proposition \ref{prop:min-process}). Also we provide here a refined version of 
fixed-point characterizations from \cite{GoKo09a}, an operator algebraic generalization of the Hewitt-Savage zero-one law (see Theorem \ref{thm:fixed-point}). These characterizations will allow us to compute the (possibly non-abelian) tail algebra in the noncommutative de Finetti theorem, Theorem \ref{thm:de-finetti}, and to identify it with the fixed point algebra of a representation of $\Sset_\infty$ or with fixed point algebras of endomorphisms induced by certain symbolic shifts on the Coxeter generators $\sigma_i$ or star generators $\gamma_i$. Related results go partially beyond the corresponding results in \cite{GoKo09a}. 

In Section \ref{section:multiplicative} we really go to work and specify our general results on automorphic representations of $\Sset_\infty$ (from the  previous section) to a unitary representation $\pi$ of $\Sset_\infty$, see
Theorem \ref{thm:indy}. Moreover we show in Proposition \ref{prop:trace} that the existence of a tracial state on $\pi(\Sset_\infty)$ and the exchangeability of the represented star generators $v_i = \pi(\sigma_i)$ are intimately connected. By virtue of the noncommutative de Finetti theorem, the exchangeable sequence of unitaries $(v_i)_{i\in \Nset_0}$ enjoys powerful factorization properties with respect to conditional expectations onto their tail algebra $\cA_0$. Although this tail algebra $\cA_0$ turns out to be commutative, this is not a situation for the classical de Finetti theorem because the random variables $v_i$ do not commute with the tail algebra or with each other. For a more thorough discussion of the relation between the classical and the noncommutative de Finetti theorem we refer to \cite{Koes10a} and the more expository treatment in \cite{GoKo10bPP}.
  
A first important feature in the proof of Thoma's theorem is a fact called Thoma multiplicativity, i.e., a Thoma trace is multiplicative with respect to the disjoint cycle decomposition of an element of $\Sset_\infty$. This fact follows naturally from our factorization results and is actually a corollary of the noncommutative de Finetti theorem. More generally, as stated in Theorem \ref{thm:thoma-mult}, if we have a (not necessarily factorial) finite trace we obtain this multiplicativity for the conditional expectation onto the center of the von Neumann algebra generated by $\pi(\Sset_\infty)$.

In Section \ref{section:limit-cycles} we develop a systematic theory of certain weak limits of cycles which we call limit cycles (see Definition \ref{def:limit-cycle}). We apply these results to show how the fixed point algebras $\cA_n$, i.e., the centralizers of the represented Coxeter generators $\sigma_{n+2}, \sigma_{n+3}, \ldots$, are generated by cycles and limit cycles. This tower of fixed point algebras 
$ \cA_{-1} \subset \cA_0 \subset \cA_1 \subset \cA_2 \subset \cdots$ 
plays an important role in our approach and we can concretely identify each $\cA_n$ in Theorem \ref{thm:fixed-points}. A further generalization of Thoma multiplicativity is obtained in Theorem \ref{thm:thoma-mult-general} which provides such formulas for the conditional expectations onto the fixed point algebras $\cA_n$.

Let us comment at this point that our limit cycles are related to the `random cycles' appearing in Okounkov's proof of Thoma's theorem \cite{Okou99a} where he uses the latter to give a presentation of the Olshansky semigroups. We give a few hints into this direction in Remark \ref{rem:olshanski} but this discussion is by no means exhaustive and, in fact, we think that it may be very interesting to study more systematically the relation between noncommutative independence on the one hand and Olshansky semigroups on the other hand. We emphasize that we do not use the theory of Olshansky semigroups \cite{Olsh89a,Olsh91a,Olsh91b,Olsh06a}
in our proof, rather we replace it by our results about noncommutative independence. 

In Section \ref{section:bernoulli-markov} we give further evidence that the framework of noncommutative probability is well suited to analyze the operator algebraic structures generated by $\Sset_\infty$. Combining earlier results 
the fixed point algebras $\cA_n$ can be written explicitly in terms of the represented star generators $\pi(\gamma_i)$ and certain limit cycles, see Theorem \ref{thm:cs-star}. Here the endomorphism associated to the symbolic shift on the star generators $\gamma_i$ turns out to be a noncommutative Bernoulli shift in the sense of \cite{GoKo09a} and, similar to subfactor theory, a rich structure of triangular towers of commuting squares is obtained.
The fixed point algebras $\cA_n$ can of course also be expressed in terms of the Coxeter generators $\sigma_i$ instead of the star generators. This more delicate situation is the subject of Theorem \ref{thm:cs-coxeter}. Here the endomorphism associated to the symbolic shift on the Coxeter generators turns out to be a noncommutative Markov shift, which fits well into the theory of noncommutative stationary processes \cite{Gohm04a}, and we obtain also triangular towers of commuting squares. In particular we show that such a Markov shift is a Bernoulli shift if and only if the center-valued conditional expectation is a center-valued Markov trace.

In Section \ref{section:commsquares} we return to our proof of Thoma's theorem.
Clearly the spectral proof of Okounkov is closer to our method than the other known proofs and some arguments in Proposition \ref{prop:mu-discrete} are inspired by his ideas. But we give his estimates a more probabilistic turn
which allows us to generalize them to the von Neumann algebraic setting of 
certain commuting squares. Surprisingly this setting is already enough to deduce discreteness of spectra in Theorem \ref{thm:discrete-spec}. Applied to a certain limit cycle this is an important step towards Thoma's theorem. On the other hand our axiomatic setting indicates new and interesting possibilities to generalize such results, for example in subfactor theory. We give a short example for Hecke algebras but we do not follow this line of thought further in this paper.

To complete the argument that there are no other possibilities than those parametrized in Thoma's theorem it remains to be shown that our spectral measure is a Thoma measure, using the terminology of Okounkov \cite{Okou99a}. To do this we follow in Section \ref{section:okounkov} rather closely some of Okounkov's ideas and explain how noncommutative independence allows to transfer these arguments into a von Neumann algebraic setting. 

The proof of Thoma's theorem is finished by establishing the existence of finite factor traces for the parameters given in Thoma's theorem. The first such construction, from a groupoid point of view, has been given by Vershik and Kerov in \cite{VeKe81a} where they also mention the idea of embeddings into Powers factors.
In Section \ref{section:Thoma} we give a brief elaboration of the latter idea in the general case, not only to make our proof self-contained but also because in this infinite tensor product setting many of our constructions can be visualized in a very satisfying way. The factoriality of the trace can be inferred from one of the fixed point characterizations given in Theorem \ref{thm:fixed-point}. We are not aware of another proof of this fact which is similarly direct. 

In the final Section \ref{section:Markov} we briefly discuss the case that the subfactor generated by the Coxeter generators $\sigma_2, \sigma_3, \ldots$ (i.e., omitting the first one, $\sigma_1$) is irreducible. We characterize it in terms of the parameters in Thoma's theorem and we note that these traces are exactly the Markov traces on the group algebra of $\Sset_\infty$ which extend to an operator algebraic setting. This makes contact to work of Ocneanu, Wenzl and Jones, see \cite{Wenz88a,Jone91a,JoSu97a},
who studied these Markov traces motivated by Jones' subfactor theory. 
Independently a very recent preprint of Yamashita \cite{Yama09a} also focuses on related questions.

To summarize, we show in this paper that there is an approach via exchangeability and noncommutative independence which leads to an operator algebraic proof of Thoma's theorem and to a very transparent and detailed structure theory for finite (factor) representations of $\Sset_\infty$. Beyond that we believe that a further investigation of the interplay with more traditional approaches will be of great interest for the development of von Neumann algebraic methods in representation theory. 
\section{Some basics of the infinite symmetric group $\Sset_\infty$}
\label{section:symmetric}
\subsection*{Presentations of symmetric groups}
The symmetric group $\Sset_n$ is presented by the Coxeter generators $\sigma_1, \sigma_2,\ldots, \sigma_{n-1}$ subject to the relations 
\begin{align*}
&\sigma_i \sigma_{j} \sigma_i = \sigma_{j} \sigma_i \sigma_{j} &&\text{if $ \; \mid i-j \mid\, = 1 $}, & \tag{B1} \label{eq:B1}\\
&\sigma_i \sigma_j = \sigma_j \sigma_i  &&\text{if $ \; \mid i-j \mid\, > 1 $},& \tag{B2} \label{eq:B2}\\
&\sigma_i^2  = \sigma_0 &&\text{if $i \in \Nset$ },\tag{S} \label{eq:sym}
\end{align*}
where $\sigma_0$ denotes the identity of $\Sset_n$ and $\Sset_1 = \langle \sigma_0 \rangle$. We realize $\Sset_n$ as permutations of the set $\{0, 1, 2, \ldots, n-1\}$ such that $\sigma_i$ is given by the transposition $(i-1,i)$. By convention the product  of two permutations $\sigma, \tau$ acts as $\sigma \tau (k) = \sigma(\tau(k))$. Throughout $\Sset_n$ is identified with the subgroup 
$\set{\sigma \in \Sset_{n+1}}{\sigma(n) = n}$ of $\Sset_{n+1}$ and $\Sset_\infty$ denotes the inductive limit of the groups $\Sset_n$ with respect to these embeddings. So $\Sset_\infty$ is the group of all finite permutations of the set $\Nset_0$, called the \emph{infinite symmetric group}.  Further we will make use of the subgroups $\Sset_{n, \infty} := \langle \sigma_n, \sigma_{n+1}, \ldots \rangle$ for $n \in \Nset$. Note that $\Sset_{1, \infty}= \Sset_\infty$.  

We recall that the relations $\eqref{eq:B1}$ and $\eqref{eq:B2}$ are the 
defining relations for the Artin generators of braid groups $\Bset_n$ and their inductive limit $\Bset_\infty$. Now let $\widehat{} \,\,\colon \Bset_\infty \to \Sset_\infty$ be the canonical epimorphism which maps Artin generators of $\Bset_\infty$ onto Coxeter generators of $\Sset_\infty$. 
This epimorphism allows us to turn the results of \cite{GoKo09a} on braid groups directly into results on symmetric groups. Throughout we will heavily make use of this. 

The investigations in \cite{GoKo09a} also reveal a new presentation of the braid group $\Bset_\infty$. This presentation plays therein a special role for braidability, a noncommutative extension of exchangeability. Adding a relation on the idempotence of the generators to this presentation, we have at hands its analogue for symmetric groups. This presentation will play a similar role for exchangeability as the square root of free generators presentation does for braidability. 
\begin{Proposition}\label{prop:star-presentation}
The symmetric group $\Sset_n$ (for $n \ge 3$) is presented by the generators 
$\set{\gamma_i}{1 \le i \le n-1}$ subject to the defining relations
\begin{align}\tag{EB} \label{eq:EB} 
       \gamma_{l} \gamma_{l-1} (\gamma_{l-2} \gamma_{l-3} \cdots \gamma_{k+1}\gamma_{k}) \gamma_l 
     = \gamma_{l-1}  (\gamma_{l-2} \gamma_{l-3}\cdots \gamma_{k+1}\gamma_{k}) \gamma_{l}\gamma_{l-1} &&(0 < k < l <n) \\
\gamma_{k}^2 = \gamma_0 && (0 < k < n). \notag
\end{align}
Here $\gamma_0$ denotes the identity of $\Sset_n$. Moreover, the generator $\gamma_k$ is realized for $1 \le k < n$ as the transposition
\[
\gamma_k = (0,k)
\]
on the set $\{0,1,\ldots, n-1\}$.  
\end{Proposition}
We will see below that the $\gamma_k$'s are convenient to write down cycles in symmetric groups. Note also the obvious equalities $\gamma_1 = \sigma_1$ and $\gamma_0=\sigma_0$.  
\begin{proof} The relations \eqref{eq:EB} define the square root of free generator presentation of $\Bset_n$ (see \cite{GoKo09a}). Adding the relations $\gamma_k^2=\gamma_0$ one obtains the group $\Sset_n$ as a quotient of $\Bset_n$. As shown in \cite{GoKo09a}, the generators $\gamma_i$ and $\sigma_j$ are related by 
\begin{equation}\label{eq:gamma2sigma}
\gamma_k= \sigma_1 \cdots \sigma_{k-1} \sigma_k \sigma_{k-1} \cdots \sigma_1.
\end{equation}
Now $\gamma_k$ is immediately identified as the transposition $(0,k)$. 
\end{proof} 
\begin{Definition}\normalfont
This presentation of $\Sset_n$ from Proposition \ref{prop:star-presentation} is referred to as the \emph{star presentation} and the generators $\gamma_i$ are called \emph{star generators} for $i >0$.   
\end{Definition}
\begin{Remark}\normalfont\label{rem:star-pres}
The star presentation is induced by a star-shaped graph, similar as for braid groups in \cite{Serg93a}. To be more precise, a presentation of the symmetric group is obtained from the planar graph presentations for braid groups in \cite{Serg93a} by adding the relations on the idempotence of each generator. Now a linear graph with $n$ vertices and $n-1$ edges yields the Coxeter presentation of $\Sset_n$ with generators $\sigma_i$. As its name already suggests, the star presentation of $\Sset_n$ is associated to a star-shaped graph, again with $n$ vertices and $n-1$ edges. Here the point $0$ is distinguished  as the central vertex of the star shaped graph. 
Actually the star presentation enjoys much more symmetry than the relations \eqref{eq:EB} indicate, as it becomes clear from Lemma \ref{lem:cycle-1} and Lemma \ref{lem:cycles-4} below. Further information on the combinatorics of star generators is contained in \cite{Pak99a,IrRa09a}.
\end{Remark}
\subsection*{Cycles}
Each permutation $\sigma \in \Sset_\infty$ can be decomposed into disjoint cycles $s_1, s_2, \ldots, s_n \in \Sset_\infty$ for some $n \in \Nset$ such that
\[
\sigma = s_1 s_2 \cdots s_n.
\]
Note that $s_i$ and $s_j$ commute for $1 \le i, j \le n$. A cycle $(n_1, n_2, \ldots, n_k)$ has the length $k$ and is referred to as a $k$-cycle. 
We will denote by $m_k(\sigma)$ the number of $k$-cycles in $\sigma$. Note that
$m_1(\sigma) = \infty$ and $ m_k(\sigma) \neq 0$ for at most finitely many $k >1$.   

We record next some elementary results which will be crucial in the proofs of our main results.
\begin{Lemma} \label{lem:cycle-1}
Let $k \in \Nset$. A $k$-cycle $\sigma = (n_1, n_2, n_3, \ldots, n_k) \in \Sset_\infty$ is of the form 
\[
\sigma = (\gamma_{n_1} \gamma_{n_2} \gamma_{n_3} \cdots \gamma_{n_{k-1}}\gamma_{n_k}) \gamma_{n_1}, 
\]
provided that $n_1=0$ if $\sigma(0) \neq 0$.
\end{Lemma}
For notational convenience we will suppress the parentheses in $(\gamma_{n_1} \gamma_{n_2} \gamma_{n_3} \cdots \gamma_{n_{k-1}} \gamma_{n_k})  \gamma_{n_1}$. Algebraically, any $1$-cycle satisfies $\gamma_{n_1}\gamma_{n_1} = \gamma_{0}$ and will be omitted in cycle decompositions of permutations.   
\begin{proof}
If $\sigma(0)=0$ , this is immediate from the definition of a $k$-cycle and $\gamma_n = (0,n)$ for all $n \in \Nset$. If $n_1 = 0$ then
\begin{eqnarray*} 
\gamma_{n_1} \gamma_{n_2} \gamma_{n_3} \cdots \gamma_{n_{k-1}}\gamma_{n_k} \gamma_{n_1}
&=& \gamma_{n_2} \gamma_{n_3} \cdots \gamma_{n_{k-1}}\gamma_{n_k}\\
&=& (0, n_2, n_3, \ldots, n_k ),
\end{eqnarray*}  
since $\gamma_0$ is the identity. See \cite[Lemma 3]{IrRa09a} for a more detailed proof. 
\end{proof}
The slight difference in the treatment of cycles containing the point $0$ is attributed to the prominent role of this point in the star presentation of $\Sset_\infty$ (see Remark \ref{rem:star-pres}). If one of the $n_i$'s is zero, then $\gamma_{n_1}\gamma_{n_2} \cdots \gamma_{n_k} \gamma_{n_1}$ is a $k$-cycle if and only if $n_1 =0$.  This condition $n_1=0$ can always be achieved by cyclic permutations of the  $n_i$'s and their re-labeling. Throughout we will assume that this re-labeling is done if necessary.  

The following elementary result will be at the heart of our application of noncommutative independence (compare Theorem \ref{thm:indy} for example). Denote by $\langle \gamma_i \mid i \in I \rangle$ the subgroup of $\Sset_\infty$ generated by the $\gamma_i$'s with $i \in I \subset \Nset$. 
\begin{Lemma}\label{lem:cycle-2}
Suppose the permutation $\sigma\in \Sset_\infty$ has the cycle decomposition
$
\sigma = s_1 s_2 \cdots s_n,
$
where $s_1,s_2,\ldots, s_n$ are disjoint non-trivial cycles. Then there exist
mutually disjoint subsets $I_1, I_2, \ldots, I_n \subset \Nset$ such that
$s_i \in \langle \gamma_i \mid i \in I_i \rangle$ for $i=1, \ldots n$.  
\end{Lemma}
\begin{proof}
This is evident from the cycle decomposition of permutations and Lemma \ref{lem:cycle-1}. 
\end{proof}
Finally let us record how cycles are effected by conjugation with elements from $\Sset_\infty$. This illustrates that the defining relations $\eqref{eq:EB}$ induce more general cycle equations.    
\begin{Lemma}\label{lem:cycles-4}
Let $\sigma \in \Sset_\infty$. 
\begin{enumerate}
\item \label{item:cycles-i}
Consider the 2-cycle $\gamma_n = \gamma_0 \gamma_n \gamma_0$, with $n \in \Nset$. Then 
\[
\sigma \gamma_n \sigma^{-1}
= \begin{cases}
 \gamma_{\sigma(0)} \gamma_{\sigma(n)} \gamma_{\sigma(0)}
 & \text{if $\sigma(n) \neq 0$},\\
  \gamma_{\sigma(n)} \gamma_{\sigma(0)} \gamma_{\sigma(n)}
 & \text{if $\sigma(0) \neq 0$}.
 \end{cases}
\]
\item \label{item:cycles-ii}
Let $\gamma_{n_1} \gamma_{n_2} \cdots \gamma_{n_k} \gamma_{n_1}$ be a $k$-cycle with $k\ge 2$. Then
\[
\sigma \gamma_{n_1} \gamma_{n_2} \cdots \gamma_{n_k} \gamma_{n_1} \sigma^{-1}
= \begin{cases}
\gamma_{\sigma(n_1)}^{}
  \gamma_{\sigma(n_2)}^{}
   \cdots 
  \gamma_{\sigma(n_k)}^{}
  \gamma_{\sigma(n_1)}^{} &  
  \text{if $0 \notin \set{\sigma(n_i)}{i=1,\ldots,k}$}, \\  
\gamma_{\sigma(n_i)}^{}
  \gamma_{\sigma(n_{i+1})}^{}
   \cdots 
  \gamma_{\sigma(n_{k})}^{}
  \gamma_{\sigma(n_1)}^{} \cdots \gamma_{\sigma(n_{i-1})}^{} \gamma_{\sigma(n_i)}^{} &
  \text{if $\sigma(n_i)=0$}. 
\end{cases}
\] 
\end{enumerate}
\end{Lemma}
\begin{proof}
(\ref{item:cycles-i})
The 2-cycle $(0,n)$ is mapped to the 2-cycle $(\sigma(0), \sigma(n))$. Now use
Lemma \ref{lem:cycle-1}. The equations in (\ref{item:cycles-ii}) are verified similarly.
\end{proof}
\subsection*{Shifts on the infinite symmetric group}
We introduce several closely related shifts on $\Sset_\infty$ which will be useful for the investigation of fixed points and tail algebras in Sections \ref{section:multiplicative} and \ref{section:limit-cycles}. 
\begin{Definition}\normalfont \label{def:shifts} 
The \emph{shift} $\sh$ is given by the endomorphism on $\Sset_\infty$ 
defined by
\[
\sh(\sigma_n) = \sigma_{n+1}
\] 
for all $n \in \Nset$. More generally, for $m \in \Nset_0$ fixed, the \emph{$m$-shift} (or \emph{partial shift)} $\sh_m$ on $\Sset_\infty$  is given by the endomorphism 
\[
\sh_m(\tau) := \sigma_m \sigma_{m-1} \cdots 
\sigma_1 \sigma_0 \sh(\tau) \sigma_0 \sigma_1 \cdots \sigma_{m-1}\sigma_m.
\]
\end{Definition}
Clearly we have $\sh_0 = \sh$. Note also that $\sh_m$ is obtained from $\sh$ by conjugation with Coxeter generators, since $\sigma_i = \sigma_{i}^{-1}$. 
\begin{Lemma}\label{lem:shift-1}
The endomorphisms $\sh_m$ on $\Sset_\infty$ are injective for all 
$m \in \Nset_0$. Moreover there exists, for every $\tau \in \Sset_\infty$, 
some $n \in \Nset$ such that 
\begin{eqnarray}\label{eq:shift-1}
\sh_m(\tau) = (\sigma_{m+1}^{}\sigma_{m+2}^{} \cdots 
\sigma_{n-1}^{}\sigma_n^{})\, \tau \,(\sigma_n^{} \sigma_{n-1}^{} 
\cdots \sigma_{m+2}^{} \sigma_{m+1}^{}).
\end{eqnarray}
\end{Lemma}
\begin{proof}
Note that a non-trivial permutation has a non-trivial cycle decomposition.
Since the endomorphism $\sh$ preserves the length of a $k$-cycle, the injectivity of $\sh$ is immediate. The endomorphism $\sh_m$ is the composition of an injective endomorphism and automorphisms, and thus 
injective. We observe that, using the braid relations \eqref{eq:B1} and  
\eqref{eq:B2} as well as \eqref{eq:sym},
\begin{eqnarray*}
\sh(\sigma_{i}^{})
= \sigma_{i+1}^{}  
= \big(\sigma_1^{} \sigma_2^{} 
\cdots \sigma_{i-1}^{}\sigma_i^{} \sigma_{i+1}^{}\big) 
\sigma_i^{} \big(\sigma_{i+1}^{} \sigma_i^{} \sigma_{i-1}^{} \cdots \sigma_{2}^{} 
\sigma_1^{}\big).   
\end{eqnarray*}
Now let $\tau \in \Sset_\infty$ be given. Then there exists some 
$n \in \Nset$ such that $\tau \in \Sset_n$ and 
\[
\sh(\tau) = (\sigma_1^{} \sigma_2^{} 
 \cdots \sigma_{n-1}^{}\sigma_n^{}) \tau (\sigma_n^{} \sigma_{n-1}^{} 
 \cdots \sigma_2^{} \sigma_1^{}).
\] 
This proves \eqref{eq:shift-1}, since $\sh_m(\tau) = 
\sigma_m^{} \cdots \sigma_1^{} \sh(\tau) \sigma_1^{} \cdots \sigma_m^{}$. 
\end{proof}
By its very definition, $\sh$ is the symbolic shift on the Coxeter generators $\sigma_i$. But as we will see next for $m \ge 1$, the shifts $\sh_m$ are certain symbolic shifts on the star generators $\gamma_i$. 
\begin{Lemma}\label{lem:shift-2} 
Let $m \in \Nset$. Then it holds
\begin{eqnarray*}
\sh_m(\gamma_i) =
\begin{cases}
\gamma_i & \text{ if $i < m$;}\\
\gamma_{i+1}  & \text{ if $i \ge m$.}
\end{cases}
\end{eqnarray*}
In particular, for all $n \in \Nset$, 
\[
\sh_1(\gamma_n)= \gamma_{n+1}.
\]
\end{Lemma}
\begin{proof}
We conclude from \eqref{eq:shift-1} and \eqref{eq:B2} that the action of $\sh_m$ on the $\sigma_i$'s is 
\begin{eqnarray*}
\sh_m(\sigma_i) =
\begin{cases}
\sigma_i & \text{ if $i < m$;}\\
\sigma_i \sigma_{i+1} \sigma_i  & \text{ if $i = m$;}\\
\sigma_{i+1} &  \text{ if $i > m$.}
\end{cases}
\end{eqnarray*} 
Use these equations to identify the action of $\sh_m$ on $\gamma_i = \sigma_1 \sigma_2 \cdots \sigma_{i-1}\sigma_i \sigma_{i-1} \cdots \sigma_2 \sigma_1$ as claimed.     
\end{proof}
\subsection*{Characters of groups and associated noncommutative probability spaces} 
\label{subsection:characters}
We will make heavy use of the close relationship between characters on symmetric groups and traces on von Neumann algebras generated by unitary representations of symmetric groups. Since this well known relationship does not rely on the specifics of symmetric groups we provide it in greater generality than needed for the purposes of this paper.

Let $G$ be a group. A positive definite function $\chi\colon G \to \Cset$ is called a \emph{character} if $\chi$ is constant on conjugacy classes of $G$ and  normalized at the identity $e$ of $G$. 
Given the pair $(G, \chi)$, the Kolmogorov decomposition (see \cite{EvLe77a} for example) of the positive definite kernel
\[
G \times G \ni (g,h) \mapsto \chi(g^{-1}h)
\]
provides us with a Hilbert space $\cH$, a unitary representation $\pi$ of $G$ and a vector $\xi \in \cH$ such that 
\[
\chi(g^{-1}h) = \langle \pi(g) \xi, \pi(h)\xi \rangle.   
\]
The Kolmogorov decomposition is said to be \emph{minimal} if $\cH = \overline{\linh\set{\pi(g)\xi}{g \in G}}$. Note that a Kolmogorov decomposition can always be turned into a minimal one.   
  
\begin{Lemma}\label{lem:kol}
Suppose $\chi$ is a character of the countable group $G$ and let $(\cH, \pi, \xi)$ be a minimal Kolmogorov decomposition of the pair $(G, \chi)$.
Then the von Neumann algebra $\cA := \vN(\pi(g) \mid g\in G)$ has separable predual and $\trace := \langle \xi, \bullet \,\xi \rangle$ is a tracial faithful normal state on $\cA$.
\end{Lemma}
\begin{Definition}\normalfont
The pair $(\cA,\trace)$ as constructed in Lemma \ref{lem:kol} is called the \emph{(noncommutative) probability space associated to $(G, \chi)$}.
\end{Definition}
\begin{proof}
Since $G$ is countable, $\cA$ has a separable predual. 
Clearly, $\trace$ is a unital normal state on $\cA$. The traciality of this state is immediate from $\chi(ghg^{-1}) = \chi(h)$ for all $g, h \in G$.   To prove the faithfulness of $\trace$, we use that 
\[
\trace(x^*a^* a x) = \trace(axx^*a^*) \le \|x\|^2 \trace(aa^*) =  \|x\|^2 \trace(a^*a)
\]
for all $a,x \in \cA$. Thus $\trace(a^*a)=0$ implies $a x\xi = 0$ for all $x \in \cA$, and further $a=0$ since $\xi$ is a cyclic vector for $\cA$. 
\end{proof}
Finally we will make use of the following well known fact.
\begin{Proposition}\label{prop:extremal-char}
Let $(\cA,\trace)$ be the probability space associated to $(G, \chi)$. Then $\cA$ is a factor if and only if $\chi$ is an extremal character.
\end{Proposition}
\begin{proof}
The von Neumann algebra $\cA$ is a factor if and only if  $\trace$ cannot be written as a non-trivial convex combination of tracial states. Indeed, if $\cA$ is a factor then there is only one tracial state, see \cite[Theorem 8.2.8]{KaRi2}. Conversely, if $\cA$ is not a factor, then choose a non-trivial central projection $z$ and write $\trace$ as the convex combination of 
the two tracial states $\trace(z \,\bullet\,)/\trace(z)$ and $\trace((\1-z)\, \bullet\,)/\trace(\1-z)$. The restriction of a trace to $\pi(G)$ is a character and hence the asserted equivalence follows. 
\end{proof}

\begin{Remark}\normalfont
The probability space $(\cA,\trace)$ associated to $(G,\chi)$ can also be obtained from the *-probability space $(\Cset G, \chi_\Cset^{})$ via the GNS construction. Here $\chi_\Cset^{}$ is understood to be the complex linear extension of the character $\chi$ to the group *-algebra $\Cset G$. 
This *-probability space is regular in the sense of \cite{Koes10b} and thus the GNS construction provides us with a cyclic and separating vector for the representation.
So the probability space associated to $(G, \chi)$ can be identified with the one obtained by an application of the GNS construction to the pair $(\Cset G, \chi_\Cset^{})$.
\end{Remark}
\section{Exchangeability and conditional independence in noncommutative probability}
\label{section:exchange-indep}
This section is devoted to the basic setting of noncommutative probability in the context of exchangeability. We will make use of results from \cite{Koes10a}, its application to braidability in \cite{GoKo09a} and the Appendix of \cite{GoKo09a} on operator algebraic noncommutative probability theory. We emphasize that all general results from \cite{GoKo09a} apply, since exchangeability implies braidability. 

A \emph{(noncommutative) probability space} $(\cA,\varphi)$ consists of a von Neumann algebra $\cA$ with separable predual and a faithful normal state $\varphi$ on $\cA$.  
A von Neumann subalgebra $\cB$ of $\cA$ is said to be \emph{$\varphi$-conditioned} if the (unique) $\varphi$-preserving conditional expectation $E_{\cB}$ from $\cA$ onto $\cB$ exists. We will mainly work with \emph{tracial probability spaces}, where the state $\varphi$ is a trace. Note that the existence of the $\varphi$-preserving conditional expectation $E_\cB$ is automatic in the tracial setting. For further information on the more general non-tracial setting we refer the reader to \cite[Appendix A]{GoKo09a}.

A \emph{random variable} $\iota$ from $(\cC_0,\varphi_0)$ to $(\cA,\varphi)$ is an injective *-homomorphism  $\iota\colon \cC_0 \to \cA$ such that $\varphi_0 = \varphi \circ \iota$  and $\iota(\cC_0)$ is $\varphi$-conditioned.  A \emph{random sequence} $\scrI$ is an infinite sequence of (identically distributed) random variables $\iota\equiv (\iota_n)_{n\in \Nset_0}$ from $(\cC_0, \varphi_0)$ to $(\cA,\varphi)$. We may assume $\cC_0 = \iota_0(\cC_0) \subset \cA$ and $\varphi_0 = \varphi|_{\cC_0}$ whenever it is convenient. 

Let us come next to distributional symmetries of random sequences. Given the two random sequences $\scrI$ and $\scrtI$ with random variables 
$(\iota_n)_{n\ge 0}$ resp.~$(\tilde{\iota}_n)_{n\ge 0}$ 
from $(\cC_0,\varphi_0)$ to $(\cA,\varphi)$, we write
\[
(\iota_0, \iota_1, \iota_2, \ldots ) \stackrel{\distr}{=} 
(\tilde{\iota}_0, \tilde{\iota}_1, \tilde{\iota}_2, \ldots )
\]
if all their multilinear functionals coincide:
\[
\varphi\big(\iota_{\ii(1)}(a_1) \iota_{\ii(2)}(a_2) 
           \cdots \iota_{\ii(n)}(a_n)\big) 
= \varphi\big(\tilde{\iota}_{\ii(1)}(a_1) \tilde{\iota}_{\ii(2)}(a_2) 
              \cdots  \tilde{\iota}_{\ii(n)}(a_n)\big) 
\]
for all $n$-tuples $\ii\colon \{1, 2, \ldots, n\} \to \Nset_0$, 
$(a_1, \ldots, a_n) \in \cC_0^n$ and $n \in \Nset$. 
\begin{Definition}\normalfont
A random  sequence $\scrI$ is 
\begin{enumerate}
\item \label{item:ds-i}
\emph{exchangeable} if its 
multilinear functionals are invariant under permutations: 
\[
 (\iota_0,\iota_1, \iota_2,  \ldots ) \stackrel{\distr}{=} 
 (\iota_{\pi(0)},\iota_{\pi(1)}, \iota_{\pi(2)},  \ldots ) 
\]
for any finite permutation $\pi \in \Sset_\infty$ of $\Nset_0$;  
\item \label{item:ds-ii}
\emph{spreadable} if its 
multilinear functionals are invariant under the passage to subsequences: 
\[
 (\iota_0,\iota_1, \iota_2,  \ldots ) \stackrel{\distr}{=} 
 (\iota_{n_0},\iota_{n_1}, \iota_{n_2},  \ldots ) 
\]
for any (strictly increasing) subsequence $(n_0, n_1, n_2, \ldots )$ of 
$(0,1,2,\ldots)$;  
\item \label{item:ds-iii}
\emph{stationary} if the multilinear 
functionals are shift-invariant:
\[
 (\iota_0,\iota_1, \iota_2,  \ldots )  \stackrel{\distr}{=} 
 (\iota_{k},\iota_{k+1}, \iota_{k+2},\ldots ) 
\]  
for all $k\in \Nset$. 
\end{enumerate}
\end{Definition}
The following hierarchy of distributional symmetries is clear: \eqref{item:ds-i} $\Rightarrow$  \eqref{item:ds-ii} $\Rightarrow$  \eqref{item:ds-iii}.
\begin{Definition}\normalfont \label{def:convention}
\begin{enumerate}
\item\label{item:convention-i}
Fixing some $a \in \cC_0$, a random sequence $\scrI$ produces a sequence of operators $(x_n)_{n\ge 0}$ with $x_n= \iota_n(a)$, called a \emph{sequence of operators induced by $\scrI$}. 
A sequence of operators 
$(x_n)_{n\ge 0}$ is said to have property `A' if it is induced by some random sequence $\scrI$ with property `A'. For example, $(x_n)_{n\ge 0}$ is stationary if $\scrI$ is so. 
\item\label{item:convention-ii}
More generally, fixing a subalgebra $\cB$ of $\cC_0$, a random sequence $\scrI$ produces a sequence of subalgebras $(\cB_n)_{n \ge 0}$ with $\cB_n = \iota_n(\cB)$, called a \emph{sequence of subalgebras induced by $\scrI$}. 
A sequence of subalgebras $(\cB_n)_{n \ge 0}$ is said to have property `A' if it is induced by some random sequence $\scrI$ with property `A'.
\end{enumerate}
\end{Definition}
We come to our concept of (noncommutative) conditional independence and factorizability. For a more detailed treatment see \cite[Section 3]{Koes10a}, for example.   
\begin{Definition}\normalfont\label{def:independence}
Given the probability space $(\cA,\varphi)$, let $\cN$ and $(\cC_i)_{i \in I}$ be $\varphi$-conditioned von Neumann subalgebras of $\cA$, where $(I,>)$ is assumed to be an ordered set.   
\begin{enumerate}
\item
The family $(\cC_i)_{i\in I}$ is \emph{full} (resp.~\emph{order}) \emph{$\cN$-independent} if 
\[
E_\cN(xy) = E_\cN(x) E_\cN(y)
\]
for all $x \in \vN(\cN, \cC_j \mid j\in J)$ and $y \in \vN(\cN, \cC_k\mid k\in K)$ whenever $J,K \subset I$ with $J \cap K = \emptyset$ (resp.~$J < K$ or $J > K$). 
\item
The family $(\cC_i)_{i\in I}$ is \emph{full} (resp.~ \emph{order}) \emph{$\cN$-factorizable} if  
\[
E_\cN(xy) = E_\cN(x) E_\cN(y)
\]
for all $x \in \vN(\cC_j \mid j\in J)$ and $y \in \vN(\cC_k\mid k\in K)$ whenever $J,K \subset I$ with $J\cap K = \emptyset$ (resp.~$J < K$ or  $J > K $).
\item 
A family of operators $(x_i)_{i\in I}$ in $\cA$ is \emph{full}/\emph{order} $\cN$-\emph{independent}/\emph{factorizable} if the family $(\vN(x_i))_{i \in I}$ is so. 
\item
A family of random variables $(\iota_i)_{i \in I}\colon (\cC_0,\varphi_0) \to (\cA,\varphi)$ is \emph{full}/\emph{order} $\cN$-\emph{independent}/\emph{factor\-izable} if the family of ranges $\big(\iota_i(\cC_0)\big)_{i \in I}$ is so. 
\end{enumerate}
\end{Definition}
Note that full $\cN$-independence implies order $\cN$-factorizability. The latter one is more easily verified in applications. It is an ingredient in the fixed point characterization of Theorem \ref{thm:fixed-point} and will be applied in Proposition \ref{prop:factor}. 

We will mainly be interested in the nonnegative integers $\Nset_0= \{0,1,2,\ldots\}$ (equipped with the natural order $>$) as index set $I$. If $I$ is the two point set $\{1,2\}$, then it is convenient to address conditional independence in terms of commuting squares. 
\begin{Lemma}\label{lem:commuting-square}
Let $\cN,\cC_1,\cC_2$ be $\varphi$-conditioned von Neumann subalgebras of $\cA$ and put $\cB_i := \vN(\cN,\cC_i)$. Then the following are equivalent:
\begin{enumerate}
\item  \label{item:cs-i}
$(\cC_i)_{i\in \{1,2\}}$ are full $\cN$-independent;
\item  \label{item:cs-ii} 
$E_{\cB_1}(\cB_2)) = \cN$;  
\item \label{item:cs-iii} 
$E_{\cB_1} E_{\cB_2} = E_{\cN}$;
\item  \label{item:cs-iv} 
$E_{\cB_1}E_{\cB_2} = E_{\cB_2}E_{\cB_1}$ and $\cN = \cB_1 \cap \cB_2$; 
\item \label{item:cs-v} 
$E_{\cN}(x y z) = E_{\cN}(x E_{\cN}(y) z)$ for all $x,z\in \cB_1$ and $y \in \cB_2$.
\end{enumerate}
\end{Lemma}
\begin{proof}
The proof of \cite[Proposition 4.2.1]{GHJ89a} transfers to the non-tracial setting after some obvious modifications.  
\end{proof}   
\begin{Definition}\normalfont \label{def:commuting-square}
If one (and thus all) of the conditions \eqref{item:cs-i} to \eqref{item:cs-v} in Lemma \ref{lem:commuting-square} are satisfied, then the inclusions 
\[
\begin{matrix}
\cB_2 &\subset &\cA \\
\cup  &        & \cup \\
\cN   & \subset  & \cB_1
\end{matrix}
\]
are said to form a \emph{commuting square}. The commuting square is \emph{minimal} if $\vN(\cB_1,\cB_2) = \cA$. 
\end{Definition}
We are ready to formulate the noncommutative extended de Finetti theorem.  
\begin{Theorem}[\cite{Koes10a,GoKo09a}] \label{thm:de-finetti}
Given the random sequence $\scrI$ from $(\cC_0,\varphi_0)$ to $(\cA,\varphi)$ with tail algebra 
\[
\cC^\tail = \bigcap_{n \in\Nset_0} \vN(\iota_k(\cC_0)\mid k \ge n),
\]
consider the following statements: 
\begin{enumerate}
\item[(a)]  $\scrI$ is exchangeable;
\item[(b)]  $\scrI$ is spreadable;
\item[(c)]  $\scrI$ is stationary and full $\cC^{\tail}$-independent.
\end{enumerate}
Then we have
(a) $\Rightarrow$ (b)  $\Rightarrow$ (c) and (c) $\not\Rightarrow$  (b) $\not\Rightarrow$ (a). 
\end{Theorem}
In contrast to the classical de Finetti theorem, the reverse implications are
no longer valid in the generality of our noncommutative setting (see \cite{GoKo09a,Koes10a} for a more detailed discussion). In Theorem \ref{thm:cs-coxeter}\eqref{item:cs-coxeter-vii} we meet a new class of examples illustrating the broken equivalence of  (b) and (c): infinite sequences of represented Coxeter generators $\sigma_1, \sigma_2, \ldots$

Whenever it is convenient we can achieve by restriction of $(\cA,\varphi)$ that the random sequence $\scrI$ is \emph{minimal}, i.e., the ranges of all random variables $\iota_n$ generate $\cA$ as a von Neumann algebra:
\[
\cA = \vN(\iota_n(\cC_0) \mid n \in \Nset_0). 
\]
The following characterization of exchangeability is an extended version of \cite[Theorem 1.9]{GoKo09a}. Let $\Aut{\cA,\varphi}$ denote the $\varphi$-preserving automorphisms of $\cA$.   
\begin{Theorem}\label{thm:exchangeability}
Given the random sequence $\scrI$, consider the following statements:
\begin{enumerate}
\item[(a)] 
$\scrI$ is exchangeable;
\item[(b)] 
there exists a representation 
$\rho \colon \Sset_\infty \to \Aut{\cA,\varphi}$, such that 
\begin{align}
&&&&\iota_n &= \rho(\sigma_n \sigma_{n-1} \cdots \sigma_1)\iota_0 
&&\text{for all $n \ge 1$},&&&&\tag{PR} \label{eq:PR-exchangeable}\\
&&&& \iota_0    & = \rho(\sigma_n) \iota_0&& \text{if $n \ge 2$;}   
&&&&  \tag{L} \label{eq:L-exchangeable}
\end{align}
\item[(c)] 
there exists a representation 
$\rho \colon \Sset_\infty \to \Aut{\cA,\varphi}$, such that 
\begin{align}
&&&&\iota_n &= \rho(\gamma_n)\iota_0 
&&\text{for all $n \ge 1$},&&&&\tag{PR'} \label{eq:PR'-exchangeable}\\
&&&& \iota_0    & = \rho(\gamma_1 \gamma_2 \cdots \gamma_n \gamma_1) \iota_0&& \text{if $n \ge 2$.}   
&&&&  \tag{L'} \label{eq:L'-exchangeable}
\end{align}
\end{enumerate}
Then (a) $\Leftarrow$ (b) $\Leftrightarrow$ (c) and, if $\scrI$ is minimal, then (a) $\Leftrightarrow$ (b) $\Leftrightarrow$ (c).
\end{Theorem}
\begin{proof}
The implications between (a) and (b) are the subject of \cite[Theorem 1.9]{GoKo09a}. We are left to prove the equivalence of (b) and (c). 
\eqref{eq:PR-exchangeable} $\Leftrightarrow$ \eqref{eq:PR'-exchangeable}
follows from 
\[
\rho(\sigma_n \cdots \sigma_2\sigma_1)\iota_0 
= \rho(\sigma_n \cdots \sigma_2\sigma_1 \sigma_2 \cdots \sigma_n) \iota_0
= \rho(\sigma_1 \sigma_2\cdots \sigma_n  \cdots \sigma_2 \sigma_1)\iota_0.
\]
Finally \eqref{eq:L-exchangeable} $\Leftrightarrow$ \eqref{eq:L'-exchangeable}
is concluded as follows. Clearly, \eqref{eq:L-exchangeable} holds true if and only if $\rho(\sigma_n \cdots \sigma_2) \iota_0 = \iota_0$ for all $n \ge 2$. Now use 
$
\sigma_n \cdots \sigma_2  = (1,2,3,\ldots,n) =\gamma_1 \gamma_2 \cdots \gamma_n \gamma_1. 
$
\end{proof}
The implication (b) $\Rightarrow$ (a) tells us in particular how to construct exchangeable random sequences from a given representation $\rho \colon \Sset_\infty \to \Aut{\cA,\varphi}$. Choose some $\varphi$-conditioned von Neumann subalgebra \[
\cC_0 \subset \cA^{\rho(\Sset_{2,\infty})},
\]
where $\cA^{\rho(\Sset_{2,\infty})}$ is the fixed point algebra of $\rho(\Sset_{2,\infty})$ in $\cA$. Then the random variable $\iota_0 := \id_{|\cC_0}$ satisfies the localization property \eqref{eq:L-exchangeable}. 
An exchangeable random sequence $\scrI_{\cC_0}$ is now canonically obtained from $\iota_0$ through the product representation property \eqref{eq:PR-exchangeable}. Of course this sequence can also be constructed by using \eqref{eq:L'-exchangeable} and \eqref{eq:PR'-exchangeable}. 
\begin{Definition}\normalfont
The random sequence $\scrI_{\cC_0}$ is called \emph{associated to the representation $\rho$}. 
\end{Definition}
A random sequence $\scrI_{\cC_0}$ associated to the representation $\rho$ is exchangeable by construction. Note also that $\scrI_{\cC_0}$ may \emph{not} be minimal, even under the maximal possible choice $\cC_0 = \cA^{\rho(\Sset_{2,\infty})}$. 
Exchangeability implies stationarity and the stationarity of a random sequence $\scrI$ yields an endomorphism $\alpha$ of $ \vN(\iota_n(\cC_0)\mid n \in \Nset_0)$ such that $\iota_n = \alpha^n \iota_0$ for all $n \in \Nset_0$ (see \cite[Section 2]{Koes10a} for example). Given the representation $\rho$, we can extend this endomorphism $\alpha$ to a possibly larger algebra in $\cA$ than $\vN(\iota_n(\cC_0)\mid n \in \Nset_0)$. For this purpose, consider the fixed point algebras
\[
\cA_{n-1}^\rho:= \cA^{\rho(\Sset_{n+1,\infty})} 
\]
which provide us with the tower of von Neumann subalgebras
\[
\cA_{-1}^\rho \subset \cA_{0}^\rho \subset \cA_{1}^\rho \subset \cA_{2}^\rho \subset \cdots \subset \cA_{\infty}^\rho \subset \cA, 
\]  
where $\cA_\infty^\rho:= \vN(\cA_n^{\rho}\mid n \in \Nset_0)$. In general,
$\cA_\infty^\rho$ may be strictly contained in $\cA$.

\begin{Definition}\normalfont\label{def:generating}
The representation $\rho\colon \Sset_\infty \to \Aut{\cA}$ has the  \emph{generating property} if $\cA_\infty^\rho = \cA$.  
\end{Definition}
Similar to the more general setting \cite[Section 3]{GoKo09a}, there are representations of $\Sset_\infty$ with and without this generating property. 
Here we will concentrate on how this generating property can be obtained by restriction. 
\begin{Proposition}
The representation $\rho: \Sset_\infty \to \Aut{\cA,\varphi}$ restricts to 
the generating representation $\rho^{\operatorname{res}}\colon \Sset_\infty 
\to \Aut{\cA^\rho_\infty, \varphi^\rho_\infty}$ such that 
$\rho(\sigma_i)(\cA_\infty^\rho) \subset \cA_\infty^\rho$ and 
$E_{\cA_\infty^\rho} E_{\cA^{\rho(\sigma_i)}} = E_{\cA^{\rho(\sigma_i)}}E_{\cA_\infty^\rho}$ 
(for all $i \in \Nset$). 
\end{Proposition}
\begin{proof}
See \cite[Proposition 3.2]{GoKo09a}.
\end{proof}
Note that a generating representation $\rho \colon \Sset_\infty \to \Aut{\cA,\varphi}$ may not satisfy the minimality condition $\cA=\vN(\alpha^n(\cA_0^\rho)\mid n\in \Nset_0)$. But in the framework of exchangeability this can always be achieved by further restriction. Slightly more general, let $\cC_0 \subset \cA_0^\rho$ be a $\varphi$-conditioned von Neumann subalgebra and define, for $n \in \Nset_0 \cup \{\infty\}$,
\[
 \cC_n 
 := \vN(\alpha^k(\cC_0)\mid 0 \le k < n+1 \le \infty) 
 =  \vN(\iota_k(\cC_0)\mid 0 \le k < n+1 \le \infty) . 
\]
\begin{Proposition}\label{prop:min-process}
The representation $\rho\colon \Sset_\infty \to \Aut{\cA,\varphi}$ restricts to a generating representation from $\Sset_\infty$ into 
$\Aut{\cC_\infty, \varphi_\infty}$. Here $\varphi_\infty$ is the restriction of $\varphi$ to $\cC_\infty$.  
\end{Proposition}
\begin{proof}
The global invariance of $\cC_\infty$ under the action of $\rho(\Sset_\infty)$ is immediate if $\rho(\sigma_k) \alpha^n(\cC_0) \subset \cC_\infty$ for all $k \in \Nset$ and $n \in \Nset_0$. But this is clear since $\alpha^n(\cC_0) = \rho(\sigma_n \sigma_{n-1} \cdots \sigma_1 \sigma_0)(\cC_0)$ and thus 
\[
\rho(\sigma_k) \alpha^n (\cC_0) =
\begin{cases}
\alpha^n (\cC_0) & \text{if $k < n$ or $k > n+1$},\\
\alpha^{n-1} (\cC_0) & \text{if $k = n$},\\
\alpha^{n+1} (\cC_0) & \text{if $k = n+1$}.\\
\end{cases}
\]
The generating property follows from $\cC_n \subset \cC_n^{\rho} := \cC_\infty \cap \cA^\rho_n$.
\end{proof}
As a consequence of the previous result, if the representation $\rho$ comes from a minimal exchangeable random sequence, then $\rho$ enjoys the generating property. 
\begin{Remark}\normalfont
The proof of Proposition \ref{prop:min-process} hinges on the idempotence of the generators of $\Sset_\infty$. This allows us to show that $\rho(\sigma_n)$ maps $\alpha^n(\cC_0)$ onto $\alpha^{n-1}(\cC_0)$. This may fail in the more general context of braid group representations in \cite{GoKo09a}. 
\end{Remark}
If a representation $\rho$ of $\Sset_\infty$ has the generating property then it is straightforward to verify that $\cA$ is generated by the fixed point algebras $\cA^{\rho(\sigma_n)}$ with $n \in \Nset$. The generating property enables us to construct so-called adapted endomorphisms with product representation as follows (see \cite[Appendix A]{GoKo09a} for a more detailed discussion of such endomorphisms).  
\begin{Proposition}\label{prop:endo-rs}
Suppose the representation $\rho\colon \Sset_\infty \to \Aut{\cA,\varphi}$
is generating. Then
\[
\alpha(x) := \sotlim_{n \to \infty} \rho(\sigma_1\sigma_2 \cdots \sigma_n)(x), \qquad x \in \cA,
\]
defines a $\varphi$-preserving endomorphism $\alpha$ of $\cA$ and the exchangeable random sequence $\scrI_{\cC_0}$ (defined above) is also given by
\[
\iota_n= \alpha^n|_{\cC_0}.
\]
\end{Proposition}
\begin{proof}
See \cite[Proposition 3.8]{GoKo09a}.
\end{proof}
We continue with a fixed point characterization deduced from the results in \cite{GoKo09a,Koes10a}. Note that $\cA_{-1}^\rho = \cA^{\rho(\Sset_\infty)}$
from the definition of the fixed point algebras $\cA_n^\rho$. 
\begin{Theorem}\label{thm:fixed-point}
Let $\scrI_{\cC_0},\scrI_{\cA_0^\rho}$ be two random sequences associated to the generating representation $\rho\colon \Sset_\infty \to \Aut{\cA,\varphi}$. Suppose further that $\cN$ is a $\varphi$-conditioned von Neumann subalgebra of the tail algebra $\cC^\tail$ of $\scrI_{\cC_0}$. Let $\cA^{\rho,\tail}$ denote the tail algebra of $\scrI_{\cA_0^\rho}$.
\begin{enumerate}
\item \label{item:fixed-point-i}
If $\scrI_{\cC_0}$ is order $\cN$-factorizable, then it holds 
\[
\cN = \cC^\tail \subset \cA^{\rho,\tail}\subset \cA^{\rho(\Sset_\infty)} \subset \cA^\alpha.
\]
\item \label{item:fixed-point-ii}
If $\scrI_{\cC_0}$ is order $\cN$-factorizable and $\cA^{\rho(\Sset_\infty)}\subset \cC_0$, then it holds 
\[
\cN = \cC^\tail = \cA^{\rho,\tail} = \cA^{\rho(\Sset_\infty)} =\cA^\alpha.
\]
\item \label{item:fixed-point-iii}
If $\scrI_{\cC_0}$ is order $\cN$-factorizable and minimal, then we have
\[
\cN = \cC^\tail = \cA^{\rho,\tail} = \cA^{\rho(\Sset_\infty)} =\cA^\alpha.
\]
\end{enumerate}
In particular in \eqref{item:fixed-point-ii} and \eqref{item:fixed-point-iii}, these five subalgebras are trivial if $\scrI_{\cC_0}$ is order $\Cset$-factorizable. 
\end{Theorem}
The last assertion is a noncommutative version of the Hewitt-Savage zero-one law. Compare \cite[Corollary 1.6]{Kalle05a} for the corresponding classical result. 

Before turning our attention to the proof of Theorem \ref{thm:indy}, let us discuss the role of the von Neumann algebra $\cN$. Clearly all statements in Theorem \ref{thm:indy} remain true under the maximal choice $\cN:=\cC^\tail$ where the condition of order $\cN$-factorizability is superfluous by Theorem \ref{thm:de-finetti}. The full power of Theorem \ref{thm:indy} comes to the surface in applications where we know the order $\cN$-factorizability of a sequence $\scrI_{\cC_0}$, but have not yet explicitly identified some of the fixed point algebras involved in the statement of Theorem \ref{thm:fixed-point}. A nice application of \eqref{item:fixed-point-iii} will be given when completing the proof of Thoma's theorem in Proposition \ref{prop:factor}.  Applications of a braided version
of \eqref{item:fixed-point-ii} (see Remark \ref{rem:fixed-point-braid}) are given in \cite[Section 6]{GoKo09a}.  
\begin{proof}
(\ref{item:fixed-point-i}) By de Finetti's theorem, Theorem \ref{thm:de-finetti}, $\scrI_{\cC_0}$ is stationary and full $\cC^\tail$-independent. Now the conclusion $\cN = \cC^\tail$ is a special case of the generalized Kolmogorov zero-one law (see \cite[Theorem 6.1]{Koes10a}). Clearly $\cC_0 \subset \cA_0^\rho$ implies $\cC^\tail \subset \cA^{\rho,\tail}$. We know also $\cA^{\rho,\tail} \subset \cA^{\rho(\Sset_\infty)}$ from the extended version of the braided Hewitt-Savage zero-one law (see \cite[Theorem 2.5]{GoKo09a}). Further it is elementary to see that $\cA^{\rho(\Sset_\infty)} \subset \cA^\alpha$, where the latter is the fixed point algebra of the endomorphism $\alpha$ from Proposition \ref{prop:endo-rs}. Altogether we have arrived at $\cN = \cC^\tail \subset \cA^{\rho,\tail} \subset \cA^{\rho(\Sset_\infty)} \subset \cA^\alpha$. 

(\ref{item:fixed-point-ii}) 
We know from the de Finetti theorem that, in particular, $\scrI$ is order $\cC^\tail$-factorizable. Adding the assumption $\cA^\rho(\Sset_\infty) \subset \cC_0$ we infer the equality $\cC^\tail = \cA^{\rho(\Sset_\infty)}$
again from \cite[Theorem 2.5]{GoKo09a}. Now (\ref{item:fixed-point-i}) implies 
$\cN = \cC^\tail =  \cA^{\rho,\tail} \subset \cA^{\rho(\Sset_\infty)} \subset \cA^\alpha$. The remaining equality $\cA^{\rho(\Sset_\infty)} = \cA^\alpha$ is a part of the fixed point characterization in \cite[Theorem 0.3]{GoKo09a}.

(\ref{item:fixed-point-iii}) The assumptions of the fixed point characterization $\cN = \cC^\tail = \cA^\alpha$ from \cite[Theorem 6.4]{Koes10a} are all in place: minimality, stationarity, order $\cN$-factorizability and $\cN \subset \cA^\alpha$. Together with (\ref{item:fixed-point-i}) this entails $\cN = \cC^\tail = \cA^{\rho,\tail} = \cA^{\rho(\Sset_\infty)} = \cA^\alpha$.
\end{proof}
\begin{Remark}\normalfont \label{rem:fixed-point-braid}
Theorem \ref{thm:fixed-point} remains valid if the symmetric group $\Sset_\infty$ is replaced by the braid group $\Bset_\infty$. In this case a random sequence $\scrI_{\cC_0}$ associated to the generating representation $\rho\colon \Bset_\infty  \to \Aut{\cA,\varphi}$ is a braidable random sequence in the sense of \cite[Definition 0.1]{GoKo09a}. 
\end{Remark}
Similar to subfactor theory \cite{GHJ89a,JoSu97a} we obtain rich structures of commuting squares.
\begin{Theorem} \label{thm:endo-braid-i}
Assume that the probability space $(\cA,\varphi)$ is equipped with the 
generating representation $\rho\colon \Sset_\infty \to \Aut{\cA,\varphi}$ 
and let $\cA_{n-1}^\rho := \cA^{\rho(\Sset_{n+1,\infty})}$, with $n\in \Nset_0$.
Then one obtains a triangular tower of inclusions such that each cell forms 
a commuting square: 
\begin{eqnarray*}
\setcounter{MaxMatrixCols}{20}
\begin{matrix}
\cA_{-1}^\rho &\subset&   \cA_0^\rho & \subset & \cA_1^\rho & \subset &\cA_2^\rho & \subset &\cA_3^\rho & \subset  & \cdots & \subset & \cA\\
        &&          \cup  &         & \cup  &         & \cup &         & \cup  &       & & & \cup  \\
        &&   \cA_{-1}^\rho&\subset&\alpha(\cA_0^\rho)&\subset&\alpha(\cA_1^\rho)&\subset&\alpha(\cA_2^\rho)&\subset& \cdots & \subset & \alpha(\cA)\\
         &&               &         & \cup  &         & \cup   &         & \cup  &       & & & \cup \\
              &&&&   \cA_{-1}^\rho& \subset & \alpha^2 (\cA_0^\rho)  & \subset & \alpha^2(\cA_1^\rho)
              & \subset  & \cdots & \subset & \alpha^2(\cA) \\
       &&&&&&   \cup           &         & \cup  &       &&  & \cup \\
        &&&&&&   \vdots           &         & \vdots  &       &&  & \vdots 
\end{matrix}
\setcounter{MaxMatrixCols}{10}
\end{eqnarray*}
Moreover, the (exchangeable) random sequence $\scrI_{\cA_0^\rho}$ associated to $\rho$ is full $\cN$-independent with
\[
\cN = \cA^\tail = \cA^{\rho}_{-1} = \cA^{\alpha} = \cA^{\rho(\Sset_\infty)},
\]
where $\cA^\tail$ is the tail algebra of the sequence $\scrI_{\cA_0^\rho}$.
\end{Theorem}
\begin{proof}
See \cite[Theorem 3.9]{GoKo09a} for the inclusions and the commuting square properties of each cell. The properties of $\scrI_{\cA_0^\rho}$ are immediate from Proposition \ref{prop:endo-rs} and Theorem \ref{thm:fixed-point} (or \cite[Theorem 0.3]{GoKo09a}). 
\end{proof}
`Shifted' representations allow us to turn each fixed point algebra $\cA_n$ into the tail algebra of a certain exchangeable random sequence. 
\begin{Corollary}\label{cor:endo-braid-i}
Let $n \in \Nset_0$ and consider, under the assumptions of Theorem \ref{thm:endo-braid-i}, the $n$-shifted endomorphism 
\[
\alpha_n := \lim_{k \to \infty} \rho \circ\sh^n(\sigma_1\sigma_2 \cdots \sigma_k)
\]
in the pointwise strong operator topology.  
Then one obtains, for each $n \in \Nset_0$, a triangular tower of inclusions such that each cell forms 
a commuting square: 
\begin{eqnarray*}
\setcounter{MaxMatrixCols}{20}
\begin{matrix}
\cA_{n-1}^\rho &\subset&   \cA_n^\rho & \subset & \cA_{n+1}^\rho & \subset &\cA_{n+2}^\rho & \subset &\cA_{n+3}^\rho & \subset  & \cdots & \subset & \cA\\
        &&          \cup  &         & \cup  &         & \cup &         & \cup  &       & & & \cup  \\
        &&   \cA_{n-1}^\rho&\subset&\alpha_n(\cA_n^\rho)&\subset&\alpha_n(\cA_{n+1}^\rho)&\subset&\alpha_n(\cA_{n+2}^\rho)&\subset& \cdots & \subset & \alpha_n(\cA)\\
         &&               &         & \cup  &         & \cup   &         & \cup  &       & & & \cup \\
              &&&&   \cA_{n-1}^\rho& \subset & \alpha_n^2 (\cA_n^\rho)  & \subset & \alpha_n^2(\cA_{n+1}^\rho)
              & \subset  & \cdots & \subset & \alpha_n^2(\cA) \\
       &&&&&&   \cup           &         & \cup  &       &&  & \cup \\
        &&&&&&   \vdots           &         & \vdots  &       &&  & \vdots 
\end{matrix}
\setcounter{MaxMatrixCols}{10}
\end{eqnarray*}
Moreover, the random sequence 
$\scrI_{\cA_n^\rho}^{(n)}$ associated to $\rho \circ \sh^n$, with random variables
\begin{eqnarray*}
\iota_k^{(n)} = \rho \circ \sh^n(\sigma_k \sigma_{k-1}\cdots \sigma_1 \sigma_0)|_{\cA_n^\rho},
\end{eqnarray*}
is exchangeable and full $\cN$-independent with
\[
\cN = \cA^{(n),\tail} = \cA^{\rho}_{n-1} = \cA^{\alpha_n} = \cA^{\rho \circ \sh^n(\Sset_\infty)},
\]
where $\cA^{(n),\tail}$ is the tail algebra of $\scrI_{\cA_n^\rho}^{(n)}$.
\end{Corollary}
Note that $\alpha_0= \alpha$ and $\iota_k^{(0)}= \iota_k$, as well as $\scrI_{\cA_0^\rho}^{(0)} =  \scrI_{\cA_0^\rho}$.    
\begin{proof}
Given the representation $\rho$ from Theorem \ref{thm:endo-braid-i}, we obtain the $n$-shifted representation $\rho \circ \sh^n \colon \Sset_\infty \to \Aut{\cA,\varphi}$, whence Theorem \ref{thm:endo-braid-i} applies again. Since $\sh(\sigma_i) = \sigma_{i+1}$ (see Definition \ref{def:shifts}), one has $\rho \circ \sh^n(\Sset_{k+2, \infty}) = \rho(\Sset_{k+n+2, \infty})$ for $k \ge -1$ and hence $\cA_k^{\rho\circ \,\sh^n} = \cA_{k+n}^{\rho}$.    
Finally, the exchangeability of $\iota^{(n)}$ is immediate since $\rho \circ\, \sh^n$ is another representation of $\Sset_\infty$ and thus the arguments from the proof of Theorem \ref{thm:endo-braid-i} transfer.
\end{proof}

\section{Unitary representations of $\Sset_\infty$ and Thoma multiplicativity}\label{section:multiplicative}
We specify our general results from Section \ref{section:exchange-indep} to tracial probability spaces generated by unitary representations of $\Sset_\infty$. Our starting point is the von Neumann algebra $\cA$ generated by such a unitary representation. Now a probability space $(\cA, \varphi)$ is obtained by choosing a fixed faithful normal state $\varphi$ on $\cA$. As we will see in Proposition \ref{prop:trace}, exchangeability puts strong constraints on the choice of such states $\varphi$. This reveals the two distinguished roles of star generators and Coxeter generators in our approach. The first ones will provide us with exchangeable sequences and the latter ones will implement certain actions on the star generators. We collect in Theorem \ref{thm:indy} some results for the tracial setting of unitary representations which are immediate from our general de Finetti type results of Section \ref{section:exchange-indep}. Finally, as an application of our approach, we give a new proof of Thoma multiplicativity from noncommutative conditional independence, see Theorem \ref{thm:thoma-mult}. 

Suppose $\pi$ is a unitary representation of $\Sset_\infty$ in $\cB(\cH)$, the bounded operators on the separable Hilbert space $\cH$. Let $\cA$ denote the von Neumann subalgebra in $\cB(\cH)$ generated by $\pi(\Sset_\infty)$ and, more generally, put 
\[
\vN_\pi(S) := \vN\big(\pi(S)\big)
\]
for a set $S \subset \Sset_\infty$. The Coxeter generators $\sigma_i$ and the star generators $\gamma_i$ will play a distinguished role in our treatment of the representation $\pi$.
\begin{Definition}\normalfont \label{def:rep-gen}
The unitaries 
\[
u_i := \pi(\sigma_i)\quad\text{and}\quad  v_i := \pi(\gamma_i)   \qquad \qquad (n \in \Nset)
\]
are called \emph{(represented) Coxeter generators} and
\emph{(represented) star generators}, respectively. The identity in $\cA$ is also denoted by $u_0 = \pi(\sigma_0)$ or $v_0:= \pi(\gamma_0)$. A unitary $v$ of the form  
\[
v = (v_{n_1}v_{n_2}\cdots v_{n_k}) v_{n_1} = \pi\big((\gamma_{n_1}\gamma_{n_2}\cdots \gamma_{n_k})\gamma_{n_1}\big)
\] 
with distinct $n_1, n_2,\ldots, n_k \in \Nset_0$ and $k \in \Nset$ is called a \emph{(represented) $k$-cycle} provided $n_1 =0$ if $\min\{n_1,\ldots,n_k\}=0$. 
The $k$-cycle $v$ is said to be \emph{central} if $v \in \cZ(\cA)$. Here $\cZ(\cA)=\set{x\in \cA}{xy=yx, y \in \cA}$ denotes the center of $\cA$.
\end{Definition}
Next we define the representation 
$\rho_0^{} := \Ad \pi$ of $\Sset_\infty$ in the (inner) automorphisms of $\cA$  
and, more generally for $N \in \Nset_0$, the $N$-shifted representations
\[
\rho_N := \rho_0^{} \circ \, \sh^N. 
\]
The representation $\rho_0^{}$ gives rise to the fixed point algebras
\begin{eqnarray*} 
\cA_{n} &:=&\cA^{\rho_0^{}(\Sset_{n+2, \infty})} \qquad (-1 \le n < \infty),\\ 
\cA_\infty &:=& \vN(\cA_{n} \mid n \ge -1).
\end{eqnarray*}
We collect some elementary properties of the algebras $\vN(\Sset_{n+1})$ and $\cA_n$. 
\begin{Lemma} \label{lem:inclusions}
The representations $\rho_N\colon \Sset_\infty \to \Aut{\cA}$ are generating for all $N \in \Nset_0$ and the fixed point algebras $\cA_n$ are relative commutants:
\begin{equation*}
\cA_{n} = \vN_\pi(\Sset_{n+2,\infty})^\prime \cap \cA. 
\end{equation*}
We have the following double tower of inclusions:
\begin{eqnarray*}
\setcounter{MaxMatrixCols}{20}
\small
\begin{matrix}
& & \vN_\pi(\Sset_1)  & \subset & \vN_\pi(\Sset_2) & \subset & \vN_\pi(\Sset_3)& \subset \cdots  
\subset & \vN_\pi(\Sset_{n+1}) & \subset \cdots
\subset & \vN_\pi(\Sset_\infty)\\
& & \cap  &      & \cap & & \cap  &&\cap &&  \shortparallel\\
\cA_{-1} & \subset & \cA_0 & \subset & \cA_1 & \subset & \cA_2 & \subset \cdots                    
\subset & \cA_n & \subset \cdots 
\subset & \cA_\infty \\ 
\shortparallel &&&&&&&&&& \shortparallel\\
\cZ(\cA) &&&&&&&&&& \cA
\end{matrix}
\setcounter{MaxMatrixCols}{10}
\end{eqnarray*}
\end{Lemma}
\begin{proof}
Our standing assumption is $\cA= \vN_\pi(\Sset_\infty)$. Now the generating property of $\rho_0^{}$ follows from
\begin{eqnarray*}
\vN_\pi(\Sset_\infty) &\supset& \bigvee_{n \ge 0} \big(\vN_\pi(\Sset_\infty)\big)^{\rho_0^{}(\Sset_{n+2,\infty})}                                              \supset \bigvee_{n \ge 0}  \vN_\pi(\Sset_{n+1})  = \vN_\pi(\Sset_\infty).   
\end{eqnarray*}
The generating property of $\rho_N$ for $N >0$ is immediate from these inclusions since $\sh^N(\Sset_{n+2,\infty}) = \Sset_{N+n+2,\infty}$. The rest is evident.

Each $\rho_0^{}(\sigma_k)$ is an inner automorphism of $\cA$ and thus the algebras $\cA_{n}$ are the relative commutants as stated above. All horizontal inclusions in the double tower are obvious. Since $\vN_\pi(\Sset_{n+1}) = \vN\set{u_k}{k \le n}$ and $\vN_\pi(\Sset_{n+2,\infty}) = \vN\set{u_k}{k \ge n+2}$, all vertical inclusions $\vN_\pi(\Sset_{n+1}) \subset \cA_n$ and $\cA_{-1} = \cZ(\cA)$ are evident from \eqref{eq:B2}.  
\end{proof}

Of central interest in the study of the representation theory of $\Sset_\infty$ is the identification of the fixed point algebras $\cA_0, \cA_1, \cA_2, \ldots$ Here we will pursue an approach to their identification within our framework of noncommutative probability spaces. Thus we need to invoke faithful normal states on the von Neumann algebra $\cA$. The representation $\rho_0^{}$ may fail to preserve a state $\varphi$ of $\cA$. But if this invariance property is in place, then we are in a particularly nice situation. Recall that the centralizer of $\cA$ w.r.t.~the faithful normal state $\varphi$ is the von Neumann subalgebra $\cA^\varphi= \{x \in \cA \mid \varphi(xy) = \varphi(yx), y \in \cA\}$. 
\begin{Proposition}\label{prop:trace}
Let the representations $\pi$ and $\rho_0^{}$ of\, $\Sset_\infty$ be as introduced above and suppose $\varphi$ is a faithful normal state on $\cA= \vN_\pi(\Sset_\infty)$ so that $(\cA,\varphi)$ is a probability space. Then the following three conditions are equivalent:
\begin{enumerate}
\item \label{item:prop-trace-a}
$\rho_0^{}(\Sset_\infty) \subset \Aut{\cA,\varphi}$;
\item \label{item:prop-trace-b}
$\varphi$ is tracial; 
\item \label{item:prop-trace-c}
the sequence $(v_i)_{i \in \Nset}$ is exchangeable and $v_1 \in \cA^\varphi$.
\end{enumerate}
If one (and thus all) of these three conditions is satisfied then, for each $n \in \Nset_0$, the $n$-shifted endomorphism 
\[
\alpha_n = \lim_{N \to \infty} \Ad \pi \circ \sh^n (\sigma_1 \sigma_2 \cdots \sigma_N) =  \lim_{N \to \infty} \Ad u_{n+1} u_{n+2} \cdots u_{N}
\] 
(from Corollary \ref{cor:endo-braid-i}) is $\varphi$-preserving. Moreover $\alpha_n$ has the fixed point algebra $ \cA^{\alpha_n} = \cA_{n-1}$ and satisfies $ \alpha_n \circ  \pi  = \pi \circ \sh_n$.
In particular, for all $i \in \Nset$, 
\begin{align} \label{eq:prop-trace}
\alpha_0(u_i) = u_{i+1}, \qquad 
\alpha_n(v_i) = 
\begin{cases}
v_i & \text{if $ i < n $,}\\
v_{i+1} & \text{if $ i \ge n$}
\end{cases}
\qquad (n \ge 1).
\end{align}
\end{Proposition}
Each character of $\Sset_\infty$ provides us with a tracial probability space (see Lemma \ref{lem:kol}) and the equivalences above show the distinguished role of the star generators $v_i$.  Condition \eqref{item:prop-trace-c} also indicates that certain non-tracial settings are in reach of our approach, roughly speaking, by dropping the centralizer condition on $v_1$ (see Remark \ref{rem:non-tracial}). In the following we will restrict our investigations to the tracial setting as it appears in Proposition \ref{prop:trace}. 
\begin{proof}
We start with the properties of the $\alpha_n$'s. For this purpose 
suppose $\rho_0^{}(\Sset_\infty) \subset \Aut{\cA,\varphi}$. Then
$\cA_{n-1}$ is the fixed point algebra of $\alpha_{n}$ by Theorem \ref{thm:fixed-point}. The other equations are clear from Definition \ref{def:shifts}, Lemma \ref{lem:shift-1} and Lemma \ref{lem:shift-2}. 
We are left to show the equivalence of \eqref{item:prop-trace-a} to \eqref{item:prop-trace-c}:

`\eqref{item:prop-trace-a} $\Leftrightarrow$ \eqref{item:prop-trace-b}':
We conclude from $\varphi = \varphi \circ \rho_0^{}(\sigma) = \varphi \circ \Ad \pi(\sigma)$ that $\varphi\big(\pi(\sigma)\pi(\tau)\big) = \varphi\big(\pi(\tau)\pi(\sigma)\big)$ for all $\sigma, \tau \in \Sset_\infty$. Since the set $\pi(\Sset_\infty)$ is weak*-total in $\cA$, the state $\varphi$ is a trace. The converse implication is clear. 

`\eqref{item:prop-trace-a} $\Rightarrow$ \eqref{item:prop-trace-c}':
We make use of the idea introduced after Theorem \ref{thm:exchangeability} of how to construct an exchangeable random sequence. Here we apply it to the $1$-shifted representation $\rho_1 = \rho_0^{} \circ \sh$. Clearly $v_1 \in \cA^{\rho_1(\Sset_{2,\infty})} = \cA^{\rho_0^{}(\Sset_{3,\infty})} = \cA_1$ by \eqref{eq:B2}. Since $\varphi$ is tracial the subalgebra $\cC_0^{(1)} :=\vN(v_1)$ is $\varphi$-conditioned. Now put $\iota_0^{(1)} := \id|_{\cC_0^{(1)}}$. By Theorem \ref{thm:exchangeability},  
\[
\iota_n^{(1)} := \rho_1(\sigma_n \cdots \sigma_1 \sigma_0)\iota_0^{(1)}   \qquad (n \in \Nset_0)   
\]  
defines an exchangeable random sequence. We conclude further by Proposition \ref{prop:endo-rs} and \eqref{eq:prop-trace} that $\iota_n^{(1)}(v_1) = \alpha_1^{n}(v_1) = v_{n+1}$. Thus $(v_i)_{i\in \Nset}$ is exchangeable. 
Finally, $v_1 \in \cA^\varphi$ is obvious from the traciality of $\varphi$.

`\eqref{item:prop-trace-c} $\Rightarrow$ \eqref{item:prop-trace-a}':
By Definition \ref{def:convention}, $(v_i)_{i \in \Nset}$ 
is induced by an exchangeable random sequence $(\iota_n)_{n=0}^\infty$ from
$(\cC_0,\varphi|_{\cC_0})$ to $(\cA,\varphi)$ such that $\iota_0|_{\cC_0} = \id|_{\cC_0}$ for some $\varphi$-conditioned subalgebra $\cC_0$ of $\cA$ and
$ \qquad v_{n+1} = \iota_n(v_1)$ for $n \in \Nset_0$. Note that $v_1 \in \cC_0$. 
Since $\big(\pi(\gamma_i)\big)_{i \in \Nset}$ generates $\cA$, this exchangeable random sequence is minimal and thus the characterization from Theorem \ref{thm:exchangeability} applies: there exists a representation $\rho\colon \Sset_\infty \to \Aut{\cA, \varphi}$ such that
\begin{align*}
&&&&v_{n+1} &= \rho(\sigma_n \sigma_{n-1} \cdots \sigma_1)(v_1) 
&&\text{for all $n \ge 1$},&&&&\\
&&&& v_1    & = \rho(\sigma_n) (v_1) && \text{if $n \ge 2$.}   
&&&& 
\end{align*}
Similar as in the proof of Proposition \ref{prop:min-process} the representation $\rho$ satisfies, for $k, n \in \Nset$,
\begin{eqnarray*}
\rho(\sigma_n)  (v_k) 
= \begin{cases}
v_k  & \text{if $k < n$ or $k > n+1$},\\
v_{k-1}  & \text{if $k = n+1$},\\
v_{k+1}  & \text{if $k = n$}.
\end{cases}
\end{eqnarray*}  
Now we can identify this representation $\rho$ as the $1$-shifted representation $\rho_1$. Indeed, a simple algebraic calculation shows that 
\[
\rho_1(\sigma_n)\big(v_k\big) =
\pi(\sigma_{n+1} \gamma_k \sigma_{n+1})
= \rho(\sigma_n) (v_k)
\] for $k, n \in \Nset$. Consequently, $\rho_1 (\Sset_\infty) = \rho_0^{}(\Sset_{2,\infty})\subset \Aut{\cA,\varphi}$. But this entails  $\rho_0^{}(\Sset_\infty) \subset\Aut{\cA,\varphi}$, since $\pi(\sigma_1) = \pi(\gamma_1) = v_1\in \cA^\varphi$ is stipulated.  
\end{proof}
In addition to the exchangeability of star generators we have a strong relationship between them and cycles, as it is evident from Definition \ref{def:rep-gen}, Lemma \ref{lem:cycle-1} and Lemma \ref{lem:cycle-2}: 
\begin{Lemma}\label{lem:cycle-rep}
The permutations $s_1, s_2, \ldots, s_k$ in $\Sset_\infty$ are non-trivial disjoint cycles if and only if the unitaries $\pi(s_1), \pi(s_2), \ldots, \pi(s_k)$ are represented cycles involving mutually disjoint sets of star generators $v_i$. 
\end{Lemma}
The combination of Proposition \ref{prop:trace} and Lemma \ref{lem:cycle-rep} implies strong factorization results of noncommutative de Finetti type for unitary representations of the infinite symmetric group, as we will show below and, more systematically, in Section \ref{section:limit-cycles}.   
\begin{Remark} \normalfont
The Coxeter generators $u_i$ do clearly lack the strong features of star generators from Proposition \ref{prop:trace} and Lemma \ref{lem:cycle-rep}. But the role of the $u_i$'s becomes apparent on the level of actions on the star generators, as they appear for the $n$-shifted endomorphisms $\alpha_n$ in Proposition \ref{prop:trace}. This allows us to apply basic tools from noncommutative ergodic theory. 

On the other hand the sequence of Coxeter generators $(u_i)_{i \in \Nset}$ enjoys stationarity as a distributional symmetry. Quite surprising will be our result, obtained in Theorem \ref{thm:cs-coxeter}, that this sequence is full $\cZ(\cA)$-independent under certain additional assumptions on the representation $\pi$, even though this sequence is neither exchangeable nor spreadable whenever $u_1 \not\in \cZ(\cA)$.  
\end{Remark}
From now on we deal with unitary representations $\pi\colon \Sset_\infty  \to \cB(\cH)$ where the von Neumann algebra $\cA= \vN_\pi(\Sset_\infty)$ is equipped with a faithful normal tracial state $\trace$. In other words, we consider representations $\pi \colon \Sset_\infty \to \cU(\cA)$, the unitaries of $\cA$, which generate the von Neumann algebra $\cA$ of the tracial probability space $(\cA,\trace)$. Consequently, each of the fixed point algebras $\cA_n$, with $-1\le n < \infty$, is $\trace$-conditioned and the maps
\[
E_n\colon \cA \to \cA_n
\]
will denote the corresponding $\trace$-preserving conditional expectations.

Let us now apply our general results from Section \ref{section:exchange-indep}. The next theorem may actually be regarded as a corollary of the noncommutative de Finetti theorem, Theorem \ref{thm:de-finetti}, and the fixed point characterizations in Theorem \ref{thm:fixed-point}. 
\begin{Theorem}\label{thm:indy}
Let $\pi\colon \Sset_\infty \to \cU(\cA)$ be a unitary representation generating the von Neumann algebra $\cA$ of the tracial probability space $(\cA,\trace)$.  
\begin{enumerate} 
\item\label{item:indy-i}
The sequence $\big(\alpha_0^k(\cA_0)\big)_{k \in \Nset_0}$ is exchangeable and full $\cN$-independent with
\[
\cN = \cB^{\tail} = \cA^{\alpha_0} = \cA^{\rho_0^{}(\Sset_\infty)} = \cA_{-1}= \cZ(\cA), 
\] 
where $\cB^\tail$ is the tail algebra of the random sequence $\scrI_{\cA_0}$ associated to $\rho_0^{} = \Ad \pi$.  
\item\label{item:indy-iii}
The sequence $\Big(\alpha_{n}^k\big(\vN_\pi(\Sset_{n+1})\big)\Big)_{k \in \Nset_0}$ is minimal, exchangeable and full $\cN$-independent with
\[
\cN = \cC^{\tail} = \cA^{\alpha_n} = \cA^{\rho_n(\Sset_\infty)} = \cA_{n-1} \qquad \qquad (n \in \Nset). 
\] 
Here $\cC^{\tail}$ is the tail algebra of the random sequence $\scrI_{\vN_\pi(\Sset_{n+1})}$ associated to $\rho_n$. 
\item\label{item:indy-ii}
The sequence of star generators $\big(v_i\big)_{i \in \Nset}$ is minimal, exchangeable and full $\cN$-independent with
\[
\cN = \cA^\tail = \cA^{\alpha_1} = \cA^{\rho_1(\Sset_\infty)} = \cA_{0},
\]
where $\cA^\tail$ is the tail algebra of the sequence $\big(v_i\big)_{i \in \Nset}$. 
\item\label{item:indy-iv}
The sequence $\big(v_i\big)_{i \in \Nset}$ is minimal and full $\cA_{n-1}$-independent for $n > 1$.
\end{enumerate}
\end{Theorem}
\begin{proof}
\eqref{item:indy-i}
Apply  Theorem \ref{thm:endo-braid-i}. 
\eqref{item:indy-iii} 
Clearly the sequence is minimal. Since $\vN_\pi(\Sset_{n+1}) \subset \cA_{n}$, Corollary  \ref{cor:endo-braid-i} applies. 
\eqref{item:indy-ii}
By Proposition \ref{prop:trace}, we have $v_i = \alpha_1^{i-1}(v_1) \in \alpha_1^{i-1}(\cA_{1})$. Thus this is the special case $n=1$ of \eqref{item:indy-iii}. 
We are left to prove \eqref{item:indy-iv}. As before by Proposition \ref{prop:trace}, $ v_i = \alpha_n^{i-n}(v_n) \in \alpha_n^{i-n}(\cA_{n})$ for $i \ge n$ and $v_i \in \cA_{n-1}$ for $i < n$. Now Corollary
\ref{cor:endo-braid-i} yields the full $\cA_{n-1}$-independence of the sequence $(v_i)_{i \in \Nset}$. 
\end{proof}
Now Thoma multiplicativity is an immediate consequence of the conditional independence properties of the two sequences $\big(\alpha^i(\cA_0)\big)_{i \in \Nset_0}$ and $\big(v_i\big)_{i \in \Nset}$. 
Recall that $m_k(\sigma)$ denotes the number of $k$-cycles in the cycle decomposition of the permutation $\sigma$. 
\begin{Theorem}[Thoma multiplicativity]\label{thm:thoma-mult}
Suppose $\pi\colon \Sset_\infty \to \cU(\cA)$ is a unitary representation generating the von Neumann algebra $\cA$ of the tracial probability space $(\cA,\trace)$.  Let $\sigma \in \Sset_\infty$. Then
\begin{eqnarray}\label{eq:thoma-mult-i}
E_{-1}\big( \pi(\sigma)\big) 
= \prod_{k =2}^{\infty}  \left(E_{-1}
  \left( \big(E_0(v_1)\big)^{k-1}\right) \right)^{m_k(\sigma)}. 
\end{eqnarray} 
If $\cA_{-1}\simeq \Cset$ then $E_{-1}$ can be replaced by $\trace$ in \eqref{eq:thoma-mult-i}. If $\cA_{0} \simeq \Cset$ then 
\begin{eqnarray} \label{eq:thoma-mult-ii}
\trace\big( \pi(\sigma)\big) 
= \prod_{k =2}^{\infty} \trace(v_1)^{(k-1) m_k(\sigma)}. 
\end{eqnarray} 
\end{Theorem}
Note that \eqref{eq:thoma-mult-i} depends only on the cycle class type of the permutation $\sigma$. The condition $\cA_{-1} \simeq \Cset$ means that $\cA$ is a factor and  then \eqref{eq:thoma-mult-i} is commonly addressed as Thoma multiplicativity (compare \cite[]{Okou99a}). We will see in Section \ref{section:Markov} that \eqref{eq:thoma-mult-ii} corresponds to the case where $\alpha_0(\cA) \subset \cA$ is an irreducible subfactor inclusion. For example, this situation occurs in the left regular representation of $\Sset_\infty$.  

Our proof of Theorem \ref{thm:thoma-mult} combines arguments involving independence over the two fixed point algebras $\cA_0$ and $\cA_{-1}$. It reveals that Thoma multiplicativity is a consequence of noncommutative independence. Related techniques will be refined to all fixed point algebras $\cA_n$ in the next section. We start with a preparatory result for the proof of Theorem \ref{thm:thoma-mult}.
\begin{Lemma}\label{lem:thoma-mult-prep}
Let $s\in \Sset_\infty$ be a non-trivial $k$-cycle. Then 
\begin{eqnarray}
E_{0}(\pi(s)) = 
\begin{cases}
E_{-1}\left(\big(E_0(v_{1})\big)^{k-1}\right) & \text{if $s(0)= 0$},\\
\big(E_0(v_{1})\big)^{k-1}                   & \text{if $s(0)\neq 0$};
\end{cases}
\end{eqnarray}  
\end{Lemma}
\begin{proof}
Since a $k$-cycle is of the form  $s = \gamma_{n_1} \gamma_{n_2} \cdots \gamma_{n_k} \gamma_{n_1}$
for mutually distinct $n_1, n_2, \ldots, n_k$, we conclude with full $\cA_0$-independence from Theorem \ref{thm:indy} \eqref{item:indy-ii} that
\begin{align*}
E_0\big(\pi(s)\big) 
&= E_0\big(v_{n_1} v_{n_2} \cdots v_{n_k} v_{n_1}\big)\\
&= E_0\big(v_{n_1} E_0(v_{n_2} \cdots v_{n_k}) v_{n_1}\big) && \text{(by
Lemma \ref{lem:commuting-square}\eqref{item:cs-v})}\\
&= E_0\big(v_{n_1} E_0(v_{n_2}) \cdots E_0(v_{n_k}) v_{n_1}\big)\\
&= E_0\left(v_{n_1} \big(E_0(v_{1})\big)^{k-1}  v_{n_1}\right).
\end{align*}
Here we have used for the last equation that $\alpha_1(v_{i}) = v_{i+1}$, and $E_0 \circ \alpha_1 = E_0$ because $\cA_0 = \cA^\alpha_1$. If $s(0)\neq 0$ then $n_1=0$ (see the discussion after Lemma \ref{lem:cycle-1}) and we are done. It remains to consider the case $s(0) = 0$. Thus $n_1>0$ by Lemma \ref{lem:cycle-1}. We infer from $v_{n_1} x v_{n_1} = \alpha_0^{n_1} (x)$ for $x\in \cA_0$ by \eqref{eq:PR'-exchangeable} in Theorem \ref{thm:exchangeability} and Proposition \ref{prop:endo-rs} that
\begin{eqnarray*}   
E_{0}\left(v_{n_1} \big(E_0(v_{1})\big)^{k-1}  v_{n_1}\right)
&=& E_{0} \alpha_0^{n_1}\left(\big(E_0(v_{1})\big)^{k-1}\right)
\end{eqnarray*}
Since $\cA_0$ and $\alpha_0^{n_1}(\cA_0)$ are full $\cA_{-1}$-independent
by Theorem \ref{thm:indy} \eqref{item:indy-i}, we conclude further with
the commuting square property from Lemma \ref{lem:commuting-square} \eqref{item:cs-ii} that  $E_0 
\alpha_0^{n_1} E_0 = E_{-1}$ and thus
\begin{eqnarray*}   
E_{0} \alpha_0^{n_1}\left(\big(E_0(v_{1})\big)^{k-1}\right)
&=& E_{-1} \left(\big(E_0(v_{1})\big)^{k-1}\right).
\end{eqnarray*}
This completes the proof of the lemma. 
\end{proof}
\begin{proof}[Proof of Theorem \ref{thm:thoma-mult}]
Suppose $\sigma$ has the cycle decomposition $\sigma = s_1 s_2\cdots s_n$, where $s_1, s_2, \ldots, s_n$ are distinct non-trivial cycles with length $k_1, k_2, \ldots k_n \ge 2$, respectively. We conclude from Lemma \ref{lem:cycle-2} (or Lemma \ref{lem:cycle-rep}) and the full $\cA_0$-independence of $(v_i)_{i \in \Nset}$ (see Theorem \ref{thm:indy}
\eqref{item:indy-ii}) that 
\begin{equation}\label{eq:thoma-neu-i}
E_0\big(\pi(s_1 s_2 \cdots s_n) \big) 
= E_0\big(\pi(s_1)\big) E_0 \big(\pi(s_2)\big) \cdots E_0\big(\pi(s_n)\big).
\end{equation}
If $n=1$, then $\cA_{-1}\subset \cA_0$ and Lemma \ref{lem:thoma-mult-prep} imply
\[
 E_{-1} \big(\pi(s_1)\big) = E_{-1} E_0\big(\pi(s_1)\big) =  E_{-1}\left(\big(E_0(v_{1})\big)^{k_1-1}\right).
\]
We are left to consider the case $n >1$. Since distinct cycles commute, we can assume that the point $0$ is contained in the cycle $s_1$ if $\sigma(0) \neq 0$. But then each of the remaining cycles $s_i$ satisfies $s_i(0)=0$ and, by Lemma \ref{lem:thoma-mult-prep},  
\[
E_0\big(\pi(s_i)\big) =  E_{-1}\left(\big(E_0(v_{1})\big)^{k_i-1}\right).
\]
Thus \eqref{eq:thoma-neu-i} becomes
\[
E_0\big(\pi(s_1 s_2 \cdots s_n) \big) 
= E_{0}\big(\pi(s_1)\big)\,\, \prod_{i=2}^n
  E_{-1}\left(\big(E_0(v_{1})\big)^{k_i-1}\right).
\]
We finally compress this equation by $E_{-1}$ to arrive at
\[
E_{-1}\big(\pi(s_1 s_2 \cdots s_n) \big) 
= \prod_{i=1}^n E_{-1}\left(\big(E_0(v_{1})\big)^{k_i-1}\right),
\]
due to $\cA_0 \subset \cA_1$, the module property of conditional expectations and Lemma \ref{lem:thoma-mult-prep}.
\end{proof}
\begin{Remark}\normalfont \label{rem:non-tracial}
An inspection of Proposition \ref{prop:trace} and its proof yields an interesting generalization beyond the tracial setting.  Similar to Section \ref{section:exchange-indep}, the idea is to replace the representation $\rho_0^{}$ of $\Sset_\infty$ by shifted representations $\rho_N = \rho_0^{} \circ \,\sh^N$. Then one arrives at the conclusion that the following conditions are equivalent for $N \in \Nset$, under the assumptions of Proposition \ref{prop:trace}:
\begin{enumerate}
\item \label{item:rem-not-trace-i} 
$\rho_N(\Sset_\infty) \subset \Aut{\cA,\varphi}$;
\item \label{item:rem-not-trace-ii} 
$\vN_\pi(\Sset_{N+1,\infty}) \subset \cA^\varphi$;
\item \label{item:rem-not-trace-iii} 
the sequence $(v_{i+N-1})_{i \in \Nset}$ is exchangeable. 
\end{enumerate}
Proposition \ref{prop:trace} is recovered in the case $N=1$ by adding the requirement $v_1 \in \cA^\varphi$ to each of the three conditions. 
Note that Thoma multiplicativity from Theorem \ref{thm:thoma-mult} is not valid in the non-tracial case because \eqref{eq:thoma-mult-i} depends only on the cycle class type and hence restricts to a character of $\Sset_\infty$. But it can be replaced by a slightly modified version of the generalized Thoma multiplicativity provided in Theorem \ref{thm:thoma-mult-general}. In view of Remark \ref{rem:olshanski} it is of interest to investigate further possible connections between this generalized version and parametrizations of so-called `irreducible admissable representations with finite depth' considered in \cite{Olsh89a, Okou99a}. 
Here we do not dwell further on such non-tracial settings since this subject goes beyond the scope of present paper and is better postponed to another publication. 
\end{Remark}

\section{Fixed point algebras, limit cycles and generalized Thoma multiplicativity}
\label{section:limit-cycles}
We continue our investigations from the previous section for unitary representations of $\Sset_\infty$ in the context of tracial probability spaces. For this purpose we introduce limit cycles, a generalization of cycles, and study their properties. This will yield our main results of this section: the identification of all fixed point algebras $\cA_n$ (see Theorem \ref{thm:fixed-points}) and a generalization of Thoma multiplicativity (see Theorem \ref{thm:thoma-mult-general}). Finally we will briefly relate our setting to Okounkov's approach \cite{Okou99a} in Remark \ref{rem:olshanski}.

Let us recall our setting from Section \ref{section:multiplicative}: we are given the tracial probability space $(\cA,\trace)$ equipped with a unitary representation $\pi\colon \Sset_\infty \to \cU(\cA)$ such that $\cA =\vN_\pi(\Sset_\infty)$. So $\rho_0^{} := \Ad \pi$ defines a generating representation of $\Sset_\infty$ in $\Aut{\cA,\trace}$. Further $\cA_{n-1}$ denotes the fixed point algebra of $\rho_0^{}(\Sset_{n+1,\infty})$ in $\cA$ and $E_{n-1}$ the $\trace$-preserving conditional expectation from $\cA$ onto $\cA_{n-1}$ for $n \in \Nset$.  
\begin{Definition}\normalfont \label{def:limit-cycle}
Suppose $v_{n_1}v_{n_2}\cdots v_{n_k} v_{n_1} \in \cA$ is a (represented) $k$-cycle with $k \ge 1$. Then 
\[
E_{n}(v_{n_1}v_{n_2}v_{n_3} \cdots v_{n_k}v_{n_1}), \qquad 
n \ge -1, 
\]
is called a \emph{limit $k$-cycle}. Two special types of limit cycles are abbreviated as follows:
\begin{enumerate} 
\item
A \emph{limit $2$-cycle} $A_i$, with $i \in \Nset_0$,  is of the form 
\[
A_i = E_{j}(v_i v_k v_i),  \qquad 0 \le i \le j < k .
\]
\item   
A \emph{limit $k$-cycle} $C_k$ is of the form 
\[
C_k = E_{-1}(v_{i_1} v_{i_2}\cdots v_{i_k} v_{i_1}), \qquad k \ge 1.   
\]
\end{enumerate}
A limit cycle is said to be \emph{trivial} if it is a scalar multiple of the identity. A limit cycle in $\cZ(\cA)$ is said to be \emph{central}.  
\end{Definition}
The symbols $A_i$ and $C_k$ will be exclusively reserved for these two types of limit cycles, respectively. As we will see in Proposition \ref{prop:limit-cycles}, the limit 2-cycles $A_i$ depend only on the index $i$, in particular
\[
A_0 = E_0(v_0 v_1 v_0) = E_0(v_1),
\]
and the limit $k$-cycles $C_k$ depend only on the cycle length $k$, in particular
\[
C_k = E_{-1}(A_0^{k-1}).
\]  
Their prominent role becomes evident from the identification of all fixed point algebras $\cA_n$ in Theorem \ref{thm:fixed-points} as well as an immediate reformulation of \eqref{eq:thoma-mult-i} on Thoma multiplicativity in Theorem \ref{thm:thoma-mult}: 
\[
E_{-1}\big( \pi(\sigma)\big) 
= \prod_{k =2}^{\infty}  \Big(E_{-1}(A_0^{k-1})\Big)^{m_k(\sigma)}
= \prod_{k =2}^{\infty}  C_k^{m_k(\sigma)}. 
\]   
We are ready to formulate our main result of this section.  
\begin{Theorem}\label{thm:fixed-points}
Suppose the tracial probability space $(\cA,\trace)$ is equipped with the generating representation $\pi\colon \Sset_\infty \to \cU(\cA)$.
Then it holds 
\begin{eqnarray*}
\cA_{-1} &=& \cZ(\cA),\\
\cA_n &=& \cA_0 \vee \vN_\pi(\Sset_{n+1}),  
\end{eqnarray*}
with $n \in \Nset_0$. Moreover, in terms of the limit cycles $A_0 = E_0(v_1)$ and $C_k$,
\begin{eqnarray*}
\cZ(\cA) &=& \vN(C_k \mid k \ge 1),\\
\cA_0 &=&  \vN(A_0, C_k\mid k \ge 1).
\end{eqnarray*} 
In particular, $\cA_0$ is an abelian von Neumann subalgebra of $\cA$.   
\end{Theorem}
Since its proof requires preparative results on limit cycles, let us first draw some immediate conclusions. Recall that the inclusion of von Neumann algebras $\cN \subset \cM$ is said to be \emph{irreducible} if $\cM \cap \cN^\prime \simeq \Cset$.  
\begin{Corollary}\label{cor:fixed-points}
Let the setting be as in Theorem \ref{thm:fixed-points}.
\begin{enumerate}
\item \label{item:cor-fixed-points-i}
$\cA = \vN_\pi(\Sset_\infty)$ is a factor if and only if all $C_k$'s are trivial. 
\item  \label{item:cor-fixed-points-ii}
The following assertions are equivalent:
\begin{enumerate}
\item
The inclusion $\vN_\pi(\Sset_{2,\infty}) \subset \vN_\pi(\Sset_\infty)$ is irreducible;
\item
$\cA_n = \vN_\pi(\Sset_{n+1})$ for any $n \ge 0$;
\item
the fixed point algebra $\cA_0$ is trivial;
\item
the limit 2-cycle $A_0$ is trivial; 
\item 
$\trace$ is a Markov trace (see Definition \ref{def:markov-trace}).
\item The inclusion $\alpha_1(\vN_\pi(\Sset_\infty)) \subset \vN_\pi(\Sset_\infty)$ is irreducible.
\end{enumerate}
\end{enumerate}
We have the implication \eqref{item:cor-fixed-points-ii} $\Rightarrow$ \eqref{item:cor-fixed-points-i}.
\end{Corollary}
\begin{proof}
\eqref{item:cor-fixed-points-i} is clear since $\cZ(\cA)$ is generated by the $C_k$'s. \eqref{item:cor-fixed-points-ii} For the equivalence of (c) and (e) see Proposition \ref{prop:Markov}. The equivalence of (b) and (f) note that
$x \alpha_0(y) = \alpha_0(y) x$ if and only if $u_1 x u_1 \alpha_1(y) = \alpha_1(y) u_1 x u_1$, where $x,y \in \vN_\pi(\Sset_\infty)$. The rest of the proof is clear.
\end{proof}
\begin{Remark}\normalfont
If $\cA$ is a factor then $\vN(A_0) = \cA_0$. It is natural to ask if
$\vN(A_0) \not= \vN(A_0, C_k \mid k \in \Nset) = \cA_0$ may occur in the
non-factorial case. 
If the center $\cZ(\cA)$ is two-dimensional, disintegration of
$(\cA,\trace)$ into $(\cA^{(1)} \oplus \cA^{(2)}, \trace^{(1)} \oplus
\trace^{(2)})$ transforms $A_0$ into the pair of limit 2-cycles $(A_0^{(1)},
A_0^{(2)})$. If $A_0^{(1)}$ and $A_0^{(2)}$ happen to have a common
eigenvalue
(which can easily be arranged, for example by using the model in Section
\ref{section:Thoma})
then it is clear that the central projection $(\1,0)$ does not belong to
$\vN(A_0)$.
This argument shows that, in the general non-factorial setting, at least
some of the limit cycles $C_k$ are needed to generate $\cA_0$. 
\end{Remark}
We will show next that a limit $k$-cycle $E_n(v_{n_1} v_{n_2} \cdots v_{n_k} v_{n_1})$ equals either the limit cycle $C_k$ or a monomial in star generators $v_{i}$ and the limit 2-cycle $A_0$. Without loss of generality we can assume that the cycle  $v_{n_1} v_{n_2} \cdots v_{n_k} v_{n_1}$ enjoys $n_1 =\min\{n_1,\ldots, n_k\}$. Our next result generalizes Lemma \ref{lem:thoma-mult-prep}. 
\begin{Proposition}\label{prop:limit-cycles}
Suppose $n_1= \min\{n_1,\ldots, n_k\}$ with $k \ge 1$. Then a limit $k$-cycle is of the form
\begin{align*}
E_{n}(v_{n_1} v_{n_2} \cdots v_{n_k} v_{n_1}) &= 
\begin{cases} 
E_{-1}\big(A_0^{k-1}\big)              & \text{if $n_1> n\ge -1$,}\\
w_{n_1} w_{n_2} \cdots w_{n_k} w_{n_1} & \text{if $n \ge n_1 \ge 0$}, 
\end{cases}
\intertext{where} 
w_{n_i} &= 
\begin{cases}
v_{n_i} & \text{if $n_i \le n$},\\
A_0     & \text{if $n_i > n $}.
\end{cases}
\end{align*}
The limit cycles $A_i$ and $C_k$ depend only on the index $i$ and the cycle length $k$, respectively: 
\begin{align}
A_i &= E_{j}(v_i v_l v_i) = v_i A_0 v_i   &&    (0 \le i \le j < l),\label{eq:2limit}\\ 
C_k &=  E_{n}(v_{n_1} v_{n_2} \cdots v_{n_k} v_{n_1})                  
= E_{-1}\big(A_0^{k-1}\big) 
&&    ( -1 \le n < n_1, n_2,\ldots, n_k).  \label{eq:klimit}
\end{align} 
\end{Proposition}
\begin{proof}
Suppose $n_1> n\ge -1$. If $n=-1$ then our results from Thoma multiplicativity, Theorem \ref{thm:thoma-mult}, apply and yield the claimed formula. So it remains to consider the subcase $n_1 > n \ge 0$. Since $E_n$ is the conditional expectation onto the fixed point algebra of the endomorphism $\alpha_{n+1}$ and
$\alpha_{n+1} (v_i) = v_{i+1} = \alpha_1(v_i)$ for all $i > n$, we conclude that
\begin{eqnarray*}
E_n(v_{n_1} v_{n_2} \cdots v_{n_k} v_{n_1}) 
&=& E_n \alpha_{n+1}^N (v_{n_1} v_{n_2} \cdots v_{n_k} v_{n_1})\\
&=& E_n \alpha_{1}^N (v_{n_1} v_{n_2} \cdots v_{n_k} v_{n_1}) 
\end{eqnarray*}  
for any $N \in \Nset_0$. Thus, by the mean ergodic theorem (see \cite[Theorem 8.3]{Koes10a}) and the $\sot$-$\sot$-continuity of the conditional expectation $E_n$,
\begin{eqnarray*}
E_n(v_{n_1} v_{n_2} \cdots v_{n_k} v_{n_1}) 
&=& \sotlim_{N\to \infty }\frac{1}{N}\sum_{i=0}^{N-1} 
E_n \alpha_1^i(v_{n_1} v_{n_2} \cdots v_{n_k} v_{n_1})\\ 
&=& E_n E_0 (v_{n_1} v_{n_2} \cdots v_{n_k} v_{n_1})\\
&=& E_0 (v_{n_1} v_{n_2} \cdots v_{n_k} v_{n_1})\\
&=& E_{-1}\big(A_0^{k-1}\big),    
\end{eqnarray*}  
since $n_1 >0$ and thus the case $s(0)=0$ of the formula in Lemma \ref{lem:thoma-mult-prep} applies. In other words, the last equality is
due to the $\cA_{-1}$-independence of $\cA_0$ and $\alpha_0^{n_1}(\cA_0)$ 
(see Theorem \ref{thm:indy} \eqref{item:indy-i}). 
Acting on these equations with $E_{-1}$, we see that the limit $k$-cycle $C_k$ depends only on $k$: 
\[
C_k = E_{-1}(v_{n_1} v_{n_2} \cdots v_{n_k} v_{n_1}) 
= E_{-1}\big(A_0^{k-1}\big).
\]
We continue with the case $n \ge n_1 \ge 0$. By the module property of conditional expectations,
\[
E_n(v_{n_1} v_{n_2} \cdots v_{n_k} v_{n_1}) 
= v_{n_1} E_n( v_{n_2} \cdots v_{n_k})v_{n_1}. 
\]
Recall that, by Proposition \ref{prop:trace}, $\cA_n$ is the fixed point algebra of the endomorphism $\alpha_{n+1}$ and 
\[
\alpha_{n+1} (v_i)
=
\begin{cases}
v_i & \text{if $i \le n$,}\\
v_{i+1} & \text{if $i > n$}.
\end{cases}
\]
Thus 
\[
v_i
=
\begin{cases}
E_n(v_i) \in \vN_\pi(\Sset_{n+2})  & \text{if $i \le n$,}\\[2pt]
\alpha_{n+1}^{i-(n+1)} (v_{n+1}) \in \alpha_{n+1}^{i-(n+1)}\big(\vN_\pi(\Sset_{n+2})\big) & \text{if $i > n$}.
\end{cases}
\]
Now the full $\cA_n$-independence of the sequence $\Big(\alpha_{n+1}^k\big(\vN_\pi(\Sset_{n+2})\big)\Big)_{k \ge 0}$ from Theorem \ref{thm:indy} \eqref{item:indy-iii} implies 
\[
E_n(v_{n_2}\cdots v_{n_k}) 
= E_n(v_{n_2}) \cdots E_n(v_{n_k}). 
\]
We are left to show that $E_n(v_i) = A_0$ for $i > n$. As argued at the beginning of the proof,
\[
E_n(v_i) = E_n \alpha_{n+1}^N(v_i) = E_{n+1} \alpha_1^N(v_i) 
\]
for all $N \in \Nset$. As before a mean ergodic argument implies 
$E_n(v_i) = E_n E_0(v_i) = E_0(v_i)$ and further $E_0(v_i) = E_0(v_1)= A_0$.   
The final statements on the limit cycles $C_k$ and $A_i$ are clear.  
\end{proof}
\begin{proof}[Proof of Theorem \ref{thm:fixed-points}]
The fixed point algebra $\cA_{-1}$ equals the center $\cZ(\cA)$, since
\[
\cA_{-1}= \cA^{\Ad \pi(\Sset_\infty)} = \cA \cap \vN_\pi(\Sset_\infty)^\prime = \cA \cap \cA^\prime = \cZ(\cA).
\]
We consider next the fixed point algebras $\cA_n$ for $n \ge 0$.
The set $E_{n} \pi(\Sset_\infty)$ is weak* total in $\cA_n$. Our goal is to show that, for any $\sigma \in \Sset_\infty$, the operator $E_{n} \pi(\sigma)$ can be written as a monomial in terms of $v_0, v_1, \ldots, v_n$, the limit 2-cycle $A_0= E_0(v_1)$ and the limit $k$-cycles $C_k=E_{-1}(A_0^{k-1})$. If the permutation $\sigma$ is a cycle then $E_n\pi(\sigma)$ is a limit cycle and we are done by Proposition \ref{prop:limit-cycles}. The general case reduces to the cycle case along the following arguments. Let $\sigma= s_1 s_2 \cdots s_p$ be the cycle decomposition of the permutation $\sigma$. Here each cycle $s_q$ involves only star generators $\gamma_i$ with $i \in I_q \subset \Nset$ such that the sets $I_q$ are mutually disjoint (see Lemma \ref{lem:cycle-2}). We claim that $E_n(\pi(\sigma))$ is a product of limit cycles, more precisely:
\begin{equation} \label{eq:fixed-points-0}
E_n\pi(\sigma) = E_n \pi(s_1) \cdots E_n \pi(s_p).
\end{equation}
Indeed, this factorization is immediate from Theorem  \ref{thm:indy}\eqref{item:indy-iii}, the full $\cA_n$-independence of
\begin{equation} \label{eq:fixed-points-1}
\Big(\alpha_{n+1}^k\big(\vN_\pi(\Sset_{n+2})\big)\Big)_{k \ge 0}.
\end{equation}
To see this, recall that $\cA_n$ is the fixed point algebra of the endomorphism $\alpha_{n+1}$ and that
\[
\alpha_{n+1}\pi(\gamma_i)
=
\begin{cases}
\pi(\gamma_i) & \text{if $i \le n$,}\\
\pi(\gamma_{i+1}) & \text{if $i > n$}.
\end{cases}
\]
Thus $\pi(s_q) \in \vN(\alpha_{n+1}^{i-n-1}\pi(\Sset_{n+2})\mid i \in I_q, i > n)$. Since the sets $I_q$ are mutually disjoint, we conclude the factorization 
\eqref{eq:fixed-points-0} from the full $\cA_n$-independence of the sequence \eqref{eq:fixed-points-1}. 

At this point we have proven that $\cA_n$ is generated as a von Neumann algebra by the the limit $2$-cycle $A_0$, the limit $k$-cycles $C_k$ and $\vN_\pi(\Sset_{n+1})$. If $n=-1$, then we infer from Thoma multiplicativity or Proposition \ref{prop:limit-cycles} that $\cA_{-1}= \cZ(\cA)= \vN(C_k\mid k \ge 1)$. It remains to show that $\cA_0 = \vN(A_0)$. 
Since $A_0 \in \cA_0$ this ensures $\cA_0 = \vN(A_0)$ and $\cA_n = \cA_0 \vee \vN_\pi(\Sset_{n+1})$. Finally, $\cA_0$ is abelian since $C_k \in \cZ(\cA)$ and thus commutes with the limit cycle $A_0$.  
\end{proof}

Clearly every $k$-cycle is a limit $k$-cycle for $n$ sufficiently large. Conversely, limit $k$-cycles are certain weak limits of $k$-cycles. This will be immediate from conditional independence and strong mixing of $\alpha_n$ (over $\cA_{n-1}$) in the weak operator topology. On the other hand, limit $k$-cycles are certain limits of C\'esaro means of $k$-cycles, now in the strong operator topology. Again, this will be immediate from conditional independence and mean ergodic averages with respect to $\alpha_n$.  Let us illustrate this for the limit $2$-cycles $A_i$. 
\begin{Lemma} \label{lem:average-2}
Let $0 \le i \le j < k$. Then 
\begin{equation}\label{eq:average-1}
A_i = E_{j}(v_i v_k v_i) = v_i A_0 v_i = \alpha_0^i(A_0). 
\end{equation}
Moreover,
\begin{eqnarray}\label{eq:average-2} 
A_i &=& \wotlim_{n \to \infty} v_i v_n v_i  
     = \wotlim_{n \to \infty} \pi\big((i,n)\big)
\end{eqnarray} 
and
\begin{eqnarray}\label{eq:average-3} 
A_i = \sotlim_{n \to \infty} \frac{1}{n} \sum_{j=1}^n v_i v_j v_i
    = \sotlim_{n \to \infty} \frac{1}{n} \sum_{\substack{j=1, \\ j \neq i}}^n \pi\big((i,j)\big).
\end{eqnarray}
\end{Lemma}
\begin{proof}
\eqref{eq:average-1} has already be shown. \eqref{eq:average-2} 
The endomorphism $\alpha_1$ is strongly mixing over its fixed point algebra $\cA_0$ (see \cite[Theorem 6.4]{Koes10a}:
\[
\wotlim_{n\to \infty} \alpha_1^n(x) = E_0(x), \qquad x \in \cA.
\]
Thus, with limits in the sense of the weak operator topology,
\begin{eqnarray*}
A_i &=& v_i E_j(v_k)v_i = v_i E_0(v_1)v_i\\
    &=& \lim_{n \to \infty} v_i \alpha_1^n(v_1) v_i
    = \lim_{n \to \infty} v_i v_n v_i\\
    &=& \lim_{n \to \infty} \pi(\gamma_i \gamma_n\gamma_i)\\
    &=& \lim_{n \to \infty} \pi((i,n)).
\end{eqnarray*}
Alternatively the limit 2-cycle $A_i$ is obtained as the mean ergodic average or limit of C\'esaro means
\[
A_i = \sotlim_{n \to \infty} \frac{1}{n} \sum_{k=0}^{n-1} v_i \alpha_1^k(v_1) v_i
    = \sotlim_{n \to \infty} \frac{1}{n} \sum_{k=1}^{n} \pi((i, k)).
\]
\end{proof}

\begin{Remark}\normalfont \label{rem:average-k}
Similar as for limit 2-cycles, a limit $k$-cycle $C_k$ ($k \ge 1$) can be understood as  multiple mean ergodic averages or as limits of multiple C\'esaro means of $k$-cycles in the strong operator topology:
\[
C_k = \lim_{n_1,\ldots, n_k \to \infty} \frac{1}{n_1 \cdots n_k}
\sum_{i_1=1}^{n_1}\cdots \sum_{i_k=1}^{n_k} v_{i_1} v_{i_2} \cdots v_{i_k}v_{i_1}.
\]
Some terms in these multiple sums are not $k$-cycles. This happens precisely when at least two of the indices $i_1, i_2, \ldots, i_k$ are the same. But an elementary counting argument shows that these terms do not contribute in the limit. Similarly, the limit $k$-cycle $C_k$ can be obtained as a multiple limit of a $k$-cycle in the weak operator topology:
\[
C_k = \lim_{n_1,\ldots, n_k \to \infty} v_{n_1}v_{n_2} \cdots v_{n_k} v_{n_1}.
\]  
For a proof start with Proposition \ref{prop:limit-cycles} and argue similar as in Lemma \ref{lem:average-2}. Moreover we can obtain the limit $k$-cycles $C_k$ as weak limits, or as mean ergodic averages resp.~C\'esaro means, from the limit $2$-cycles $A_i$. More precisely, for $k \in \Nset$, 
\begin{eqnarray*}
C_k &=& \wotlim_{i \to \infty} A_i^{k-1},\\
C_k &=& \sotlim_{n \to \infty} \frac{1}{n} \sum_{i=0}^{n-1} A_i^{k-1}.
\end{eqnarray*}
These equations are evident from $\alpha_0(A_i)= A_{i+1}$, the fixed point characterization $\cA^{\alpha_0} =\cA_{-1}$ and $C_k= E_{-1}(A_0^{k-1})$.  
\end{Remark}

We collect further fundamental properties of limit 2-cycles. Some of them should be compared with the properties of `random cycles' derived from Olshanski semigroups,  especially those from \cite[Proposition 1 in Section 1] {Okou99a}. 
\begin{Proposition}\label{prop:A-prop}
The limit 2-cycles $A_i$ enjoy the following properties:
\begin{align}
\label{eq:A-item1}
\rho_0^{}(\sigma)(A_i) &= A_{\sigma(i)}
&& \text{for all $i \in \Nset_0$ and $\sigma \in \Sset_\infty$};\\ 
\label{eq:A-item2}
A_i A_j &= A_j A_i 
&& \text{for all $i,j \in \Nset_0$};\\
\label{eq:A-item3}
A_i v_j &= v_j A_i  
&& \text{if $0 < i \neq j$}; \\ 
\label{eq:A-item4}
E_{0} (A_i^k) &=  E_{-1}(A_0^k) = C_{k+1}
&&\text{if $i \neq 0$ and $k \in \Nset$};\\ 
\label{eq:A-item5}
E_{-1}\Big(\prod_{r=1}^p A_{i_r}^{k_r}\Big)
&= \prod_{r=1}^p E_{-1}(A_{0}^{k_r}) 
&&  \text{whenever $i_q \neq i_r$ for $q \ne r $ and $k_1, k_2, \ldots,k_p \in \Nset$.} \\
\label{eq:A-item6}
E_{0}\Big(\prod_{r=1}^p A_{i_r}^{k_r}\Big)
&= \prod_{r=1}^p E_{0}(A_{i_r}^{k_r}) 
&&  \text{whenever $i_q \neq i_r$ for $q \ne r $ and $k_1, k_2, \ldots,k_p \in \Nset$.} 
\end{align}
\end{Proposition}
\begin{proof}
Ad \eqref{eq:A-item1}:
Clearly $\sigma\big( (i,n) \big) =  \big(\sigma(i),\sigma(n)\big)$ for any 2-cycle $(i,n)$ and $\sigma \in \Sset_\infty$. Now Lemma \ref{lem:average-2} implies
\begin{eqnarray*}
\rho_0^{}(\sigma)(A_i)
&=& \wotlim_{n \to \infty}\rho_0^{}(\sigma)\Big(\pi\big((i,n)\big)\Big) \\
&=& \wotlim_{n \to \infty}\rho_0^{}(\sigma)\Big(\pi\big(\sigma(i),\sigma(n)\big)\Big)\\
&=& \wotlim_{n \to \infty}\rho_0^{}(\sigma)\Big(\pi\big(\sigma(i),n \big)\Big)\\ 
&=& A_{\sigma(i)}.
\end{eqnarray*}

Ad \eqref{eq:A-item2}: 
Since disjoint cycles commute and by Lemma \ref{lem:average-2}, 
\begin{eqnarray*}
A_i A_j
&=& \lim_{N \to \infty}   A_i \,v_j v_N v_j\\
&=& \lim_{N\to \infty} \lim_{M\to \infty}  v_i v_M v_i\, v_j v_N v_j\\
&=& \lim_{N\to \infty} \lim_{M\to \infty}  v_j v_N v_j\, v_i v_M v_i\\
&=& \lim_{N\to \infty}   v_j v_N v_j\, A_i\\
&=& A_j A_i,
\end{eqnarray*}
with all limits in the sense of the weak operator topology. 

Ad \eqref{eq:A-item3}: 
Approximate as before $A_i$ as the weak limit of $(v_i v_n v_i)$ for $n \to \infty$. Since $v_j$ is a cycle disjoint from $v_iv_nv_i$ for sufficiently large $n$, it follows $A_i v_j = v_j A_i$ whenever $0 < i\neq j$. 

Ad \eqref{eq:A-item4}: 
This is a direct consequence of the $\cA_{-1}$-independence of the algebras $\alpha_0^k(\cA_0)$:
\begin{align*}
E_{0}(A_i^k) 
&=  E_0\big(v_i A_0^k v_i\big)  
&& \text{(by \eqref{eq:2limit})}\\      
&= E_0 \alpha_0^i (A_0^k)          
&&\text{(by \eqref{eq:PR'-exchangeable} 
from Theorem \ref{thm:exchangeability})}\\
&= E_{-1}\alpha_0^i (A_0^k) = E_{-1}(A_0^k)
&&\text{(by Theorem \ref{thm:indy} \eqref{item:indy-i})}. 
\end{align*}

Ad \eqref{eq:A-item5}: 
Use $A_{i} = v_{i} A_0 v_{i} = \alpha_0^{i}(A_0)$. The claimed formula is now immediate from the full $\cA_{-1}$-independence of the sequence $(\alpha_0^k(\cA_0))$. 

Ad \eqref{eq:A-item6}:
Since $A_i$ and $A_j$ commute by \eqref{eq:A-item2}, we can assume with out loss of generality that $i_1 < i_2 < \ldots < i_p$. We know from Theorem \ref{thm:endo-braid-i} that $\bigvee_{i \in I} \alpha_0^i(\cA_0)$ and
$\bigvee_{j \in J} \alpha_0^j(\cA_0)$ are $\cA_{-1}$-independent for disjoint subsets $I$ and $J$ of $\Nset_0$. If $i_1\neq 0$, this independence implies 
\begin{eqnarray*}
E_{0}(A_{i_1}^{k_1} A_{i_2}^{k_2} \cdots A_{i_p}^{k_p})
&=& E_{-1}(A_{i_1}^{k_1} A_{i_2}^{k_2} \cdots A_{i_p}^{k_p})\\
&=& E_{-1}(A_{0}^{k_1}) E_{-1}( A_{0}^{k_2}) \cdots E_{-1}( A_{0}^{k_p})\\
&=& E_{0}(A_{0}^{k_1}) E_{0}( A_{0}^{k_2}) \cdots E_{0}( A_{0}^{k_p}).
\end{eqnarray*}
Now \eqref{eq:A-item4} completes the proof for the case $i_1 \neq 0$.
In the case $i_1=0$ we first apply the module property of conditional expectations so that
\[
E_{0}(A_{i_1}^{k_1} A_{i_2}^{k_2} \cdots A_{i_p}^{k_p})
= A_{i_1}^{k_1} E_{0}( A_{i_2}^{k_2} \cdots A_{i_p}^{k_p})
\]
Here the first factor $A_{i_1}^{k_1}$ equals of course $E_{0 }(A_{i_1}^{k_1})$, since $i_1=0$. Now $i_2 \neq 0$ and the remaining factorization is done as before. 
\end{proof}
We continue with a generalization of Thoma multiplicativity from Theorem \ref{thm:thoma-mult} to higher fixed point algebras than $\cA_{-1}$. We need to provide some additional notation. 

Let $n \ge -1$. The cycle $\partial_n(s)$ of the 
the non-trivial $k$-cycle $s:= \gamma_{n_1}\gamma_{n_2}\cdots\gamma_{n_k} \gamma_{n_1} \in \Sset_\infty$ with $n_1= \min\{n_1, n_2, \ldots, n_k\}$ is given by
\[
\partial_n(s) := 
\begin{cases}
\gamma_0 & \text{if $n_1 > n$}\\
\gamma_{n_1}\gamma_{\chi(n_2)} \gamma_{\chi(n_3)}\cdots\gamma_{\chi(n_k)}\gamma_{n_1} 
& \text{if $n_1 \le n$},
\end{cases} 
\]
where here $\chi$ is the characteristic function on $\Nset_0$ for the set $[n]:= \{0,1,2,\ldots,n\}$. We put $\partial_n(\gamma_0):= \gamma_0$. If $\sigma \in \Sset_\infty$ has the disjoint cycle decomposition $\sigma = s_1 s_2 \cdots s_l$, then $\partial_n(\sigma)$ is introduced as the multiplicative extension
\[
\partial_n(\sigma):= \partial_n(s_1) \partial_n(s_2) \cdots \partial_n(s_l).
\]
More intuitively, $\partial_n$ removes all points larger than $n$ from each non-trivial cycle in $\sigma$.
\begin{Definition}\normalfont
$\partial_n(\sigma)$ is  called the \emph{$n$-derivative} of the permutation $\sigma \in \Sset_\infty$.     
\end{Definition}

Further we put, for $n\ge -1$ and $k \in \Nset_0$, 
\[
\ell_{n,k}(\sigma) := 
\begin{cases}
\min\set{p\in \Nset_0}{\sigma^{-(p+1) }(k)\le n}
& \text{if $k \le n$},\\
0 
&\text{if $k >n$}.
\end{cases}
\]
For a permutation $\sigma \in \Sset_\infty$, the number $\ell_{n,k}(\sigma)$ represents the \emph{length of the excursion} (into the set $\{n+1, n+2, \ldots\}$) of a point $k \in [n]$ under the action of $\sigma^{-1}$. Let us illustrate this in the example  
\[
n=6,\qquad \sigma = (1, 8, 7, 4, 10, 5).
\]
If $k=4$, then $\sigma^{-1}(4) = 7$,  $\sigma^{-2}(4)= 8$ and $\sigma^{-3}(4) =1$. Thus the length of the excursion is $\ell_{6,4}(\sigma) = 2$. If $k=1$, then $\sigma^{-1}(1)=5$ and thus $\ell_{6,2}(\sigma) = 0$. Note also that fixed points yield zero excursion length, for example:  $\ell_{6,9}(\sigma) = 0$ and $\ell_{6,3}(\sigma) = 0$. 

Finally, denote by $\Nset_0\slash \langle \sigma \rangle$ the set of all orbits of $\Nset_0$ under the action of the subgroup $\langle \sigma \rangle$ generated by the permutation $\sigma \in \Sset_\infty$. The set of orbits $\Nset_0\slash \langle \sigma \rangle$ forms a partition $\{V_1, V_2,\ldots\}$  of $\Nset_0$. The cardinality of a block $V_i$ of this partition will be denoted by $|V_i|$. Note that only finitely many blocks $V_i$ have a cardinality larger than 1. We make use of the convention $0^0=1$.
\begin{Theorem}[Generalized Thoma multiplicativity]\label{thm:thoma-mult-general}
Let $n \ge -1$ and $\sigma \in \Sset_\infty$. Then
\begin{eqnarray}
\label{eq:thoma-general-1}
E_n\big(\pi(\sigma)\big) 
&=& 
\pi\big(\partial_n(\sigma)\big)     
\Big(
\prod_{\substack{V \in \Nset_0\slash\langle \sigma \rangle \\ \min V > n}}
   C_{|V|} \Big) \,\,
\Big(
\prod_{\substack{V \in \Nset_0\slash\langle \sigma \rangle \\ \min V  \le n}}\,\,
\prod_{k \in V} A_k^{\ell_{n,k}(\sigma)}\Big)\\ 
\label{eq:thoma-general-2}
&=&
\Big(
\prod_{\substack{V \in \Nset_0\slash\langle \sigma \rangle \\ \min V > n}}
   C_{|V|} \Big) \,\,
\Big(
\prod_{\substack{V \in \Nset_0\slash\langle \sigma \rangle \\ \min V  \le n}}\,\,
\prod_{k \in V} A_k^{\ell_{n,k}(\sigma^{-1})}\Big) \,\,\,
\pi\big(\partial_n(\sigma)\big),  
\end{eqnarray}
and
\[
E_n\big(\pi(s_1s_2\cdots s_p)\big)= E_n\big(\pi(s_1))\, E_n\big(\pi(s_2)\big)  \cdots E_n\big(\pi(s_p)) 
\]
for disjoint cycles $s_1, s_2, \ldots, s_p \in \Sset_\infty$. 
\end{Theorem}
If $\sigma\in \Sset_{n+2}$ then the formula above simplifies to
\begin{equation}\label{eq:thoma-mult-gen}
E_n\big(\pi(\sigma)\big) =
\begin{cases}
\pi(\sigma) & \text{if $\sigma(n+1)=n+1$},\\
\pi\big(\partial_n (\sigma)\big) A_{\sigma(n+1)} =
 A_{\sigma^{-1}(n+1)} \,\pi\big(\partial_n (\sigma)\big)& \text{if $\sigma(n+1)\neq n+1$}.
\end{cases}
\end{equation}
Of course this can be verified more directly starting from \eqref{eq:thoma-mult-gen} and Proposition \ref{prop:limit-cycles}, using the intertwining properties $A_0 v_{\sigma(n+1)}= v_{\sigma(n+1)} A_{\sigma(n+1)}$ and  $v_{\sigma^{-1}(n+1)} A_0 =  A_{\sigma^{-1}(n+1)} v_{\sigma^{-1}(n+1)}$    from \eqref{eq:2limit} if the point $n+1$ is not fixed under $\sigma \in \Sset_{n+2}$. Note that $\pi\big(\partial_n (\sigma)\big)$ and $A_{\sigma(n+1)}$ do not commute whenever the limit $2$-cycle $A_{\sigma(n+1)}$ is non-central (see Remark \ref{rem:central-lim-cycles}).   

We infer also from Theorem \ref{thm:thoma-mult-general} that all elements of the form   
\[
\pi(\sigma) \prod_{l=1}^p C_{k_l}  \prod_{j=1}^{q} A_{i_j} \qquad(\sigma \in \Sset_{n+1}; k_1,\ldots, k_p \in \Nset; i_1, \ldots, i_q \in \Nset_0; p,q \in \Nset_0)
\]
give a weak*-total set in $\cA_n$, with the convention $\prod_{r=1}^{0}X_r= 1$.    
\begin{proof}
We consider first the case of a non-trivial $k$-cycle $\sigma=(n_1, n_2, \ldots, n_k)$ with $n_1 = \min\{n_1,\ldots,n_k\}$. If $n_1 > n \ge -1$ then
$\partial_n(\sigma)= \sigma_0$, since all points are removed from the cycle. 
Moreover
$
E_n\big(\pi(\sigma)\big)= C_k
$
by Proposition \ref{prop:limit-cycles}. Since $\sigma$ has precisely one orbit with cardinality $k$ and $C_1=1$ we have verified the product of the $C_{|V|}$'s. By the usual convention, the products inside the second parenthesis are indexed by the empty set and thus are one. Thus \eqref{eq:thoma-general-1} is verified in this case.

We turn our attention to the case $n \ge n_1 \ge 0$ of \eqref{eq:thoma-general-1}. Again by Proposition \ref{prop:limit-cycles}, 
\begin{align*}
E_{n}\big(\pi(\sigma)\big) = 
w_{n_1} w_{n_2} \cdots w_{n_k} w_{n_1}  
\qquad
\text{with} 
\quad
w_{n_i} = 
\begin{cases}
v_{n_i} & \text{if $n_i \le n$},\\
A_0     & \text{if $n_i > n $}.
\end{cases}
\end{align*}
Consider next some fixed $n_{i_0} \le n$ with $n_{{i_0}-1}>n$ (where the subscript $i_0$ is understood to be $\operatorname{mod} k$). Thus we have
\[
E_n\big(\pi(\sigma)\big) = w_{n_1}w_{n_2}\cdots w_{n_{{i_0}-2}} A_0 v_{n_{i_0}} \cdots w_{n_k}w_{n_1}.
\] 
Since $A_0 v_{n_{i_0}} = v_{n_{i_0}}A_{n_{i_0}}$ (by \eqref{eq:average-1}) and $A_{n_{i_0}}w_j = w_j A_{n_{i_0}}$ for all $j \neq i_0$ (by \eqref{eq:A-item3}), the factor $A_{n_{i_0}}$ can be moved to the right,
\[
E_n\big(\pi(\sigma)\big) = w_{n_1}w_{n_2}\cdots w_{n_{i-2}} v_{n_i} \cdots w_{n_k}w_{n_1} \, A_{n_{i_0}}.
\]
We iterate this procedure for $n_{i_0}$ until the next factor $w_{n_{j_0}}$ with $n_{j_0} \le n$, i.e.~$w_{n_{j_0}}= v_{n_{j_0}}$, appears as the predecessor of the factor $v_{n_{i_0}}$. At this point we have pulled out the factor $A_{n_{i_{0}}}^{i_{0}-j_{0}-1}$, where the number $(i_{0}-j_{0}-1)$ equals $\ell_{n,n_{i_0}}(\sigma)$, the length of the past excursion of the point $n_{i_0} \in [n]$. We repeat this procedure until all factors $w_i$ of the form $A_0$ are pulled out to the right. Clearly, after removing all these factors we are left with the $n$-derivative $\pi\big(\partial_n(\sigma)\big)$. Note also that we can always reorder products $A_{p}^{\ell_{n,p}(\sigma)} A_{q}^{\ell_{n,q}(\sigma)}$  on the right, since $A_q A_p = A_p A_q$ for $p\neq q$ by \eqref{eq:A-item2}, and so the form of the factor inside the right parenthesis of \eqref{eq:thoma-general-1} is easily verified.   

We are left to verify \eqref{eq:thoma-general-1} for the general case 
of a disjoint product of non-trivial cycles $\sigma = s_1 s_2 \cdots s_p$.
Due to \eqref{eq:fixed-points-0} we have already at hands the factorization
\[
E_n(\pi(\sigma))= E_n(\pi(s_1)) \,  E_n(\pi(s_2)) \cdots E_n(\pi(s_p)).  
\]   
So we are left to show that all appearing factor can be arranged in the order stated in \eqref{eq:thoma-general-1}. But this is evident from $C_k\in \cZ(\cA)$ and \eqref{eq:A-item3}. The proof of \eqref{eq:thoma-general-2} is done in the same manner, now pulling out factors $A_0$ to the left instead of the right.
\end{proof}
\begin{Remark} \normalfont \label{rem:central-lim-cycles}
If the star generator $v_1$ is central, so are all star generators $v_{i+1}$ and limit $2$-cycles $A_i$ for $i \in \Nset_0$. Indeed, $\cZ(\cA) \subset \cA^{\alpha_1}$ and $\alpha_1(v_i) = v_{i+1}$ implies $v_1 = v_i \in \cZ(\cA)$ for all $i\in \Nset$. Thus $A_0 = \wotlim_{n \to \infty} v_n \in \cZ(\cA)$ and consequently $A_0 = v_i A_0 v_i = A_i \in \cZ(\cA)$ for all $i \in \Nset_0$. Altogether the centrality of $v_1$ implies $v_{i+1} = v_1 = A_0 = A_i$ for all $i \in \Nset_0$. Such a situation occurs in direct sums of the trivial and sign representation of $\Sset_\infty$. Note also for a non-central limit $2$-cycle $A_1$ that $\alpha_0(A_i) = A_{i+1}$ and $\cA^{\alpha_0} = \cZ(\cA)$ entails $A_i \neq A_j$ whenever $i \neq j$. We finally remark that the centrality of the $A_i$'s does not imply the centrality of the $v_j$'s. Such a situation occurs for Markov traces, where $A_0$ is trivial and $v_1$ non-central, see Section \ref{section:Markov}.        
\end{Remark}
\begin{Remark}\normalfont \label{rem:olshanski}
For the convenience of the reader, we briefly relate our approach to that of Okounkov in \cite{Okou99a} via the representation theory of Olshanski semigroups. Working in the GNS representation of $(\cA,\trace)$ with GNS Hilbert space $L^2(\cA,\trace)$ and the cyclic separating vector $\1_{L^2(\cA,\trace)}$, we obtain the quadruple (compare \cite[Section 1]{Okou99a}) 
\begin{equation}\label{eq:olshanski}
\Big\langle \pi(\Sset_\infty)J\pi(\Sset_\infty)J, A_0, C_j, e_k \Big\rangle_{j\ge 1, k \ge -1}.
\end{equation}
Here $J$ is the modular conjugation on $L^2(\cA,\trace)$ and $e_k$ is the extension of the conditional expectation $E_k$ with  $e_k x e_k = E_k(x) e_k$ for $x \in \cA$.  Note further that the representation $\rho_0^{}\colon \Sset_\infty \to \Aut{\cA,\trace}$ with $\rho_0^{}(\sigma) = \Ad \pi(\sigma)$  extends to the representation $U \colon \Sset_\infty \to \cU(L^2(\cA, \trace))$
such that
\[
U_\sigma = \pi(\sigma) J \pi(\sigma)J.
\] 
In other words: $U_\sigma$ is a unitary in the diagonal subgroup of  $\pi(\Sset_\infty) J \pi(\Sset_\infty) J$.   
Altogether, the quadruple \eqref{eq:olshanski} gives rise to a semigroup as it emerges from spherical representations on Hilbert spaces of certain Olshanski semigroups in \cite{Okou99a}.   
\end{Remark}
\section{Noncommutative Markov shifts and commuting squares from unitary representations of $\Sset_\infty$}
\label{section:bernoulli-markov}

Our analysis of representations of $\Sset_\infty$ presented so far is motivated by ideas from noncommutative probability. By making this more explicit the resulting structures become more transparent.
Hence in this section we have a closer look especially at the commuting squares obtained in the previous sections and we will discuss them with respect to noncommutative (unilateral) versions of Bernoulli shifts and Markov shifts. 
We start with some general concepts.

\begin{Definition} \normalfont \label{def:m-b}

Let $(\cM,\psi)$ be a probability space and $\cB_0$ a $\psi$-conditioned subalgebra of $\cM$. Suppose further that $\alpha$ is a $\psi$-preserving endomorphism of $\cM$ such that the subalgebras
\[
\cB_{[m,n]} := \vN(\alpha^k(\cB_0) \colon m \le k \le n)
\]
are also $\psi$-conditioned. We denote the $\psi$-preserving conditional expectation from $\cM$ onto $\cB_{[m,n]}$ by $\mathcal{E}_{[m,n]}$ and we define an endomorphism $\beta$ as the restriction of $\alpha$ to $\cB := \vN\set{\alpha^k(\cB_0}{k \in \Nset_0}$.

\begin{enumerate}
\item[(i)] \label{item:def-m-b-i}
The endomorphism $\beta$ is called a {\emph{Markov shift} with generator $\cB_0$ if the following \emph{Markov property} is valid:
\[
\mathcal{E}_{[0,n]}\;\beta^k(x) = \mathcal{E}_{[n,n]}\;\beta^k(x)
\]
for all $n,k \in \Nset_0$ with $n \le k$ and all $x \in \cB_0$.
In this case
\[
R := \mathcal{E}_{[0,0]}\; \beta\; \mathcal{E}_{[0,0]}
\]
is called the transition operator. 
\item[(ii)] \label{item:def-m-b-ii}
The endomorphism $\beta$ is called a \emph{(full/ordered) Bernoulli shift} over $\cN$ with generator $\cB_0$} if $\cN \subset \cB_0 \cap \cM^\alpha$ is a $\psi$-conditioned subalgebra and $\big(\alpha^n(\cB_0)\big)_{n \in \Nset_0}$ is (full/order) $\cN$-independent. Compare Definition \ref{def:independence}. 
\end{enumerate}
\end{Definition}

Note that it is sufficient to verify
the Markov property for $k=n+1$ and that for all $k \in \Nset_0$ and $x \in \cB_0$
we have
\[
\mathcal{E}_{[0,0]}\; \beta^k(x) = R^k(x).
\]
Further it is easy to check that a Bernoulli shift over $\cN$ is a Markov shift with transition operator $R = E_{\cN}$, the conditional expectation onto $\cN$. 
If $\cB$ is commutative then a Markov shift (Bernoulli shift) is called classical and all the concepts and properties above are very familiar from classical probability theory. More detailed discussions and surveys about the noncommutative generalizations used here can be found in \cite[Chapter 2]{Gohm04a} and \cite[Appendix]{GoKo09a}.
\\

Let us now go back to representations of $\Sset_\infty$, using the notations introduced in previous sections.
For convenience we put $C \equiv (C_k)_{k \in \Nset}$ and $\lara{(X_i)_{i\in I}}:= \vN(X_i| i \in I)$.  
\begin{Theorem}\label{thm:cs-star}
Suppose the tracial probability space $(\cA,\trace)$ is equipped with a representation $\pi\colon \Sset_\infty \to \cU(\cA)$ such that 
$\cA = \vN_\pi(\Sset_\infty)$ and consider the $1$-shifted representation 
\[
\rho_1 = \Ad \pi \circ \sh \colon  \Sset_\infty \to \Aut{\cA, \trace}. 
\]
We arrive at the following conclusions:
\begin{enumerate} 
\item \label{item:cs-star-i}
$ \rho_1$ has the generating property and, for all $n \in \Nset_0$,  
\begin{eqnarray*}
 \cA_{-1} &=& \vN(C) = \cZ(\cA),\\
\cA^{\rho_1(\Sset_\infty)} = \cA_{0} &= &\vN(C,A_0),\\
\cA^{\rho_1(\Sset_{n+2,\infty})} = \cA_{n+1}&= & \vN\big(C,A_0,v_0,v_1, \ldots, v_{n+1}\big).
\end{eqnarray*}
Moreover 
\[
\alpha_1(x) = \sotlim_{n \to \infty} \rho_1(\sigma_1 \sigma_2 \cdots \sigma_n)(x)
\] 
is a $\trace$-preserving endomorphism of $\cA$ with fixed point algebra $\cA_0$ such that, for $i,k \in \Nset$, 
\[
\alpha_1(C_{k}) =  C_{k}, \qquad 
\alpha_1(A_0) =  A_{0}, \qquad
\alpha_1(v_i) =  v_{i+1}.
\]
\item \label{item:cs-star-ii}
The sequence $(v_i)_{i \in \Nset}$ is minimal, exchangeable and has the tail algebra 
\[
\bigcap_{n \ge 0} \vN(v_k \mid k \ge n) = \cA_0 = \vN(C,A_0).
\] 
In particular this sequence is spreadable, stationary and full $\cA_0$-independent.   
\item \label{item:cs-star-iii}
$\alpha_1$ is a full Bernoulli shift over $\cA_0 = \vN(C, A_0)$ with generator $\cA_1= \vN(C, A_0,v_1)$.
\item \label{item:cs-star-iv} 
Each cell of the triangular tower of inclusions is a commuting square: 
\begin{eqnarray*}
\setcounter{MaxMatrixCols}{20}
\small
\begin{matrix} 
\cA_0 &\subset&
\cA_1 &\subset&
\cA_2 &\subset&
\cA_3  &\subset\cdots\subset& 
\cA_\infty\\
    \shortparallel
 && \shortparallel
 && \shortparallel
 && \shortparallel
 && \shortparallel\\
\lara{C,A_0} && 
\lara{C,A_0,v_1} &&   
\lara{C,A_0,v_1,v_2} && 
\lara{C,A_0,v_1,v_2,v_3} && 
\cA\\ 
 && \cup
 && \cup
 && \cup
 && \cup\\ 
           &&                 
\lara{C,A_0} &\subset& 
\lara{C,A_0,v_2} &\subset&   
\lara{C,A_0,v_2,v_3} &\subset\cdots\subset& 
\alpha_1 (\cA)\\
 && 
 && \cup
 && \cup
 && \cup\\
           && 
           &&           
\lara{C,A_0} &\subset& 
\lara{C,A_0,v_3} &\subset\cdots\subset& 
\alpha_1^2 (\cA)\\
 && 
 && 
 && \cup
 && \cup\\
 && 
 && 
 && \vdots
 && \vdots\\
\end{matrix}
\end{eqnarray*}
\setcounter{MaxMatrixCols}{10}
\item \label{item:cs-star-v} 
The following are equivalent:
\begin{enumerate}
\item  \label{item:cs-star-v-a} 
The limit $2$-cycle $A_0$ is central.
\item  \label{item:cs-star-v-b} 
The sequence $(v_i)_{i \in \Nset}$ has the tail   
algebra $\vN(C) = \cZ(\cA)$.
\item \label{item:cs-star-v-c} 
The triangular tower of commuting squares from \eqref{item:cs-star-iv} equals 
\begin{eqnarray*}
\setcounter{MaxMatrixCols}{20}
\small
\begin{matrix} 
\lara{C} &\subset& 
\lara{C,v_1} &\subset&   
\lara{C,v_1,v_2} &\subset& 
\lara{C,v_1,v_2,v_3} &\subset\cdots\subset& 
\cA\\ 
 && \cup
 && \cup
 && \cup
 && \cup\\ 
           &&                 
\lara{C} &\subset& 
\lara{C,v_2} &\subset&   
\lara{C,v_2,v_3} &\subset\cdots\subset& 
\alpha_1 (\cA)\\
 && 
 && \cup
 && \cup
 && \cup\\
           && 
           &&           
\lara{C} &\subset& 
\lara{C,v_3} &\subset\cdots\subset& 
\alpha_1^2 (\cA)\\
 && 
 && 
 && \cup
 && \cup\\
 && 
 && 
 && \vdots
 && \vdots\\
\end{matrix}
\end{eqnarray*}
\setcounter{MaxMatrixCols}{10}
\end{enumerate}
\item \label{item:cs-star-vi} 
The following are equivalent:
\begin{enumerate}
\item \label{item:cs-star-vi-a} 
The limit $2$-cycle $A_0$ is trivial.
\item \label{item:cs-star-vi-b} 
The sequence $(v_i)_{i \in \Nset}$ has the tail algebra $\Cset$.  
\item \label{item:cs-star-vi-c} 
The triangular tower of commuting squares from \eqref{item:cs-star-iv} equals: 
\begin{eqnarray*}
\small
\setcounter{MaxMatrixCols}{20}
\begin{matrix} 
\Cset &\subset& 
\lara{v_1} &\subset&   
\lara{v_1,v_2} &\subset& 
\lara{v_1,v_2,v_3} &\subset\cdots\subset& 
\cA\\ 
 && \cup
 && \cup
 && \cup
 && \cup\\
           &&                 
\Cset &\subset& 
\lara{v_2} &\subset&   
\lara{v_2,v_3} &\subset\cdots\subset& 
\alpha_1 (\cA)\\
 && 
 && \cup
 && \cup
 && \cup\\
           && 
           &&           
\Cset &\subset& 
\lara{v_3} &\subset\cdots\subset& 
\alpha_1^2 (\cA)\\
 && 
 && 
 && \cup
 && \cup\\
 && 
 && 
 && \vdots
 && \vdots\\
\end{matrix}.
\end{eqnarray*}
\setcounter{MaxMatrixCols}{10}
\item \label{item:cs-star-vi-d} 
$\alpha_1(\cA) \subset \cA$ is an irreducible inclusion of subfactors.
\end{enumerate}
\end{enumerate}
\end{Theorem}
\begin{proof}
\eqref{item:cs-star-i} is the content of Lemma \ref{lem:inclusions}, Proposition \ref{prop:trace} and Theorem \ref{thm:fixed-points}. 
For \eqref{item:cs-star-ii} see Theorem \ref{thm:indy}\eqref{item:indy-ii} and use the identification of the tail algebra $\cA_0$ from \eqref{item:cs-coxeter-i} above. Finally the properties stated additionally for $(v_i)_{i\in \Nset}$ follow from the extended de Finetti theorem, Theorem \ref{thm:de-finetti}.
Note for \eqref{item:cs-star-iii} that $\vN(C,A_0,v_1)$ is a generator for $\alpha_1$ and that \eqref{item:cs-star-ii} gives the full $\cA_0$-independence of the sequence of subalgebras $\big(\vN(C,A_0,v_i)\big)_{i \in \Nset}$. 
\eqref{item:cs-star-iv} follows from Corollary \ref{cor:endo-braid-i}.
All equivalences in \eqref{item:cs-star-v} are immediate from
\eqref{item:cs-star-i}, \eqref{item:cs-star-ii}, \eqref{item:cs-star-iv}   
and the fact that $\vN(C,A_0) = \vN(C)$ if and only if $A_0$ is central. 
Similarly, all equivalences in \eqref{item:cs-star-vi} follow from the fact
that $\vN(C,A_0) \simeq \Cset$ if and only if $A_0$ is trivial (see Corollary \ref{cor:fixed-points} \eqref{item:cor-fixed-points-ii}).    
\end{proof}
Theorem \ref{thm:cs-star} is formulated in the framework of $1$-shifted representations $\rho_1 = \rho_0^{} \circ \sh$ and its parts \eqref{item:cs-star-i} to \eqref{item:cs-star-iv} transfer appropriately to the more general setting of an $N$-shifted representation $\rho_N = \rho_0^{} \circ \sh^N$ with $N > 1$ (see also Remark \ref{rem:b-m-shifted}). 
We continue with the counterpart of Theorem \ref{thm:cs-star} for the $0$-shifted representation $\rho_0$.  
\begin{Theorem}\label{thm:cs-coxeter}
Suppose the tracial probability space $(\cA,\trace)$ is equipped with a representation $\pi\colon \Sset_\infty \to \cU(\cA)$ such that 
$\cA = \vN_\pi(\Sset_\infty)$. Consider the $0$-shifted representation 
\[
\rho_0^{} = \Ad \pi \colon  \Sset_\infty \to \Aut{\cA, \trace}. 
\]
We arrive at the following conclusions:
\begin{enumerate}
\item \label{item:cs-coxeter-i}
$\rho_0^{}$ has the generating property and, for all $n \in \Nset_0$,  
\begin{eqnarray*}
\cA^{\rho_0^{}(\Sset_\infty)} &=& \cA_{-1} = \vN(C) = \cZ(\cA),  \\ 
\cA^{\rho_0^{}(\Sset_{n+2,\infty})} 
&=& \cA_n = \vN\big(C,A_i,u_0,u_1, \ldots, u_n\big) \qquad (0 \le i \le n).
\end{eqnarray*}
Moreover 
\[
\alpha_0(x) = \sotlim_{n \to \infty} \rho_0^{}(\sigma_1 \sigma_2 \cdots \sigma_n)(x)
\] 
is a $\trace$-preserving endomorphism of $\cA$ with fixed point algebra $\cA_{-1}$ such that, for $i,k \in \Nset$, 
\[
\alpha_0(C_k) =  C_k, \qquad 
\alpha_0(A_i) =  A_{i+1}, \qquad
\alpha_0(u_i) =  u_{i+1}.
\]
\item \label{item:cs-coxeter-ii} 
The sequence $(u_i)_{i \in \Nset}$ is minimal, stationary and has the tail algebra 
\[
\bigcap_{n \ge 0} \vN(u_k \mid k \ge n) = \cA_{-1} = \vN(C).
\] 
This sequence may be neither (order/full) $\cA_{-1}$-independent, spreadable nor exchangeable.
\item \label{item:cs-coxeter-iii}
$\alpha_0$ is a Markov shift with generator $\cA_1 = \vN(C,A_0, u_1)$ 
and transition operator 
\[
R_0 := E_1 \alpha_0 E_1 = \alpha_0 E_0.  
\] 
More explicitly, the transition operator $R_0$ is given by
\begin{eqnarray*}
R_0\Big(C_{k_1}^{n_1}\cdots C_{k_p}^{n_p} A_0^{n} u_\epsilon\Big)
&=& C_{k_1}^{n_1}\cdots C_{k_p}^{n_p} A_1^{n+\epsilon} \qquad (\epsilon = 0,1)
\end{eqnarray*}
for  $k_1,\ldots, k_p \ge 1$ and $n_1, \ldots, n_p,n \ge 0$, with $p \in \Nset$. 
\item  \label{item:cs-coxeter-iv}
Each cell of the triangular tower of inclusions is a commuting square: 
\begin{eqnarray*}
\setcounter{MaxMatrixCols}{20}
\small
\begin{matrix}
\cA_{-1} &\subset& 
\cA_0 &\subset& 
\cA_1 &\subset&   
\cA_2 &\subset& 
\cA_3 &\subset\cdots\subset& 
\cA_\infty\\ 
    \shortparallel
 && \shortparallel
 && \shortparallel
 && \shortparallel
 && \shortparallel
 && \shortparallel\\
\lara{C} && 
\lara{C,A_0} && 
\lara{C,A_0,u_1} &&   
\lara{C,A_0,u_1,u_2} && 
\lara{C,A_0,u_1,u_2,u_3} && 
\cA\\ 
 && \cup
 && \cup
 && \cup
 && \cup
 && \cup\\
           && 
\lara{C} &\subset&                
\lara{C,A_1} &\subset& 
\lara{C,A_1,u_2} &\subset&   
\lara{C,A_1,u_2,u_3} &\subset\cdots\subset& 
\alpha_0 (\cA)\\
 && 
 && \cup
 && \cup
 && \cup
 && \cup\\
           && 
           &&
\lara{C} &\subset&           
\lara{C,A_2} &\subset& 
\lara{C,A_2,u_3} &\subset\cdots\subset& 
\alpha_0^2  (\cA)\\
 && 
 && 
 && \cup
 && \cup
 && \cup\\
 && 
 && 
 && \vdots
 && \vdots
 && \vdots\\
\end{matrix}
\end{eqnarray*}
\setcounter{MaxMatrixCols}{10}
\item \label{item:cs-coxeter-v}
The following are equivalent:
\begin{enumerate}
\item \label{item:cs-coxeter-v-a}
The limit $2$-cycle $A_0$ is central. 
\item \label{item:cs-coxeter-v-b}
$E_{-1}$ is a center-valued Markov trace (compare Section \ref{section:Markov}),
i.e.
\[
E_{-1} (x\,u_n) = E_{-1} (x)\, E_{-1} (u_n)
\]
for all $n \in \Nset$ and $x \in \vN_{\pi}(\Sset_n)$.
\item \label{item:cs-coxeter-v-c}
The sequence $(u_i)_{i \in \Nset}$ is full $\cZ(\cA)$-independent. 
\item \label{item:cs-coxeter-v-d}
The Markov shift $\alpha_0$ 
is a full Bernoulli shift over the center $\cZ(\cA) = \cA_{-1}$. In particular, the transition operator $R_0$ reduces to
\[
R_0 = \alpha_0 E_0 = E_{-1}.
\]
\item \label{item:cs-coxeter-v-e}
The triangular tower of commuting squares from \eqref{item:cs-coxeter-iv} equals
\begin{eqnarray*}
\setcounter{MaxMatrixCols}{20}
\small
\begin{matrix}
\lara{C} &\subset& 
\lara{C} &\subset& 
\lara{C,u_1} &\subset&   
\lara{C,u_1,u_2} &\subset& 
\lara{C,u_1,u_2,u_3} &\subset\cdots\subset& 
\cA\\ 
 && \cup
 && \cup
 && \cup
 && \cup
 && \cup\\
           && 
\lara{C} &\subset&                
\lara{C} &\subset& 
\lara{C,u_2} &\subset&   
\lara{C,u_2,u_3} &\subset\cdots\subset& 
\alpha_0 (\cA)\\
 && 
 && \cup
 && \cup
 && \cup
 && \cup\\
           && 
           &&
\lara{C} &\subset&           
\lara{C} &\subset& 
\lara{C,u_3} &\subset\cdots\subset& 
\alpha_0^2  (\cA)\\
 && 
 && 
 && \cup
 && \cup
 && \cup\\
 && 
 && 
 && \vdots
 && \vdots
 && \vdots\\
\end{matrix}
\end{eqnarray*}
\setcounter{MaxMatrixCols}{10} 
\end{enumerate}
\item \label{item:cs-coxeter-vi}
The following are equivalent:
\begin{enumerate}
\item \label{item:cs-coxeter-vi-a}
The limit 2-cycle $A_0$ is trivial.
\item \label{item:cs-coxeter-vi-b}
$\trace$ is a Markov trace (compare Section \ref{section:Markov}),
i.e.
\[
\trace (x\,u_n) = \trace (x)\, \trace (u_n)
\]
for all $n \in \Nset$ and $x \in \vN_{\pi}(\Sset_n)$.
\item \label{item:cs-coxeter-vi-c}
The sequence $(u_i)_{i \in \Nset}$ is full $\Cset$-independent.  
\item \label{item:cs-coxeter-vi-d}
The Markov shift $\alpha_0$ is a full Bernoulli shift over $\Cset$. 
\item \label{item:cs-coxeter-vi-e}
The triangular tower of commuting squares from \eqref{item:cs-coxeter-iv} equals 
\begin{eqnarray*}
\setcounter{MaxMatrixCols}{20}
\small
\begin{matrix}
\Cset &\subset& 
\Cset &\subset& 
\lara{u_1} &\subset&   
\lara{u_1,u_2} &\subset& 
\lara{u_1,u_2,u_3} &\subset\cdots\subset& 
\cA\\ 
 && \cup
 && \cup
 && \cup
 && \cup
 && \cup\\
           && 
\Cset &\subset&                
\Cset &\subset& 
\lara{u_2} &\subset&   
\lara{u_2,u_3} &\subset\cdots\subset& 
\alpha_0 (\cA)\\
 && 
 && \cup
 && \cup
 && \cup
 && \cup\\
           && 
           &&
\Cset &\subset&           
\Cset &\subset& 
\lara{u_3} &\subset\cdots\subset& 
\alpha_0^2 (\cA)\\
 && 
 && 
 && \cup
 && \cup
 && \cup\\
 && 
 && 
 && \vdots
 && \vdots
 && \vdots\\
\end{matrix}
\end{eqnarray*}
\setcounter{MaxMatrixCols}{10}
\item \label{item:cs-coxeter-vi-f}
$\alpha_0(\cA)\subset \cA $ is an irreducible inclusion of subfactors. 
\end{enumerate}
\item \label{item:cs-coxeter-vii}
The restrictrion $\beta$ of the Markov shift $\alpha_0$ to $\cB:=\vN(C,A_i\mid i \in \Nset_0)$ is a classical Bernoulli shift over $\cZ(\cA) = \vN(C)$. The triangular tower of commuting squares from \eqref{item:cs-coxeter-iv} restricts to 
\begin{eqnarray*}
\setcounter{MaxMatrixCols}{20}
\small
\begin{matrix}
\lara{C} &\subset& 
\lara{C,A_0} &\subset& 
\lara{C,A_0,A_1} &\subset&   
\lara{C,A_0,A_1,A_2} &\subset& 
\lara{C,A_0,A_1,A_2,A_3} &\subset\cdots\subset& 
\cB\\ 
 && \cup
 && \cup
 && \cup
 && \cup
 && \cup\\
           && 
\lara{C} &\subset&                
\lara{C,A_1} &\subset& 
\lara{C,A_1,A_2} &\subset&   
\lara{C,A_1,A_2,A_3} &\subset\cdots\subset& 
\beta (\cB)\\
 && 
 && \cup
 && \cup
 && \cup
 && \cup\\
           && 
           &&
\lara{C} &\subset&           
\lara{C,A_2} &\subset& 
\lara{C,A_2,A_3} &\subset\cdots\subset& 
\beta^2  (\cB)\\
 && 
 && 
 && \cup
 && \cup
 && \cup\\
 && 
 && 
 && \vdots
 && \vdots
 && \vdots\\
\end{matrix}
\end{eqnarray*}
\setcounter{MaxMatrixCols}{10}
The sequence $(A_i)_{i \in \Nset_0}$ is exchangeable with tail algebra $\vN(C)$. The representation $\rho_0$ restricts to a representation of $\Sset_\infty$ in the automorphisms of $\cB$.  
\end{enumerate}
\end{Theorem}

\begin{proof}
\eqref{item:cs-coxeter-i} This is immediate from Lemma \ref{lem:inclusions}, Proposition \ref{prop:trace}, Theorem \ref{thm:fixed-points} and Lemma \ref{lem:average-2}. 

\eqref{item:cs-coxeter-ii} The minimality and stationarity are clear by Proposition \ref{prop:trace}. Let $\cN:= \bigcap_{n \ge 0} \vN(u_k \mid k \ge n)$ denote the tail algebra of the sequence $(u_i)_{i \in \Nset}$. It follows 
$\cN \subset \cZ(\cA) = \cA_{-1} = \cA^{\alpha_0}$ from \eqref{eq:B2} and 
Theorem \ref{thm:indy}\eqref{item:indy-i}. We infer further from $\alpha_0(u_i) = u_{i+1}$ (by Proposition \ref{prop:trace}) that $\cN$ equals the tail algebra $\cA^{\tail,\alpha_0}$ of $\alpha_0$, where $\cA^{\tail,\alpha_0} := \bigcap_{n \ge 0} \alpha_0^n(\cA)$. Now $\cA^{\alpha_0} \subset \cA^{\tail,\alpha_0} = \cN \subset \cZ(\cA) = \cA^{\alpha_0}$ implies that $(u_i)_{i \in \Nset}$ has the tail algebra $\cZ(\cA)= \cA_{-1}$. That $(u_i)_{i\in \Nset}$ may lack the properties listed 
in \eqref{item:cs-coxeter-ii} is concluded from the equivalences in \eqref{item:cs-coxeter-v} and Remark \ref{rem:lack-u} below.  

\eqref{item:cs-coxeter-iii}
We want to check the Markov property for $\alpha_0$ with generator $\cA_1$ and hence we must consider the subalgebras $\cB_{[m,n]} := \vN(\alpha_0^k(\cA_1) \colon 0 \le k \le n)$ with conditional expectations $\mathcal{E}_{[m,n]}$.
Inspection of \eqref{item:cs-coxeter-iv} shows that
\[
\begin{matrix}
\cB_{[0,n]} = \vN(C, A_0,\ldots, A_n, u_{1},\ldots, u_{n+1}) &\subset & \vN_\pi(\Sset_\infty)  \\
\cup & & \cup\\
\cB_{[n,n]} = \vN(C, A_n, u_{n+1}) & \subset & \alpha_0^n(\vN_\pi(\Sset_\infty))& = \cB_{[n, \infty)}  
\end{matrix}
\]
is a commuting square. Consequently, $\mathcal{E}_{[0,n]}\mathcal{E}_{[n,\infty)} = \mathcal{E}_{[n,n]}$ by Lemma \ref{lem:commuting-square}\eqref{item:cs-iii} which clearly implies the Markov property. We need to identify the transition operator $R_0$. With
$\alpha_0 = \lim_{k \to \infty} \rho_0 (\sigma_1\sigma_2 \cdots \sigma_k)$
(see Corollary \ref{cor:endo-braid-i}) we can compute
\begin{align*}
R_0 = E_{1} \alpha_0 E_{1} 
&= E_{1} \rho_0^{}(\sigma_1\sigma_2) E_{1} 
&&\text{(since $\rho_0^{}(\sigma_i) E_{1} = E_1 $ for $i \ge 3$)}\\
&= \rho_0^{}(\sigma_1)E_{1} \rho_0^{}(\sigma_2) E_{1}
&&\text{(since $E_{1} \rho_0^{}(\sigma_1) = \rho_0^{}(\sigma_1) E_{1}$)}\\
&= \rho_0^{}(\sigma_1) E_{1} \alpha_1 E_{1} 
&&\text{(since $\rho_0^{}(\sigma_i) E_{1} = E_1 $ for $i \ge 3$)} \\
&= \rho_0^{}(\sigma_1) E_{0}
&&\text{(by Theorem \ref{thm:cs-star}\eqref{item:cs-star-iii})} \\
&= \alpha_0 E_{0} 
&& \text{(since $\rho_0^{}(\sigma_i) E_{0} = E_0 $ for $i \ge 2$).}
\end{align*}
Next we identify the action of $R_0$ on a weak*-dense set in $\cA_1 = \vN(C,A_0,u_1)$. By the module property of conditional expectations, using $A_0^\epsilon = E_0(u_\epsilon)$ for $\epsilon= 0,1$, 
\begin{eqnarray*}
R_0\Big(C_{k_1}^{n_1}\cdots C_{k_p}^{n_p} A_0^{n} u_\epsilon\Big)
&=& \alpha_0 E_0\Big(C_{k_1}^{n_1}\cdots C_{k_p}^{n_p} A_0^{n} u_\epsilon\Big)\\
&=& \alpha_0 \Big(C_{k_1}^{n_1}\cdots C_{k_p}^{n_p} A_0^{n} E_0(u_\epsilon)\Big)\\
&=& \alpha_0 \Big(C_{k_1}^{n_1}\cdots C_{k_p}^{n_p} A_0^{n+\epsilon}\Big)\\
&=& C_{k_1}^{n_1}\cdots C_{k_p}^{n_p} A_1^{n+\epsilon}.
\end{eqnarray*}

\eqref{item:cs-coxeter-iv} 
We apply Theorem \ref{thm:endo-braid-i} where $\alpha$ is chosen to be the $0$-shifted endomorphism $\alpha_0$. Combined with Theorem \ref{thm:fixed-points} and \eqref{item:cs-coxeter-i} above we obtain this triangular tower of commuting squares. 

\eqref{item:cs-coxeter-v} 
By Theorem \ref{thm:fixed-points} the center $\cZ(\cA)$ is equal to $\vN(C)$ and hence the equivalence of \eqref{item:cs-coxeter-v-a} and \eqref{item:cs-coxeter-v-e} is clear. From \eqref{item:cs-coxeter-v-e}, namely from the commuting squares
\[
\begin{matrix}
\vN(C, u_{1},\ldots, u_n) &\subset & \cA  \\
\cup & & \cup\\
\vN(C) & \subset & \alpha_0^n(\cA) 
\end{matrix}
\]
we obtain the order $\cZ(\cA)$-independence of the sequence $(u_i)_{i \in \Nset}$. 
This can be upgraded to full $\cZ(\cA)$-independence, that is  \eqref{item:cs-coxeter-v-c}, by additionally using the commutation relations 
(\ref{eq:B2}) of the Coxeter generators $u_i$. These arguments are so similar to
those proving the corresponding statement in \cite[Theorem 5.6(vi)]{GoKo09a} that we refer the reader to this source. 

Now \eqref{item:cs-coxeter-v-d} is just a reformulation of \eqref{item:cs-coxeter-v-c}
using the terminology of Bernoulli shifts and clearly \eqref{item:cs-coxeter-v-c} or
\eqref{item:cs-coxeter-v-d} imply the apparently weaker property \eqref{item:cs-coxeter-v-b}. Hence we can finish the proof of \eqref{item:cs-coxeter-v}
by showing that \eqref{item:cs-coxeter-v-b} implies \eqref{item:cs-coxeter-v-a}. 
First note that for any $g\in \Sset_n$
\[
      E_{-1}(\pi(g) \, u_n)
= E_{-1} \big((u_1 \ldots u_{n-1} \,\pi(g)\, u_{n-1} \ldots u_1)
              (u_1 \ldots u_{n-1} u_n u_{n-1} \ldots u_1)\big)
     = E_{-1}( \pi(h) \, v_{n})
\]
where $h := \sigma_1 \ldots \sigma_{n-1} \,g\,\sigma_{n-1} \ldots \sigma_1 \in \Sset_n$.
Note that $E_{-1}$ as a $\trace$-preserving conditional expectation onto the center $\cZ(\cA)$ is a center-valued trace (compare \cite[Chapter 8]{KaRi2}). If additionally $E_{-1}$ is Markov, from \eqref{item:cs-coxeter-v-b}, we conclude that
\[
E_{-1}( \pi(h) \, v_n) =
        E_{-1}(\pi(g) u_n)
        = E_{-1}(\pi(g)) E_{-1}(u_n)
       = E_{-1}(\pi(h)) E_{-1}(v_n)
\]
This is valid for all $g$ and hence for all $h$ in $\Sset_n$.
We show now that this implies order $\cZ(\cA)$-factorizability of the sequence $(v_i)_{i \in \Nset}$.

In fact, from the disjoint cycle decomposition together with Lemma \ref{lem:cycle-1} it is clear that every permutation can be written as a
product $\gamma_{n_1} \ldots \gamma_{n_k}$ of star generators where the maximal subscript $n_\ell$ appears only once. Hence with
$n_\ell = n$
\[
E_{-1}(v_{n_1} \ldots v_{n_k}) = E_{-1}( \pi(h) \, v_{n_\ell})
= E_{-1}(\pi(h)) E_{-1}(v_{n_\ell})
\]
where $h \in \Sset_n$.
By iterating this argument and extending it to the weak closures of linear spans the order $\cZ(\cA)$-factorizability of the sequence
$(v_i)_{i \in \Nset}$ follows.

But we know from Theorem \ref{thm:cs-star}(\ref{item:cs-star-ii})
that the minimal sequence $(v_i)_{i \in \Nset}$ always has the tail
algebra $\cA_0 \supset \cZ(\cA)$ and we conclude by Theorem \ref{thm:fixed-point} (\ref{item:fixed-point-iii}) that $\cA_0 = \cZ(\cA)$. This implies \eqref{item:cs-coxeter-v-a}.

\eqref{item:cs-coxeter-vi} 
This follows from \eqref{item:cs-coxeter-v} by taking
Corollary \ref{cor:fixed-points} into account.

\eqref{item:cs-coxeter-vii} 
Since $\alpha_0(C_k) = C_k$ and $\alpha_0(A_i)= A_{i+1}$, the endomorphism $\alpha_0$ restricts to an endomorphism $\beta$ on $\cB:=\vN(C, A_i \mid i \in \Nset_0)$. Clearly $\vN(C,A_0)$ is a generator of $\beta$. Inspecting \eqref{item:cs-coxeter-iv} we see that the sequence $\Big(\beta^{k}\big(\vN(C,A_0)\big)\Big)_{k \in \Nset_0}$ is order $\vN(C)$-independent. Thus $\beta$ is a Bernoulli shift over $\cZ(\cA) = \vN(C)$ with generator $\vN(C,A_0)$. Clearly $\cB$ is abelian since all limit cycles
$C_k$ and $A_i$ mutually commute, see Proposition \ref{prop:A-prop}. Hence $\beta$ is a classical Bernoulli shift. The rest is now immediate.

\end{proof}
\begin{Remark} \normalfont \label{rem:lack-u}
The following are equivalent:
\begin{enumerate}
\item \label{item:cs-coxeter-viii-a}
The unitary $u_1$ is central.
\item \label{item:cs-coxeter-viii-b}
The sequence $(u_i)_{i \in \Nset_0}$ is spreadable. 
\end{enumerate}
In fact, if $u_1$ is central then all $u_i$ are central. This implies $u_i = \alpha^{i-1}_0(u_1) = u_1$ and spreadability (in a commutative algebra $\cA$). Conversely assume that $(u_i)_{i \in \Nset_0}$ is spreadable. Then we have
\begin{eqnarray*}
\trace( |u_1 u_2 - u_2 u_1|^2) 
&=& 2 - \trace(u_1 u_2 u_1 u_2) - \trace(u_2 u_1 u_2 u_1) 
= 2 - \trace(u_1 u_3 u_1 u_3) - \trace(u_3 u_1 u_3 u_1) \\
&=& 2 - \trace(u_1 u_1 u_3 u_3) - \trace(u_1 u_1 u_3 u_3) 
= 2 - 2 \trace(\1) = 0,
\end{eqnarray*}
so $u_1$ and $u_2$ commute. Similarly it follows that $u_1$ commutes with all $u_i$. 

The situation characterized above is a very special one, compare Remark \ref{rem:central-lim-cycles}. More interesting is the situation when $A_0$ is central (resp. trivial) and $u_1$ is non-central, see Section \ref{section:Markov} for concrete examples. Then we conclude with Theorem \ref{thm:cs-coxeter} \eqref{item:cs-coxeter-v} (resp. \eqref{item:cs-coxeter-vi}) that $(u_i)_{i\in \Nset}$ is a full $\cZ(\cA)$ (resp. $\Cset$)-independent stationary sequence which fails to be spreadable. This exemplifies the failure of the implication (c) $\Rightarrow$ (b) in the noncommutative extended de Finetti theorem, Theorem \ref{thm:de-finetti}. 
\end{Remark}
\begin{Remark}\normalfont \label{rem:b-m-shifted}
Exploiting the full power of $N$-shifted representations $\rho_0^{} \circ \sh^N$, with $N \in \Nset_0$, each $N$-shifted endomorphism $\alpha_N$ is a Markov shift with generator $\cA_{N+1}= \vN(C,A_0,v_1, \ldots, v_{N+1})$. The transition operator $R_N$ of $\alpha_{N}$ is determined by a calculation similar to the one in the proof of Theorem \ref{thm:cs-coxeter} \eqref{item:cs-coxeter-iii}
\[
R_N 
= E_{N+1} \alpha_N E_{N+1}
= \alpha_N E_{N}.
\] 
If $N>0$ then, simultaneously, the endomorphism $\alpha_N$ is a Bernoulli shift over $\cA_{N-1}$
with generator $\cA_{N}= \vN(C,A_0,v_1, \ldots, v_{N})$. 
For $N=0$ this is also true except that we have to restrict $\alpha_0$ to obtain the classical Bernoulli shift described in 
Theorem \ref{thm:cs-coxeter} \eqref{item:cs-coxeter-vii}.
\end{Remark}
\section{Commuting squares over $\Cset$ with a normality condition}
\label{section:commsquares}
We have seen in the previous sections that the star generators $\gamma_i = (0,i)$ yield an exchangeable sequence such that its tail algebra is generated by the limit $2$-cycle $A_0$ for an extremal character of $\Sset_\infty$. A goal of present section is to identify this tail algebra as an abelian atomic von Neumann algebra. The related spectral analysis of $A_0$  
in Proposition \ref{prop:mu-discrete} involves noncommutative independence in a crucial manner. Our approach is inspired by Okounkov's proofs in  \cite{Okou99a} and, roughly speaking, replaces properties coming from a certain Olshanski semigroup by properties of limit cycles. As we will see in 
the main results of this section, Theorem \ref{thm:discrete-spec} and Corollary \ref{cor:discrete-spec}, our approach carries over to a general setting of certain commuting squares over $\Cset$. This will be briefly illustrated by Hecke algebras in Example \ref{example:cs3}.

We start with the spectral analysis of the limit $2$-cycle $A_0$. Because  $u_1$ is selfadjoint,
\[
A_0 = E_{0}(u_1)
\] 
is a selfadjoint contraction. Denote by $\chi_B(A_0)$ the spectral projection of $A_0$ corresponding to the Borel set $B \subset [-1,1]$.
Let the measure $\mu$ be the spectral measure of $A_0$ with respect to the state $\trace$, which is uniquely determined by the moments
\[
c_{k+1} := \int t^k \mu(dt) = \trace(A_0^k). 
\]
Note that if $\cA_{-1} = \Cset$, i.e., if $\cA$ is a factor, then $c_k \1 = C_k$, one of the limit cycles considered in Section \ref{section:limit-cycles}.
Our next result is inspired by \cite[Theorem 1]{Okou99a}.
\begin{Proposition}\label{prop:mu-discrete} 
Suppose $\cA$ is a factor. Then the spectral measure $\mu$ is discrete and its atoms can only accumulate at $0 \in [-1,1]$.
\end{Proposition} 
\begin{proof}
Let $B$ be a Borel subset in $[\epsilon, 1]$ where $1 > \epsilon >0$ and denote by $\chi_B$ its characteristic function. We claim that
\[
\epsilon\, \mu(B) \le \mu(B)^{3/2}.
\] 
To this end, we prove the two inequalities 
\[
\epsilon\, \mu(B) \le \trace(\chi_{B}\, u_1 ) \le  \mu(B)^{3/2}, 
\]
where we have abbreviated the spectral projection $\chi_{B}(A_0)$ by $\chi_{B}$. The first inequality is immediate from   
\begin{eqnarray*}
 \trace(\chi_{B}\, u_1 ) 
 =  \trace(\chi_{B}\, E_{0}(u_1))
 = \trace(\chi_{B}\, A_0)
 = \int_B t\, \mu(dt) \ge \epsilon\, \mu(B).
\end{eqnarray*}
The second inequality follows from 
\begin{eqnarray*}
 \trace(\chi_{B}\, u_1) 
&=& \trace(\chi_{B}\, u_1\, \chi_{B})\\
&=& \trace(u_1\chi_{B}\, u_1\, \chi_{B}\, u_1)\\
&=& \trace(u_1\,\chi_{B}\, u_1 \,\chi_{B}\, \chi_{B}\, u_1)\\  
&=& \trace(\alpha(\chi_{B})\,\chi_{B}\,  \chi_{B}\, u_1)\\
&\le& \trace(|\alpha(\chi_{B})\,\chi_{B}|^2)^{1/2}
    \trace(|\chi_{B}\, u_1|^2) ^{1/2} 
\end{eqnarray*}
The second factor becomes
\[
 \trace(|\chi_{B}\, u_1|^2) = \trace(\chi_{B} ) = \mu(B).
\] 
For the first factor, we use that $\alpha_0(x)$ and $x$ are $\cA_{-1}$-independent for $x\in \cA_0$. So
\begin{eqnarray*}
\trace(|\alpha(\chi_{B})\,\chi_{B}|^2)
&=& \trace(\alpha(\chi_{B})\,\chi_{B})\\
&=& \trace( E_{-1}(\chi_{B}) E_{-1}(\chi_{B})).
\end{eqnarray*}
because $E_{-1} \circ \alpha = E_{-1}$. 
At this place we make use of the assumption that $\cA$ is a factor.
Thus $\cA_{-1} = \cZ(\cA) \simeq \Cset$ and $E_{-1}(x) = \trace(x)\1_\cA$ for $x \in \cA$. Consequently,  
\begin{eqnarray*}
\trace(|\alpha(\chi_{B})\,\chi_{B}|^2)
&=& \trace( E_{-1}(\chi_{B}) E_{-1}(\chi_{B}))\\
&=& (\trace(\chi_{B}))^2 = \mu(B)^2.  
\end{eqnarray*}
Altogether this proves the second inequality
\[
 \trace(\chi_{B}\, u_1)  \le \mu(B)^{3/2}. 
\]
It follows from the inequality $\epsilon \mu(B) \le \mu(B)^{3/2}$ that either
$\mu(B) = 0$ or $\mu(B) \ge  \epsilon^2$. A similar estimate holds for $B \subset [-1,-\epsilon]$. This implies that the measure $\mu$ is discrete. Since $\mu$ is a probability measure, each of the two intervals $[\epsilon, 1]$ and $[-1, -\epsilon]$ contains no more than $1/\epsilon^2$ atoms of $\mu$. Thus $0$ is the only possible accumulation point of these atoms.  This finishes the proof.
\end{proof}
Our estimate is slightly different from Okounkov's but it leads in a similar way to Thoma's theorem, as we will see later. The proof of Property \ref{prop:mu-discrete} suggests the following generalization to an axiomatic setting involving commuting squares of von Neumann algebras.
\begin{Theorem}\label{thm:discrete-spec}
Let $(\cM,\psi)$ be a probability space and suppose that  $E_{\cM_0}$ is a $\psi-$preserving conditional expectation from $\cM$ onto a von Neumann subalgebra $\cM_0$ of $\cM$. Suppose a unitary $u$ in the centralizer $\cM^\psi$ satisfies the following two conditions:
\begin{enumerate}
\item
$u$ implements the commuting square 
\begin{eqnarray}\label{eq:discrete-cs}
\begin{matrix}
u \cM_0 u^*   & \subset &  \cM\\
\cup        &         &   \cup \\ 
\Cset       & \subset & \cM_0
\end{matrix}\,\,,
\end{eqnarray}
or, equivalently, $\cM_0$ and $u\cM_0 u^*$ are $\Cset$-independent;
\item 
the contraction $E_{\cM_0}(u)$ is normal. 
\end{enumerate}
Then $E_{\cM_0}(u)$ has discrete spectrum which can only accumulate at the point $0$.  
\end{Theorem}
\begin{Remark}\normalfont
It is presently an open problem to construct a commuting square \eqref{eq:discrete-cs} where $E_{\cM_0}(u)$ fails to be normal, if possible at all. 
\end{Remark}
\begin{proof}[Proof of Theorem \ref{thm:discrete-spec}]
Let $\{\chi_B\}_B$ denote the spectral resolution of the normal operator $E_{\cM_0}(u)$, where $B \subset \Cset$ is a Borel set. Further let $\mu$ be the probability measure associated to $\{\chi_B\}_B$ with respect to the state $\psi$, i.e.
\[
\mu(B) := \psi(\chi_B).
\]
Since $u \in \cM^\psi$ and the modular automorphism group of $(\cM,\psi)$ commutes with $E_{\cM_0}$, we have $E_{\cM_0}(u) \in \cM^\psi$. Thus $E_{\cM_0}(u) \in \cM_0 \cap \cM^\psi$ and $\psi$ restricts to a trace on the von Neumann subalgebra generated by $u$ and $E_{\cM_0}(u)$. Moreover this subalgebra contains the projections $\chi_B$. 

Since $E_{\cM_0}(u)$ is a contraction, its spectrum $\spec (E_{\cM_0}(u))$ is contained in $D_1:=\set{z \in \Cset}{|z|\le 1}$.  For $0 < \epsilon < 1$, 
let $S(\epsilon)$ be one of the four segments
\begin{eqnarray*}
S_{\Re}^{+}(\epsilon)&:=& \set{z \in D_1}{\epsilon \le \Re z \le 1},\\
S_{\Re}^{-}(\epsilon)&:=& \set{z \in D_1}{-1 \le \Re z \le -\epsilon},\\
S_{\Im}^{+}(\epsilon)&:=& \set{z \in D_1}{\epsilon \le \Im z \le 1},\\
S_{\Im}^{-}(\epsilon)&:=& \set{z \in D_1}{-1 \le \Im z \le -\epsilon}.
\end{eqnarray*} 
We will show the estimates 
\begin{eqnarray}\label{eq:discretemeasure}
\epsilon\, \psi(\chi_B) \le |\psi(\chi_B u )| \le  \psi(\chi_B)^{3/2}   \end{eqnarray} 
are valid for any Borel set $B \subset S(\epsilon)$. This ensures
\[
\epsilon\, \mu(B) \le \mu(B)^{3/2}
\] 
and thus either $\mu(B) = 0$ or $\epsilon^2 \le \mu(B)$. We conclude from this that the measure $\mu$ is discrete when restricted to $S(\epsilon)$. Since $\mu$ is a probability measure, the segment $S(\epsilon)$ contains no more than $1/\epsilon^2$ atoms of $\mu$. Consequently,  $E_{\cM_0}(u)$ has pure point spectrum which can only accumulate at $0 \in D_1$. \\
We are still left to ensure the validity of \eqref{eq:discretemeasure}.

By the spectral calculus for normal operators and the Pythagoras theorem,
\begin{eqnarray*}
|\psi(\chi_B \, u)|^2 
&=& |\psi(\chi_B \, E_{\cM_0}(u))|^2 \\
&=& \left| \int_B z \, \mu(dz)\right|^2 
= \left| \int_B \Re z \,\mu(dz)\right|^2 + \left| \int_B \Im z\, \mu(dz)\right|^2 \\ 
&\ge& \left|\epsilon \,\mu(B) \right|^2.
\end{eqnarray*}
This proves the first inequality of \eqref{eq:discretemeasure},
\begin{eqnarray*}
\epsilon \psi(\chi_B) \le |\psi(\chi_B\, u )|. 
 \end{eqnarray*} 
We turn our attention to the second inequality of \eqref{eq:discretemeasure},
\[
|\psi(\chi_B \, u)| \le \psi(\chi_B)^{3/2}.
\]
Indeed,
\begin{eqnarray*}
\psi(\chi_B \,u) 
&=& \psi(\chi_B\, u \,\chi_B)\\ 
&=& \psi(u \,\chi_B\, u \,\chi_B\, u^*) \\
&=&  \psi\Big((u \,\chi_B) (\chi_B \, u \, \chi_B \, u^*)\Big)
\end{eqnarray*} 
is immediate since $\chi_B$ is an orthogonal projection in $\cA^\psi$ and
$u$ is a unitary in $\cA^\psi$. We apply the Cauchy-Schwarz inequality to obtain 
\begin{eqnarray*}
|\psi(\chi_B \, u)|^2 
&\le& \psi(u \,\chi_B \,\chi_B\, u^*)  \cdot
     \psi\Big((u \,\chi_B\, u^* \,\chi_B) (\chi_B\, u\,\chi_B\,  u^*)\Big)\\
&=& \psi(\chi_B)  \cdot
     \psi\Big((u \,\chi_B \,u^*)\cdot \chi_B \cdot (u\,\chi_B \, u^*)\Big).    \end{eqnarray*}
Since $\chi_B$ and $ u \, \chi_B \, u^*$ are $\Cset$-independent, we conclude further for the second factor that 
\begin{eqnarray*}
 \psi\Big((u \chi_B u^*) \chi_B (u\chi_B  u^*)\Big) 
 &=& \psi\Big((u \chi_B u^*) \psi(\chi_B) (u\chi_B  u^*)\Big) \\
 &=& \psi(\chi_B) \psi\Big((u \chi_B u^*) (u\chi_B  u^*)\Big) \\
 &=& \psi(\chi_B) \psi(\chi_B).
\end{eqnarray*}
Thus 
\[
|\psi(\chi_B u)|^2 \le \psi(\chi_B)^3, 
\]
which establishes the second inequality of \eqref{eq:discretemeasure}. 
\end{proof}
We can restrict to the abelian von Neumann algebra $\cB_0$ of $\cM_0$ generated by the operator $E_{\cM_0}(u)$ from Theorem \ref{thm:discrete-spec}:
\[
\cB_0 := \vN\{E_{\cM_0}(u)\}.
\]
We infer from  $u \in \cM^\psi$ that  
\[
 \cB := \cB_0 \vee \vN\{u\} 
\]
is contained in $\cM^\psi$.  Thus the state $\psi$ on $\cM$ restricts to a faithful normal trace on $\cB$, denoted by $\trace$. Let us identify $\Cset \1_\cB$ and $\Cset$ for notational simplicity. 
\begin{Corollary} \label{cor:discrete-spec}
Under the assumptions of Theorem \ref{thm:discrete-spec} and with $\cB_0$, $\cB$ and $\trace$ as introduced above, the commuting square \eqref{eq:discrete-cs} restricts to the commuting square
\begin{eqnarray}\label{eq:discrete-cs-min}
\begin{matrix}
u \cB_0 u^*   & \subset & \cB \\
\cup        &         &   \cup \\ 
\Cset       & \subset &   \cB_0
\end{matrix}.
\end{eqnarray}
In particular, $\cB_0$ is an abelian atomic von Neumann subalgebra of $\cB$ generated by the normal contraction $E_{\cB_0}(u)$.
\end{Corollary}   
\begin{proof}
This is elementary.
\end{proof}
%

A classification of the commuting squares \eqref{eq:discrete-cs-min} is of course of general interest, but would go beyond the scope of present paper.  
Instead let us now look at a few examples. We start with an example which underlies the representations of the symmetric group from Thoma's theorem, Theorem \ref{thm:ethoma}. 
\begin{Example}\label{example:cs1}\normalfont
Define the probability space $(\cM, \psi)$ in Theorem \ref{thm:discrete-spec} as follows. Consider $\cM:= \bfM_n(\Cset) \otimes \bfM_n(\Cset)$ and $\cM_0:= \bfM_n(\Cset) \otimes \1$, with $1\le n \le \infty$.  Let $e_{ij}$ be the standard matrix units of $\bfM_n(\Cset)$ and $\Trace$ the (non-normalized) trace on $\bfM_n(\Cset)$ with $\Trace(e_{ii})=1$. The tensor product state 
$\psi$ on $\cM$ is defined by the density operator $D \otimes D$ with the diagonal density operator $D = \diag(a_1,a_2,\ldots)$ so that $a_i >0$ for all $i$ and $\sum_i a_i = 1$. Further consider a unitary $u$ which implements a tensor flip
\[
u = \sum_{i,j} e_{ij} \otimes e_{ji}.
\]
Here $e_{ij}$ are the standard matrix units.
A short computation yields
\begin{equation}  \label{eq:example-cs-1} 
E_{\cM_0}(u) = \sum_i a_i\, e_{ii}
\end{equation}
and both the assumptions and the conclusion of Theorem
\ref{thm:discrete-spec} are easily verified in this example. 
%
\end{Example}
Example \ref{example:cs1} has the special property that the ratio between the state applied to a spectral projection of $E_{\cM_0}(u)$ and the corresponding eigenvalue is an integer, namely the geometric multiplicity of the eigenvalue.  This property plays a role in the proof of Thoma's theorem and we will take up this issue in Section \ref{section:okounkov}. Let us give a trivial example for the fact that this special property is not automatic from the assumptions of Theorem \ref{thm:discrete-spec}. 
\begin{Example}\label{example:cs2}\normalfont
Let $u$ be any unitary on a separable Hilbert space $\cH$
and denote by $\cB$ the abelian von Neumann algebra generated by $u$. Now consider as $\cB_0$ the one-dimensional subalgebra $\Cset \1_{\cB}$ of $\cB$. Let $\trace$ be a faithful normal trace on $\cB$. Then the conditional expectation $E_{\cB_0}$ from $\cB$ onto $\cB_0$ is given by $E_{\cB_0}(x) = \trace(x) \1$. The assumptions of Theorem \ref{thm:discrete-spec} or Corollary
\ref{cor:discrete-spec} are trivially satisfied. The only spectral projection of $E_{\cB_0}(u)$ is $\1$. But the eigenvalue is $\trace(u)$
for which we may choose any number in the unit disk by an appropriate choice of the unitary $u$. Thus we do not always have an integer ratio as in Example 
\ref{example:cs1}.
\end{Example}
If $u$ is selfadjoint then the compression $E_{\cM_0}(u)$ is automatically selfadjoint and hence normal. This will be satisfied in our application to the symmetric group $\Sset_\infty$. But there are other interesting examples coming from certain Hecke algebras considered in subfactor theory where the unitary $u$ is non-selfadjoint and $E_{\cM_0}(u)$ is still normal. 
\begin{Example}\label{example:cs3}\normalfont
The Hecke algebra $H_{q,n}$ over $\Cset$ with parameter $q\in \Cset\backslash \{0\}$ is the unital algebra with generators $g_1,\ldots, g_{n-1}$ and relations
\begin{align}
&&&&g^2_i &= (q-1) g_i + q, &&&&&\label{eq:hecke-1}\\ 
&&&&g_i g_j &= g_i g_j  & \text{if $\mid i-j \mid \geq 2$,}&&&&&\notag\\
&&&&g_i g_{j} g_i &= g_{j} g_i g_{j} & \text{if $\mid i-j \mid = 1$.}&&&&&\notag
\end{align}
Note that $H_{1,n}$ is the group algebra $\Cset\, \Sset_n $. Let $H_{q,\infty} = \bigcup_{n \ge 2} H_{q,n}$ be the inductive limit of the Hecke algebras.  
If $q$ is a root of unity then it is possible to define an involution and a Markov trace $\trace$ such that the $g_n$ become unitaries $u_n$ and the generated von Neumann algebra $\cM$ is isomorphic to the hyperfinite II$_1$-factor, see \cite{Wenz88a,Jone91a,Jone94a} or \cite[Example 6.3]{GoKo09a}. Compare also Section \ref{section:Markov} below.

Now let $\cM_0$ be generated by $u_1$ and consider the unitary $u=u_1 u_2$. Then one has 
\[
u u_1 u^* = u_1 u_2 u_1 u_2^* u_1^* = u_2 u_1 u_2 u_2^* u_1^*= u_2.
\]
It follows from the definition of a Markov trace that $\cM_0$ and $u \cM_0 u^*$ are the corners of a commuting square over $\Cset$.  In other words, $\cM_0$ and $u \cM_0 u^*$ are $\Cset$-independent.
Now it is readily checked that
\begin{equation}\label{eq:example-3}
E_{\cM_0}(u) = E_{\cM_0}(u_1 u_2)= u_1 E_{\cM_0}(u_2) =  \trace(u_2) u_1\,.    \end{equation}
Hence $E_{\cM_0}(u)$ is normal.
From the quadratic relation \eqref{eq:hecke-1} we find that the eigenvalues of $u_1$ are $q$ and $-1$. We conclude from \eqref{eq:example-3} that
if $q$ is not real then $E_{\cM_0}(u)$
is not selfadjoint and hence $u$ is not selfadjoint too. 
Also we infer from \eqref{eq:hecke-1} and \eqref{eq:example-3} that $E_{\cM_0}(u)$ generates a two-dimensional (von Neumann) subalgebra.  
\end{Example}

\section{Thoma measures -- Okounkov's argument revisited}
\label{section:okounkov}
This section is devoted to the complete identification of the spectrum of the limit cycle $A_0$ in the factorial case. We know already from Proposition \ref{prop:mu-discrete} or Theorem \ref{thm:discrete-spec} that $\spec A_0$ is discrete. A first hint on its additional properties was provided in Example \ref{example:cs1}, where we have observed that the ratio between a state applied to a spectral projection of $A_0$ and the corresponding eigenvalue is a positive integer. Here we will take up this crucial issue in a systematic way. Our approach is self-contained and will follow closely the arguments of Okounkov in \cite{Okou99a}.      
\begin{Definition}\normalfont
A discrete probability measure $\mu$ on the interval $[-1,1]$ is said to be a \emph{Thoma measure} if 
\[
\frac{\mu(t)}{|t|} \in \Nset_0 \qquad (t \neq 0). 
\]

\end{Definition}
\begin{Theorem}\label{thm:thoma}
Let $\mu$ be the spectral measure of the $2$-cycle $A_0$ w.r.t.~the faithful (normal) tracial state $\trace$ on the factor $\vN_\pi(\Sset_\infty)$. Then $\mu$ is a Thoma measure. 
\end{Theorem}

We recall some notation from Section \ref{section:limit-cycles} for our next results. $\Nset_0\slash \langle \sigma \rangle$ denotes the set of all orbits of $\Nset_0$ under the action of the subgroup $\langle \sigma \rangle$ generated by the permutation $\sigma \in \Sset_\infty$. The set of orbits $\Nset_0\slash \langle \sigma \rangle$ forms a partition $\{V_1, V_2,\ldots\}$  of $\Nset_0$. The cardinality of a block $V_i$ of this partition will be denoted by $|V_i|$. Note that only finitely many blocks $V_i$ have a cardinality larger than 1. 

The following lemma is slightly more general than we will need it. Recall that
$E_{-1}$ is the $\trace$-preserving conditional expectation from the finite von Neumann algebra $\vN_\pi(\Sset_\infty)$ onto its center.
\begin{Lemma}\label{lem:thoma-1}  
Let $f_i(t)$, $i=0,1,2,\ldots$, be bounded Borel functions on $[-1,1]$, all but finitely many equal to $1$. Then, for $\sigma \in \Sset_\infty$,
\[
E_{-1}\Big(\pi(\sigma) \prod_{i=0}^\infty f_i(A_i)\Big) 
= \prod_{V \in \Nset_0\slash \langle \sigma \rangle} 
   E_{-1}\Big( A_0^{|V|-1} \prod_{j \in V} f_j(A_0)\Big).   
\]  
\end{Lemma}
Note at this point that the spectrum of $A_0$ may contain continuous parts since $\vN_\pi(\Sset_\infty)$ is not assumed to be factorial and, consequently, Theorem \ref{thm:discrete-spec} does not apply directly.  
\begin{proof}
The formula is immediate for the identity $\sigma = \sigma_0$ from the $\cA_{-1}$-independence of $\big(\alpha^k(\cA_0)\big)_{k \ge 0}$, Theorem \ref{thm:indy} \eqref{item:indy-i}, Lemma \ref{lem:average-2} and the functional calculus of bounded Borel functions (as explained in more detail at the end of the proof). From now on assume $\sigma \neq \sigma_0$. 
  
Suppose the functions $f_i$ are of the form $f_i(t) = t^{k_i}$, where $k_i \in \Nset_0$. Then, with Lemma \ref{lem:average-2}, 
\[
f_i(A_i) = v_i A_0^{k_i} v_i = \big(v_i A_0 v_i\big)^{k_i} 
= \big(v_i E_0(v_1) v_i\big)^{k_i}.
\]
The factors $A_i = v_i E_0(v_1) v_i$ commute by \eqref{eq:A-item2} in Proposition \ref{prop:A-prop}. Further note that, by full $\cA_0$-independence of 
the sequence $(v_i)_{i \in \Nset}$ (see Theorem \ref{thm:indy} \eqref{item:indy-ii}),
\[
E_0(x v_i A_0 v_i y) = E_0(x v_i v_{N_i} v_i y)  
\] 
for $x,y \in \vN(A_0, v_n \mid n \le N)$ and $N_i > N$. More general, we can replace powers of limit 2-cycles by cycles such that 
\[
E_0(x v_i A_0^{k_i} v_i y) = E_0(x v_i v_{N_{i,1}} \cdots v_{N_{i, k_i}} v_i y),  
\] 
where all $N_{i,l}$ are pairwise distinct and larger than $N$. 

Now choose $N \in \Nset$ such that $\sigma \in \Sset_N$ and $i < N$ whenever $f_i$ is non-trivial. Suppose further that $\pi(\sigma)$ has the cycle decomposition $\pi(\sigma) = w_1w_2 \cdots w_{m(\sigma)}$, where $m(\sigma)$ is the number of non-trivial cycles in $\sigma$. Let $V_l$ be the orbit in $\Nset_0$ corresponding to the cycle $w_l$; that means: $i \in V_l$ if and only if $v_i$ appears in the cycle $w_l$.  Further let $W = \bigcup_{l=1}^{m(\sigma)} V_l$. Since disjoint cycles commute, we can regroup the factors and apply $\cA_0$-independence to find
\begin{eqnarray*}
E_{0}\Big(\pi(\sigma) \prod_{i=0}^\infty f_i(A_i)\Big)
&=& E_{0}\Big(\prod_{l=1}^{m(\sigma)} w_l  \cdot \prod_{i=0}^\infty v_i A_0^{k_i} v_i\Big)\\ 
&=& E_{0}\Big(\prod_{l=1}^{m(\sigma)} w_l  \cdot \prod_{i=0}^\infty v_i v_{N_{i,1}} \cdots v_{i, N_{i, k_i}} v_i\Big)\\
&=& E_{0}\Big(\prod_{l=1}^{m(\sigma)} 
\big(w_l  \prod_{i \in V_l} v_i v_{N_{i,1}} \cdots v_{i, N_{i, k_i}} v_i \big) 
\prod_{i \not \in W} v_i A_0^{k_i}v_i\Big)\\
&=& \prod_{l=1}^{m(\sigma)} E_{0}\Big(w_l  
    \prod_{i \in V_l} v_i v_{N_{i,1}} \cdots v_{i, N_{i, k_i}} v_i 
   \Big) \cdot  \prod_{i\not \in W} E_0\Big(v_i A_0^{k_i} v_i\Big)\\
&=& \prod_{l=1}^{m(\sigma)} E_{0}\Big(w_l  \prod_{i \in V_l} v_i A_0^{k_i} v_i \Big) \cdot  \prod_{i\not \in W} E_0\Big(v_i A_0^{k_i} v_i\Big).
\end{eqnarray*}
We consider next each factor of the $m(\sigma)$-fold product separately. Note for the following that, for $p,q \ge 2$, the product of a $p$-cycle $\sigma$ and a $q$-cycle $\tau$ with a single common point is a $(p+q-1)$-cycle. We claim that 
\[
E_{0}\Big(w_l  \prod_{i \in V_l} v_i A_0^{k_i} v_i \Big)
= 
E_{0}\Big(\tilde{w}_l \Big),
\]
where $\tilde{w}_l$ is a $|\tilde{V}_l|$-cycle with orbit
\[
\tilde{V}_l:= V_l \cup \set{N_{i,j}}{i \in V_l, j =1, \ldots, k_i}.
\] 
Indeed, we can replace in this expression each of the limit cycles $v_i A_0^{k_i} v_i$ by the corresponding $(k_i+1)$-cycle $v_i v_{N_{i,1}} \cdots v_{N_{i, k_i}}v_i$. Since this  $(k_i+1)$-cycle and the $|V_l|$-cycle $w_l$ have only the point $i$ in common, their product is a $(|V_l|+k_i)$-cycle.  Iterating this replacement for each $(k_i+1)$-cycle involved we arrive at a $(|V_l| + \sum_{i \in V_l} k_i)$-cycle $\tilde{w_l}$ which is compressed by $E_0$. Finally, the formula for $\tilde{V}_l$ and  $|\tilde{V}_l| = |V_l| + \sum_{i \in V_l} k_i$ is clear from above iteration. 
\[
E_{0}\Big(w_l  \prod_{i \in V_l} v_i A_0^{k_i} v_i \Big)
= E_{0}(\tilde{w}_l).
\]
By construction, the orbits $\tilde{V_l}$ are mutually disjoint and thus 
the cycles $\tilde{w}_l$ are $\cA_0$-independent. Also the $f_i(A_i)$'s with $i \notin W$ and $\tilde{w}_l$ are $\cA_0$-independent. We conclude from this that
\begin{eqnarray*}
E_{0}\Big(\pi(\sigma) \prod_{i=0}^\infty f_i(A_i)\Big)
&=& \prod_{l=1}^{m(\sigma)} E_{0}(\tilde{w}_l)  \cdot \prod_{i \notin W}  E_{0} \Big(f_i(A_i) \Big).
\end{eqnarray*}
We compress this equation with the conditional expectation $E_{-1}$ and calculate, as in the proof of Theorem \ref{thm:thoma-mult} (Thoma multiplicativity), 
\begin{eqnarray*}
E_{-1}\Big(\pi(\sigma) \prod_{i=0}^\infty f_i(A_i)\Big)
= \prod_{l=1}^{m(\sigma)} E_{-1}(\tilde{w}_l)  \cdot \prod_{i \notin W}  E_{-1} \Big(f_i(A_i) \Big).
\end{eqnarray*}      
Now 
\begin{eqnarray*}
E_{-1}(\tilde{w}_l)  
&=& E_{-1} \Big(A_0^{|\tilde{V}_l|-1}\Big)\\
&=& E_{-1}\Big(A_0^{|V_l|-1} \prod_{i \in V_l} A_0^{k_i}\Big)\\
&=& E_{-1}\Big(A_0^{|V_l|-1} \prod_{i \in V_l} f_i(A_0)\Big)
\end{eqnarray*}
and
\begin{eqnarray*}
E_{-1} \Big(f_i(A_i) \Big) 
&=& E_{-1} (v_i A_0^{k_i} v_i) 
= E_{-1}(A_0^{k_i})\\ 
&=& E_{-1}(f_i(A_0)). 
\end{eqnarray*}
Altogether we have arrived for monomials $f_i$ at the equation 
\begin{equation}\label{eq:thoma-1}
E_{-1}\Big(\pi(\sigma) \prod_{i=0}^\infty f_i(A_i)\Big) 
= \prod_{V \in \Nset_0\slash \langle \sigma \rangle} 
   E_{-1}\Big( A_0^{|V|-1} \prod_{j \in V} f_j(A_0)\Big).   
\end{equation}
We can extend this formula from monomials $f_i(t)= t^{k_i}$ to continuous functions and then to bounded Borel functions on $[-1,1]$. In fact, any bounded Borel function can be approximated by continuous functions in the sense of (bounded) pointwise convergence a.e.~(see the corollary to Lusin's theorem in \cite[2.24 ]{Rudi87a}). This translates into $\sot$-convergence in the functional calculus for bounded Borel functions of the selfadjoint contractions $A_j$, by dominated convergence. Hence by approximation \eqref{eq:thoma-1} is valid for bounded Borel functions as an equation in the von Neumann algebra $\cA$.

\end{proof}
Returning to the setting where $\vN_\pi(\Sset_\infty)$ is a factor we can replace $E_{-1}$ by the trace $\trace$. Now, with the spectral measure $\mu$ of $A_0$ (w.r.t.~$\trace$), Lemma \ref{lem:thoma-1} corresponds to 
\cite[Lemma 2 a) in Section 2]{Okou99a} and
writes as follows.  
\begin{Corollary}\label{cor:thoma}
Under the assumptions of Lemma \ref{lem:thoma-1} and if $\vN_\pi(\Sset_\infty)$ is a factor, then
\[
\trace\Big(\pi(\sigma) \prod_{i=0}^\infty f_i(A_i)\Big) 
= \prod_{V \in \Nset_0\slash \langle \sigma \rangle} 
   \int_{\spec{A_0}} t^{|V|-1} \prod_{j \in V} f_j(t) d\mu(t).   
\]  
\end{Corollary}
In particular we can choose $f_i$ to be the characteristic function $\chi_{B}$, where $B$ is some Borel set in $[-1,1]$. Since we know already that $A_0$ has discrete spectrum we are of course interested in $\chi_{\{t\}}(A_0)$ for $t \in \spec{A_0}$. 
\begin{proof}[Proof of Theorem \ref{thm:thoma}]
Fix $t \in \spec{A_0} \backslash\{0\}$ and let $\mu:= \mu(\{t\}) = \trace\big(\chi_{\{t\}}(A_0)\big)$.  Note that $\mu >0$. For $n \in \Nset$ let 
\[
f_i(s) = 
\begin{cases}
\chi_{\{t\}}(s) & \text{if $0 \le i < n$,}\\
1           & \text{if  $i \ge n$}
\end{cases}
\]
in Corollary \ref{cor:thoma}. Then for $\sigma \in \Sset_n$ Thus
\begin{eqnarray*}
\trace\Big(\pi(\sigma) \prod_{i=0}^{n-1} \chi_{\{t\}}(A_i)\Big) 
&=& \prod_{V \in \{0,1,\ldots,n-1\}\slash \langle \sigma \rangle} 
   \big(t^{|V|-1} \mu\big)  
= t^{n- \ell(\sigma)} \mu^{\ell(\sigma)}, 
\end{eqnarray*}
where $\ell(\sigma)$ is the number of orbits of $\sigma$ in  $\{0,1,\ldots,n-1\}$, in other words: $ \ell(\sigma)$ is the cardinality of
the set of orbits $\{0,1,\ldots,n-1\}\slash \langle \sigma \rangle$.  

We split the remaining discussion into the two cases of positive and negative spectral value $t$.

Let $t>0$ and consider the orthogonal projection 
\[
p_{-}^{(n)}:= \frac{1}{n!} \sum_{\sigma \in \Sset_n} \operatorname{sgn}(\sigma)\pi(\sigma). 
\]
(Since the operators $\big(f_i(A_i)\big)_{i=0}^{n-1}$ mutually commute we can think of $p_{-}^{(n)}$ to be the projection onto antisymmetric functions in $n$ variables. Compare also the tensor product model in Section \ref{section:Thoma}.) 
We deduce 
\[
\trace\Big(p_{-}^{(n)} \prod_{i=0}^{n-1} \chi_{\{t\}}(A_i)\Big) \ge 0
\]
from traciality. One the other hand one computes that
\begin{eqnarray*}
\trace\Big(p_{-}^{(n)} \prod_{i=0}^{n-1} \chi_{\{t\}}(A_i)\Big) 
&=& \frac{1}{n!} \sum_{\sigma \in \Sset_n} \operatorname{sgn}(\sigma) 
    \trace\Big(\pi(\sigma) \prod_{i=0}^{n-1} \chi_{\{t\}}(A_i)\Big) \\
&=& \frac{t^n}{n!} \sum_{\sigma \in \Sset_n} \operatorname{sgn}(\sigma)
    t^{- \ell(\sigma)} \mu^{\ell(\sigma)}\\
&=& \frac{t^n}{n!} \sum_{\sigma \in \Sset_n} \operatorname{sgn}(\sigma)
    \nu^{\ell(\sigma)},    
    \end{eqnarray*}
where we have put $\nu := \mu/|t|$. Note that $\nu>0$. Since $\operatorname{sgn}(\sigma)= (-1)^{n+\ell(\sigma)}$ and, by a well known combinatorial formula (compare \cite[Proposition 1.3.4]{Stan86a}),
\[
\sum_{\sigma \in \Sset_n} x^{\ell(\sigma)} = x(x+1) \cdots (x+n-1),
\]
we calculate further that 
\begin{eqnarray*}
\trace\Big(p_{-}^{(n)} \prod_{i=0}^{n-1} \chi_{\{t\}}(A_i)\Big) 
&=& \frac{(-t)^n}{n!} \sum_{\sigma \in \Sset_n} 
    (-\nu)^{\ell(\sigma)}\\
&=&  \frac{t^n}{n!} 
      \nu 
     \big( \nu-1\big)
     \cdots
     \big(\nu-n+1\big).       
\end{eqnarray*}
Because the original expression is positive and the formula above is true for all $n \in \Nset$, the product must terminate, in other words: $\nu \in \Nset$. 

We consider next the case $t< 0$. For this purpose let  
\[
p_{+}^{(n)}:= \frac{1}{n!} \sum_{\sigma \in \Sset_n} \pi(\sigma). 
\]
A similar computation as before yields
\begin{eqnarray*}
\trace\Big(p_{+}^{(n)} \prod_{i=0}^{n-1} \chi_{\{t\}}(A_i)\Big) 
&=& \frac{1}{n!} \sum_{\sigma \in \Sset_n}  
    \trace\Big(\pi(\sigma) \prod_{i=0}^{n-1} \chi_{\{t\}}(A_i)\Big) \\
&=& \frac{t^n}{n!} \sum_{\sigma \in \Sset_n} 
    t^{- \ell(\sigma)} \mu^{\ell(\sigma)}\\
&=& \frac{t^n}{n!} \sum_{\sigma \in \Sset_n} 
    (-\nu)^{\ell(\sigma)}\\    
&=& \frac{t^n}{n!} (-\nu) (-\nu+1) \cdots (-\nu+n-1)\\
&=& \frac{|t|^n}{n!} \nu (\nu-1) \cdots (\nu-n+1). 
    \end{eqnarray*}
Now the positivity of the original expression implies again that $\nu \in \Nset$. Consequently the spectral measure $\mu$ is a Thoma measure.
\end{proof}

\section{Completion of the proof of Thoma's theorem and the connection with Powers factors}
\label{section:Thoma}
Let us summarize what we have achieved so far with respect to the proof of the Thoma theorem, Theorem \ref{thm:ethoma}. We know from Proposition \ref{prop:extremal-char} that an extremal character of the group $\Sset_\infty$ gives rise to a unitary representation such that the weak closure is a factor $\cA$ with a finite faithful (normal) trace $\trace$. By Thoma multiplicativity from Theorem \ref{thm:thoma-mult}, with $\cA_{-1} = \cZ(\cA) = \Cset$, we obtain for $\sigma \in \Sset_\infty$
\begin{eqnarray*}
\trace\big( \pi(\sigma)\big) 
= \prod_{k =2}^{\infty}  \left(\trace
  \left( \big(E_0(v_1)\big)^{k-1}\right) \right)^{m_k(\sigma)}. 
\end{eqnarray*} 
If $\sigma$ is a $k$-cycle then
\[
\trace\big(\pi(\sigma)\big) 
= \trace\left( \big(E_0(v_1)\big)^{k-1}\right)
= \trace(A^{k-1}_0) = \int^1_{-1} t^{k-1}\,\mu(dt)
\]
where $\mu$ is the spectral measure of $A_0$. From Proposition \ref{prop:mu-discrete} we know that $\mu$ is discrete and from 
Theorem \ref{thm:thoma} that $\mu$ is a Thoma measure, i.e., 
$\frac{\mu(\{t\})}{|t|} \in \Nset_0$ for $t \neq 0$. Hence there exists
a function $\nu: [-1,1] \rightarrow \Nset_0$, nonzero for at most countably many points, such that
\[
\trace\big(\pi(\sigma)\big) = \sum_{t\in[-1,1]} t^{k-1} |t|\, \nu(t)
= \sum_{t>0} t^k\, \nu(t) + (-1)^{k-1} \sum_{t<0} |t|^k \,\nu(t)
\]
and from that we can read off the $a_i$ and $b_j$ (which may include repetitions) from the classical formulation of Thoma's theorem in Theorem \ref{thm:ethoma}.

To complete our proof of Thoma's theorem we have to show the existence of the representation, given the $a_i$ and $b_j$,
with the required properties. There are several ways to do this. 
The first explicit construction of the von Neumann factor corresponding to a Thoma trace has been given by Vershik and Kerov in \cite{VeKe81a}, from a groupoid point of view. They also mention the embedding into a Powers factor (in the case $a_1 + \ldots + a_M =1$) which is the basic idea of the construction given here. Further variations and modifications of such constructions can be found in \cite{Wass81a,Olsh89a,TsilVers07a}, for example. 

From the point of view of this paper the realization by embedding into an infinite tensor product is the most natural
because it is instructive to revisit our constructions in this model.  

Given the two sequences of positive parameters $(a_i)_{i=1}^M$ and $(b_j)_{j=1}^N$ with $0 \le M,N \le \infty$, as in Thoma's theorem, we put $c := 1 - \sum_i a_i - \sum_j b_j \ge 0$. By convention, $(a_i)_{i=1}^0$ and $(b_j)_{j=1}^0$ are understood to denote the empty set. We also include the possibilities $M=\infty$ or $N=\infty$ in which case some of the formulas below are an abuse of notation with an obvious interpretation. 
Let us form a Hilbert space
\[
\cH = \cH^+ \oplus \cH^- \oplus \cH^0
\]
where $\cH^+ = \{0\}$ if $M=0$ and has orthonormal basis $(\delta^+_i)_{i=1}^M$ otherwise, $\cH^- = \{0\}$ if $N=0$ and has orthonormal basis $(\delta^-_j)_{j=1}^M$ otherwise,
and $\cH^0 = \{0\}$ if $c=0$ and has countably infinite orthonormal basis $(\delta^0_k)_{k=1}^\infty$ otherwise. 
We write $\delta_r$ for an element of the orthonormal basis of $\cH$ formed by all $\delta^+_i, \delta^-_j, \delta^0_k$.
and from now on we identify $\cB(\cH)$ with the matrix algebra induced by this basis. 
We define a state $\psi_\ell$ on $\cB(\cH)$ by the diagonal density operator
\[
\diag(a_1,a_2,\ldots, b_1, b_2,\ldots, \underbrace{\frac{c}{\ell},\ldots,\frac{c}{\ell}}_{\ell \, \text{times}},
0,0,\ldots)
\]
and a state $\psi$ as any accumulation point of the sequence $\psi_\ell$ in the weak* topology on $\cB(H)^*$. Note that this is not a normal state if $\cH^0$ is nontrivial but this will not cause problems. 

Let us further represent a transposition on $\cH \otimes \cH$ by a unitary operator which is a (signed) flip

\[
U: \delta_r \otimes \delta_s \mapsto
\begin{cases}
-\delta_s \otimes \delta_r & \text{if }\; \delta_r, \delta_s \in \cH^- ,\\
\phantom{-}\delta_s \otimes \delta_r & \text{otherwise}.
\end{cases}
\]
Because $\Sset_\infty$ is generated by transpositions we obtain a unitary representation $\pi$ of $\Sset_\infty$ in the infinite 
(C$^*$)-tensor product $\bigotimes^\infty_0 \cB(\cH)$ by representing the Coxeter generators $\sigma_i$ by $\pi(\sigma_i) = u_i$ which by definition is the unitary $U$ from above acting at the tensor positions $i-1$ and $i$. 
It is not difficult to check that the restriction of the infinite product state $\bigotimes^\infty_0 \psi$ gives Thoma's formula on the represented $\Sset_\infty$. In fact, because the state is diagonal, this amounts to counting fixed points of the corresponding non-free action of 
$\Sset_\infty$ on $\{1,\ldots,M+N\}^{\Nset_0}$ with product measure
of $(a_1, \ldots, a_M, b_1, \ldots, b_N)$. The $c$-part does not give a contribution as can be shown by approximation.
If we now form the GNS-representation with respect to $\bigotimes^\infty_0 \psi$ and consider the von Neumann algebra $\cA$ generated by this representation of $\Sset_\infty$ then we obtain a tracial probability space $(\cA,\trace)$, where
$\trace = \bigotimes^\infty_0 \psi |_{\cA}$ is the Thoma trace. The fact that it is indeed a trace follows immediately from the fact that it is constant on conjugacy classes of $\Sset_\infty$ which are given by cycle types.

Let us revisit some of the constructions of this paper in this setting. First of all note that the
endomorphism $\alpha$ is nothing but the restriction of the tensor shift to $\cA$. To go further let us simplify notation by assuming for the moment that we have $a_1 + \ldots + a_M = 1$. Then the GNS-space for $\big( \cB(\cH),\psi \big)$ can be written as $\cH \otimes \cH$ with cyclic unit vector 
$\Omega = \sum^M_{i=1} \sqrt{a_i}\; \delta^+_i \otimes \delta^+_i$ and the GNS-space of 
$\big( \bigotimes^\infty_{n=0} \cB(\cH), \bigotimes^\infty_{n=0}\psi \big)$
is the infinite tensor product 
$\bigotimes^\infty_{n=0} (\cH \otimes \cH)$
of Hilbert spaces along the distinguished sequence of unit vectors $\Omega_n$ (that is, $\Omega$ at position $n$). The algebra $\cA$ acts on the left component. We can speak of localized elements if the expression differs from $\Omega$ (for vectors in the GNS-space) or from $\1$ (for operators) only at finitely many positions of the infinite tensor product. For example let $q_t$ be the orthogonal projection from $\cH$ onto the subspace spanned by all $\delta^+_i$ such that 
$a_i = t$. Then we have a localized projection 
\[
Q_t = \big( q_t \otimes \1 \big)_0 \otimes (\1 \otimes \1)_{[1,\infty)}
\in \cB\Big((\cH \otimes \cH) \otimes\bigotimes^\infty_{n=1} (\cH \otimes \cH)\Big)
\]
Let us identify it with an object already known. As usual let 
\[
v_n = \pi(\gamma_n) = \pi\big((0,n)\big)
\] 
which now is a flip of positions $0$ and $n$ in the tensor product.
Hence if $\xi, \eta$ are in the GNS-space $\bigotimes^\infty_{n=0} (\cH \otimes \cH)$ such that $Q_t\, \eta = \eta$ we can check, by approximating $\xi, \eta$ with localized vectors and then using the formula for $\Omega$, that
\[
\lim_{n\to\infty} \langle \xi,\,v_n\, \eta \rangle
\;=\; t\;\langle \xi,\, \eta \rangle
\]
In fact, if $n$ is big enough and if we only write down positions $0$ and $n$ (as the rest is not changed by $v_n$) then we have $\xi = \xi_0 \otimes \ldots \otimes \Omega_n \otimes \ldots$ and $\eta = \eta_0 \otimes \ldots \otimes \Omega_n \otimes \ldots$. Because $Q_t\, \eta = \eta$ we may write $\eta_0 = \delta^+_i \otimes \zeta$ such that
$a_i = t$ (or a linear combination of such terms). Then (for $n$ big enough)
\begin{eqnarray*}
& & \lim_{n\to\infty} \langle \xi,\,v_n\, \eta \rangle \\
&=& \langle \xi_0 \otimes \ldots \otimes (\sum^M_{j=1} \sqrt{a_j}\; \delta^+_j \otimes \delta^+_j)_n \otimes \ldots,
\,v_n\, (\delta^+_i \otimes \zeta)_0 \otimes \ldots \otimes (\sum^M_{k=1} \sqrt{a_k}\; \delta^+_k \otimes \delta^+_k)_n \otimes \ldots \rangle \\
&=& \langle \xi_0 \otimes \ldots \otimes (\sum^M_{j=1} \sqrt{a_j}\; \delta^+_j \otimes \delta^+_j)_n \otimes \ldots,
\,(\delta^+_k \otimes \zeta)_0 \otimes \ldots \otimes (\sum^M_{k=1} \sqrt{a_k}\; \delta^+_i \otimes \delta^+_k)_n \otimes \ldots \rangle \\
&=& a_i \;\langle \xi,\, \eta \rangle = t \;\langle \xi,\, \eta \rangle
\end{eqnarray*}

Similarly we can include $\cH^-$ and $\cH^0$ into the picture and find the same formula with $t= -\beta_i$ resp. $t=0$. For $t<0$ we have to take the minus sign in the representation of transpositions into account, for $t=0$ the result can be obtained by first considering the defining approximation of the singular state and then going to the limit.
Because, by (\ref{eq:average-2}), 
\[
A_0 = E_0(v_1) = \wotlim_{n\to\infty} v_n
\]
and because $\bigoplus_t Q_t = \1$ we conclude that $Q_t$ is exactly the (represented) spectral projection of the limit $2$-cycle $A_0$ with eigenvalue $t$. This means that we can identify the projection $Q_t$ with $\chi_{\{t\}}(A_0)$ in the notation of Section \ref{section:okounkov}. In particular $Q_t \in \cA$. 

Our proof of Thoma's theorem is finally completed by the following result.
\begin{Proposition}\label{prop:factor}
The von Neumann algebra $\cA$ constructed above is a factor.
\end{Proposition}
Let $E$ be the $\bigotimes^\infty_0 \psi$-preserving conditional expectation
from $\bigotimes^\infty_{n=0} \cB(\cH)$ onto the $0$-th component. $E$ is
obtained by evaluating the state at all tensor positions $n \ge 1$. In the
following we will use repeatedly the elementary fact that we obtain
factorizations $E(xy) = E(x)E(y)$ for such a tensor conditional expectation
$E$ whenever $x$ and $y$ are localized at tensor positions $\{0\} \cup I$ resp. $\{0\} \cup J$ where $I, J \subset \Nset$ and $I \cap J = \emptyset$.
\begin{Lemma}\label{lem:factor}
The conditional expectation $E$ maps $\cA$ onto $\vN(A_0) \subset \cA$.
\end{Lemma}
Note that $\vN(A_0)$ is generated by the spectral projections
$\chi_{\{t\}}(A_0)$ with $t \in \spec(A_0)$ and hence, by the discussion
above which identified these spectral projections in our model, $\vN(A_0)$ is also a subalgebra of the diagonal of the $0$-th tensor position.
\begin{proof}
Because $v_i$ for $i \ge 1$ is realized as a (signed) tensor flip we can
verify by a direct computation similar to the one in Example 5.5 that
$E(v_i) = A_0$.  If $\sigma$ is any permutation with disjoint cycle decomposition $\sigma = s_1 \ldots s_m$ then
\[
E( \pi(\sigma) ) = E( \pi(s_1) ) \cdots  E( \pi(s_m) ).
\]
For cycles $s$ with $s(0)=0$ the factor $E( \pi(s) )$ is a multiple of the
identity. If for one of the cycles we have $s(0) \not= 0$ then we can write
\[
\pi(s) = v_{n_2} \cdots v_{n_k}
\]
with distinct $n_2, \ldots, n_k$. Then
\[
E( \pi(s) ) = E( v_{n_2} ) \cdots E( v_{n_k} ) = A^{k-1}_0 \in \vN(A_0).
\]
Because the $\pi(\sigma)$ with $\sigma \in \Sset_\infty$ generate $\cA$ we
are done.
\end{proof}
\begin{proof}[Proof of Proposition \ref{prop:factor}]
The sequence $(v_i)_{i \in \Nset}$ is exchangeable and, because 
$\cA = \vN \set{v_i}{i \in \Nset}$, it is minimal. We observe that it is
also order $\cN$-factorizable with $\cN = \vN(A_0)$. In fact, if
$x \in \cA_I = \vN\set{v_i}{i \in I}$ and $y \in \cA_J = \vN\set{v_i}{i
\in J}$ with $I < J$ then, with Lemma \ref{lem:factor},
\[
E_\cN (xy) = E (xy) = E(x) E(y) = E_\cN (x) E_\cN (y).
\]
Further note that $\cN \subset \cA_0$ because by definition $\cA_0$ is the
relative commutant in $\cA$ of the $u_i$'s with $i \ge 2$ (see Lemma \ref{lem:inclusions}) and these $u_i$ trivially commute with operators which are localized at the $0$-th tensor position. (Alternatively this inclusion also follows from Theorem \ref{thm:fixed-points}.) But $\cA_0$ is also the tail algebra of the sequence $(v_i)_{i \in \Nset}$, see Theorem \ref{thm:indy}\eqref{item:indy-ii}. So we are in a position to apply the fixed point characterization from Theorem \ref{thm:fixed-point} \eqref{item:fixed-point-iii} and find that $\cN = \cA_0$.

We have identified $\cA_0$ as a subalgebra of the diagonal of the $0$-th
tensor position. If $z \in \cZ(\cA) \, (= \cA_{-1} \subset \cA_0$) then in
addition $z$ also commutes with $u_1$ which is a (signed) flip of the tensor
positions $0$ and $1$. Hence $z$ is a multiple of the identity. It follows
that $\cZ(\cA) = \Cset$ and $\cA$ is a factor.
\end{proof}
\begin{Remark}\normalfont
In retrospect it turns out that the tensor conditional expectation $E$
restricted to $\cA$ is nothing but the conditional expectation $E_0$ and
that some of the formulas obtained above coincide with those in Proposition
4.4 about $E_0$. Hence in the factor case the tensor realization provides a
concrete model for those.
\end{Remark}
\begin{Remark} \normalfont
We also see that (in this factorial case) $\cA_0$ is spanned by the $Q_t$ and that it is localized at the tensor position $0$. The commutative Bernoulli shift which maps $A_i$ to $A_{i+1}$ for all $i \in \Nset_0$ (compare (\ref{eq:average-3})) is nothing but the restriction of the tensor shift to certain diagonal operators. Note however that for a projection $q$ strictly between $0$ and $q_t$ the corresponding projection $Q = \big( q \otimes \1 \big)_0 \otimes (\1 \otimes \1)_{[1,\infty)}$
does not belong to $\cA$. In fact, suppose for example that 
$a_1 = a_2 = t$ and let $q^{(1)}$ resp. $q^{(2)}$
project onto $\Cset \delta^+_1$ resp. $\Cset \delta^+_2$. Let 
$\kappa$ be a state-preserving automorphism of $\cB(\cH)$ which interchanges $q^{(1)}$ and $q^{(2)}$. Then the infinite tensor product $\bigotimes^\infty_0 \kappa$ fixes the represented $\Sset_\infty$ and hence $\cA$ pointwise but it interchanges the corresponding projections $Q^{(1)}$ and $Q^{(2)}$.
\end{Remark}
\begin{Remark} \normalfont
It is also natural to think of $\cA$ as the fixed point algebra of an infinite tensor product group action (in a generalized Schur-Weyl duality). For example the argument in the previous remark goes along these lines. To use this theory, see \cite{Pric82a,BaPo83a,BaPo83b,BaPo86a}, to prove factoriality (Proposition \ref{prop:factor}) seems to be rather difficult technically. Our argument which is ultimately based on the probabilistic de Finetti-type arguments from \cite{Koes10a} seems to give it more directly. 
\end{Remark}
\section{Irreducible subfactors and Markov traces}
\label{section:Markov}
Suppose $\cZ(\cA) = \cA_{-1} = \Cset$, so $\cA$ is a factor and the
trace is given by Thoma's formula with parameters $a_i$ and $b_j$.
Let us study the case $\cA_0 = \Cset$ in more detail. From Corollary \ref{cor:fixed-points} this is equivalent to the fact that $\cA \cap \alpha(\cA)'=\Cset$, i.e., we have an irreducible subfactor
\[
\alpha(\cA) = \{u_2,u_3,\ldots \}'' \quad \subset \quad
\cA = \{u_1,u_2,u_3,\ldots \}''.
\]
\begin{Proposition}\label{prop:irreducible}
We have $\cA_0 = \Cset$ if and only if
\begin{eqnarray*}
&&1) \;\; a_i = b_j = 0 \quad \mbox{for all}\;\; i,j \\
& \mbox{or} & 2a) \;\; a_1 = \ldots = a_M = \frac{1}{M} \quad
\mbox{for some}\; M \in \Nset \\
& \mbox{or} & 2b) \;\; b_1 = \ldots = b_N = \frac{1}{N} \quad
\mbox{for some}\; N \in \Nset \\
\end{eqnarray*}
\end{Proposition}
\begin{proof}
From Theorem \ref{thm:fixed-points} we know that $\cA_0 = \Cset$ means that $\cA_0$ is spanned by $E_0(u_1) = \trace(u_1) \1$. Hence with $t = \trace(u_1)$ we have the following cases:

$1)\; t=0$. In this case we find the regular representation of $\Sset_\infty$, i.e., for all elements $g \in \Sset_\infty$ except the identity we have $\trace(g) = 0$. This situation can also be described by $a_i = b_j = 0$ for all $i,j$.

$2)\; t \not=0 $. Using Thoma multiplicativity (Theorem \ref{thm:thoma-mult}) on the one hand and Thoma's formula (Theorem \ref{thm:ethoma})
on the other hand we obtain the value of $\trace(g)$ for a $k$-cycle
$g$ (for all $k \in \Nset$) as
\[
t^{k-1} = \left(\sum_{i=1}^\infty a_i^k + (-1)^{k-1} \sum_{j=1}^\infty b_j^k\right)
\]
We can think of these quantities as $(k-1)$th moments of finite positive Borel measures in the interval $[-1,1]$. Because such measures are uniquely determined by their moments we conclude that
\[
\delta_t = \sum_i  a_i\, \delta_{a_i} + \sum_j b_j\, \delta_{-b_j} 
\]
where the $\delta_x$ are Dirac measures.
If $t>0$ this implies that for all $i,j$
\[
b_j=0, \quad a_i=t,\quad \sum_i a_i = 1
\]
and hence there exists some $M \in \Nset$ such that
\[
a_1 = \ldots = a_M = \frac{1}{M}\,.
\]
Similarly if $t<0$ then we obtain for all $i,j$
\[
a_i=0, \quad -b_j=t,\quad \sum_j b_j = 1
\]
and hence there exists some $N \in \Nset$ such that
\[
b_1 = \ldots = b_N = \frac{1}{N}\,.
\]
\end{proof}
\begin{Remark} \normalfont
Note that for $\trace(u_1)=\pm 1$ we get the one dimensional identical and sign representations but for
$|t|<1$ we have an irreducible inclusion of hyperfinite II$_1$-factors. The case $t=0$ is the group factor of the ICC-group $\Sset_\infty$ in the usual sense, see for example 
\cite{JoSu97a}, while cases $(2a)$ and $(2b)$ can also be characterized by the fact that in these cases the embedding constructed in Section \ref{section:Thoma} actually embeds into the hyperfinite II$_1$-factor constructed by an infinite tensor product of $M\times M$- resp. $N \times N$-matrices and weak closure with respect to the trace.
\end{Remark}
\begin{Definition}\label{def:markov-trace}
A tracial functional $\trace$ on the group algebra $\Cset \Sset_\infty$ is called a Markov trace if for all $n \in \Nset$ and $g\in \Sset_n$
\[
\trace(g\, \sigma_n) = \trace(g)\,\trace(\sigma_n)\,.
\]
\end{Definition}
\begin{Proposition}\label{prop:Markov}
The trace on $\cA$ is a Markov trace on $\Cset \Sset_\infty$ if and only if $\cA_0 = \Cset$.
\end{Proposition}
\begin{proof}
This has already been shown in Theorem \ref{thm:cs-coxeter} \eqref{item:cs-coxeter-vi}.
\end{proof}
\begin{Remark} \normalfont
Markov traces can be defined more generally for Hecke algebras and our definition is a special case of that. 
Hence the result above reproduces a part of Ocneanu's classification described in \cite{Wenz88a,Jone91a}.
But all the treatments we have seen emphasize the $q$-deformed situations and we have not found a straightforward argument for $\Sset_\infty$ such as the one above. Our parametrization by $\trace(u_1) = t = \pm \frac{1}{N}$
with $N\in \Nset$ resp. $t=0$ corresponds in \cite{Jone91a} to the case $\tau = \frac{1}{4}$ with $\eta = \frac{t+1}{2} = \frac{\pm 1 + N}{2N}$ resp. 
$\eta = \frac{1}{2}$. 
\end{Remark}
\begin{Remark}\label{rem:irreducible-tower} \normalfont 
With $\cA_0 = \Cset$ we can combine Theorem \ref{thm:endo-braid-i}
and Theorem \ref{thm:fixed-points} to obtain the tower of commuting squares
\begin{eqnarray*}
\setcounter{MaxMatrixCols}{20}
\begin{matrix}
\Cset &\subset&   \Cset & \subset & \Cset S_1 & \subset & \Cset S_2 & \subset & \Cset S_3 & \subset  & \cdots & \subset & \cA \\
        &&          \cup  &         & \cup  &         & \cup &         & \cup  &       & & & \cup  \\
        &&   
\Cset &\subset&\Cset&\subset&\alpha(\Cset S_1)&\subset&\alpha(\Cset S_2)&\subset& \cdots & \subset & \alpha(\cA)\\
\end{matrix}
\setcounter{MaxMatrixCols}{10}
\end{eqnarray*}
We conclude that the index of 
the irreducible subfactor $\alpha(\cA) \subset \cA$
can be computed as the limit $n\to\infty$ of the 
Pimsner-Popa index \cite{PiPo86a} of the finite dimensional inclusions $\big(\alpha(\Cset \Sset_{n-1}) \subset \Cset \Sset_n \big)
\simeq \big(\Cset \Sset_{n-1} \subset \Cset \Sset_n \big)$
(with respect to the given Markov trace $\trace(u_1)=\pm \frac{1}{N})$, see \cite[Proposition 5.1.9]{JoSu97a}.
It would be interesting to find an efficient way to finish the computation of the index within the toolkit of this paper. The result is known to be $N^2$ as can be deduced from Wassermann's invariance principle for subfactors \cite{Wass88a}. In the case of the regular representation $\big(\trace(u_1)=0\big)$ the index is infinite.
We would like to thank Hans Wenzl for orienting us in the relevant 
literature.
\end{Remark}
\begin{Remark} \normalfont
After deriving the results about irreducible subfactors in this section with our methods we learned that M.~Yamashita in the recent preprint \cite{Yama09a} obtained some of these results by combinatorial methods, in particular Proposition \ref{prop:irreducible} and the commuting squares in Remark \ref{rem:irreducible-tower}. 
\end{Remark} 
\bibliographystyle{alpha}                 
\label{section:bibliography}

\end{document}